\documentclass{amsart}
\usepackage[utf8]{inputenc}
\usepackage{amsfonts}
\usepackage{amsthm}
\usepackage{amssymb}
\usepackage{amsmath}
\usepackage{eucal}
\usepackage{mathrsfs}
\usepackage{wasysym}
\usepackage{comment}
\usepackage[centering,bottom=3cm,margin=2.7cm,top=3cm]{geometry}
\usepackage{tikz-cd}
\usetikzlibrary{matrix,arrows,decorations.pathmorphing}
\usepackage{biblatex}
\usepackage{xcolor}
\usepackage{hyperref}
\usepackage{cleveref}

\newcommand\myshade{85}
\colorlet{mylinkcolor}{violet}
\colorlet{mycitecolor}{blue}
\colorlet{myurlcolor}{green}

\hypersetup{
  linkcolor  = mylinkcolor!\myshade!black,
  citecolor  = mycitecolor!\myshade!black,
  urlcolor   = myurlcolor!\myshade!black,
  colorlinks = true,
}

\renewcommand{\sec}[0]{\section}
\newcommand{\ssec}[0]{\subsection}
\newcommand{\sssec}[0]{\subsubsection}

\newcommand{\on}{\operatorname}
\newcommand{\nc}{\newcommand}
\nc{\TV}{To\"en--Vezzosi}
\nc{\et}{{\on{\acute{e}t}}}
\nc{\virg}[1]{``#1"}
\nc{\bigt}[1]{\big( #1 \big) }
\nc{\Bigt}[1]{\Big( #1 \Big) }
\nc{\wt}{\tilde}

\theoremstyle{plain}
\newtheorem{thm}[subsubsection]{Theorem}
\newtheorem{mainthm}{Theorem}

\newtheorem{lem}[subsubsection]{Lemma}
\newtheorem{prop}[subsubsection]{Proposition}
\newtheorem{cor}[subsubsection]{Corollary}

\theoremstyle{definition}

\newtheorem{defn}[subsubsection]{Definition}
\newtheorem{notat}[subsubsection]{Notation}

\theoremstyle{remark}
\newtheorem{rmk}[subsubsection]{Remark}

\newtheorem{rem}[subsubsection]{Remark}

\theoremstyle{plain}

\newtheorem*{uncor}{Corollary}
\newtheorem*{unconj}{Conjecture}
\theoremstyle{definition}

\newcommand{\ccC}[0]{\mathcal{C}}

\newcommand{\ccF}[0]{\mathcal{F}}

\newcommand{\ccH}[0]{\mathcal{H}}

\newcommand{\ccM}[0]{\mathcal{M}}

\newcommand{\ccO}[0]{\mathcal{O}}

\newcommand{\ccS}[0]{\mathcal{S}}

\newcommand{\bbA}[0]{\mathbb{A}}

\newcommand{\bbD}[0]{\mathbb{D}}
\newcommand{\bbE}[0]{\mathbb{E}}

\newcommand{\bbL}[0]{\mathbb{L}}

\newcommand{\bbP}[0]{\mathbb{P}}
\newcommand{\bbQ}[0]{\mathbb{Q}}

\newcommand{\bbU}[0]{\mathbb{U}}

\newcommand{\bbZ}[0]{\mathbb{Z}}

\newcommand{\ffF}[0]{\mathfrak{F}}
\newcommand{\ffG}[0]{\mathfrak{G}}

\newcommand{\rmB}[0]{\mathrm{B}}

\newcommand{\scrG}[0]{\mathscr{G}}

\newcommand{\scrO}[0]{\mathscr{O}}

\newcommand{\scrV}[0]{\mathscr{V}}

\nc{\sA}{{\mathsf{A}}}
\nc{\sB}{{\mathsf{B}}}
\nc{\sC}{{\mathsf{C}}}
\nc{\sD}{{\mathsf{D}}}
\nc{\sE}{{\mathsf{E}}}
\nc{\sF}{{\mathsf{F}}}
\nc{\sG}{{\mathsf{G}}}
\nc{\sK}{{\mathsf{K}}}
\nc{\sM}{{\mathsf{M}}}
\nc{\sN}{{\mathsf{N}}}
\nc{\sO}{{\mathsf{O}}}
\nc{\sW}{{\mathsf{W}}}
\nc{\sQ}{{\mathsf{Q}}}
\nc{\sP}{{\mathsf{P}}}
\nc{\sR}{{\mathsf{R}}}
\nc{\sS}{{\mathsf{S}}}
\nc{\sT}{{\mathsf{T}}}
\nc{\sU}{{\mathsf{U}}}
\nc{\sV}{{\mathsf{V}}}
\nc{\sZ}{{\mathsf{Z}}}

\newcommand{\up}[1]{\on{#1}}
\newcommand{\ul}[1]{\underline{#1}}

\renewcommand{\epsilon}{\varepsilon}
\newcommand{\eps}{\epsilon}
\newcommand{\op}[0]{{\on{op}}}
\nc{\cat}{\on{cat}} 
\newcommand{\sg}[0]{{\on{sg}}}

\newcommand{\oo}[0]{\infty}

\newcommand{\uF}[0]{\on{F}}
\newcommand{\uG}[0]{\on{G}}
\newcommand{\uH}[0]{\on H}
\newcommand{\uI}[0]{\on{I}}
\newcommand{\uK}[0]{\on{K}}
\newcommand{\uL}[0]{\on{L}}
\newcommand{\GLK}[0]{\on{G}_{\on{L/K}}}
\newcommand{\PLK}[0]{\on{P}_{\on{L/K}}}
\newcommand{\uKs}[0]{\overline{\on{K}}}
\newcommand{\uP}[0]{\on{P}}
\newcommand{\Tr}[0]{\on{Tr}}
\renewcommand{\lg}{\on{length}}
\newcommand{\coFib}[0]{\on{coFib}}

\DeclareMathOperator{\coker}{coker} 
\nc{\act}{{\on{act}}}
\nc{\rev}{{\on{rev}}}
\nc{\env}{{\on{env}}}

\renewcommand{\rightarrow}{\to}
\newcommand{\xto}[1]{\xrightarrow{#1}}

\newcommand{\hto}{\hookrightarrow}
\DeclareMathOperator{\ev}{ev} 
\DeclareMathOperator{\coev}{coev} 
\DeclareMathOperator{\id}{id} 
\newcommand{\longto}{\longrightarrow}

\newcommand{\Gal}[0]{{\scrG}al}
\newcommand{\IK}{{\uI_{\uK}}}
\newcommand{\IL}{{\uI_{\uL}}}

\newcommand{\Ql}[1]{\bbQ_{\ell #1}}
\newcommand{\BU}[0]{\rmB \bbU}
\newcommand{\End}[0]{\up{End}}
\newcommand{\Qell}{\bbQ_{\ell}}
\nc{\sBenv}{\sB^{\env}}
\nc{\sCenv}{\sC^{\env}}
\newcommand{\QellI}{\bbQ_\ell^{\uI}}
\newcommand{\QellIG}{\bbQ_\ell^{\uI}[\GLK]}
\newcommand{\QellIGbeta}{\bbQ_\ell^{\uI}(\beta)[\GLK]}
\newcommand{\QellIGred}{\bbQ_\ell^{\uI}(\GLK)}
\newcommand{\QellIGredbeta}{\bbQ_\ell^{\uI}(\beta)(\GLK)}
\nc{\Benv}{\CB^{\env}}
\nc{\Sym}{\on{Sym}}

\newcommand{\oml}[1]{\omega_{\ell #1}}

\newcommand{\Spec}[1]{\on{Spec}(#1)}
\newcommand{\dSch}[0]{\on{dSch}}
\newcommand{\Sch}[0]{\on{Sch}}
\nc{\bareta}{\bar{\eta}}

\DeclareMathOperator{\Hom}{Hom} 
\newcommand{\pr}[1]{\on{pr}_{#1}}

\newcommand{\x}[1]{\times_{#1}}
\nc{\Ind}[0]{\on{Ind}}

\newcommand{\Ext}[0]{\on{Ext}}

\newcommand{\HH}[0]{\on{HH}}
\newcommand{\CH}[0]{\on{CH}}
\newcommand{\HK}[0]{\on{HK}}

\newcommand{\dgcat}[0]{\on{dgCat}}
\newcommand{\dgCat}{\on{dgCat}}
\newcommand{\dgCAT}[0]{\on{dgCAT}}
\newcommand{\mmod}{\on{-mod}}
\newcommand{\Mod}[0]{\on{Mod}}
\newcommand{\Fun}[0]{\on{Fun}}
\newcommand{\biMod}[0]{\on{biMod}}
\newcommand{\dgMod}[0]{\on{Mod}_{\on{dg}}}
\newcommand{\cofdgMod}[0]{\on{Mod}_{\on{dg}}^{\on{cof}}}

\DeclareMathOperator{\Coh}{\mathsf{D^b_{coh}}} 
\DeclareMathOperator{\Perf}{\mathsf{D_{perf}}} 
\DeclareMathOperator{\QCoh}{\mathsf{D_{qcoh}}} 
\DeclareMathOperator{\Sing}{\mathsf{D_{sg}}} 
\DeclareMathOperator{\MF}{\on{MF}}

\newcommand{\SH}[0]{\ccS \ccH}

\newcommand{\Shv}[0]{\on{Shv}_{\Ql{}}}
\newcommand{\Shvcon}[0]{\on{Shv}_{c,\Ql{}}}

\newcommand{\Mv}[0]{\ccM^{\vee}}
\newcommand{\rl}[0]{\on{r}^{\ell}}
\newcommand{\chern}[0]{\ccC\!\on{h}^{\ell}}

\newcommand{\Sw}[0]{\on{Sw}}
\newcommand{\dimtot}[0]{\on{dimtot}}
\newcommand{\sw}[0]{\on{sw}_{\uL\!/\!\uK}}
\newcommand{\ar}[0]{\on{ar}_{\uL\!/\!\uK}}
\newcommand{\nuIquot}{\nu^{\IL}\big/\nu^{\IK}}
\nc{\arcat}{\on{ar}_{\uL\!/\!\uK}^{\on{cat}}}
\nc{\idcat}{\id^{\cat}}
\nc{\Bl}{\on{Bl}} 
\nc{\Ar}{\on{Ar}} 

\newcommand{\usotimes}{\otimes}

\newcommand{\upast}{
\begin{tikzpicture}[baseline=-1.1mm, scale =0.12]
\node (A) at (0,0) {$\ast$};
\end{tikzpicture}
}

\newcommand{\boxast}[1]{
\: \rlap{\hskip -0.65mm $\upast$} \square_{#1} \:
}

\nc{\hB}{\on{h}_{\sB}}
\nc{\hC}{\on{h}_{\sC}}

\nc{\restr}[2]{\left. #1 \right |_{#2}}

\nc{\wideprime}[1]{#1'}

\renewcommand{\setminus}{\smallsetminus}
\renewcommand{\sim}{\simeq}

\title{Proof of the Deligne--Milnor conjecture}
\author{Dario Beraldo \and Massimo Pippi}
\date{\today}

\address[D.~Beraldo]{Department of Mathematics, University College London, London WC1H 0AY, United Kingdom}
\email{d.beraldo@ucl.ac.uk}

\address[M.~Pippi]{Univ Angers, CNRS-UMR 6093, LAREMA, SFR MATHSTIC, F-49000 Angers,France}
\email{massimo.pippi@univ-angers.fr}

\begin{document}

\begin{abstract}
We apply methods of derived and non-commutative algebraic geometry to understand ramification phenomena on arithmetic schemes.
As an application, we prove the Deligne--Milnor conjecture and, in the pure characteristic case, a generalization of Bloch conductor conjecture.
\end{abstract}

\maketitle

\tableofcontents

\section{Introduction}\label{sec: intro}

This paper is a contribution to the study of \emph{ramification theory on arithmetic schemes} by means of \emph{derived} and \emph{non-commutative} algebraic geometry.
In particular, we give a non-commutative interpretation of the \emph{total dimension} of the sheaf of vanishing cycles of an arithmetic scheme.

In the companion paper \cite{beraldopippi22}, we give a non-commutative interpretation of intersection-theoretic invariants of arithmetic schemes.
As we explain below, the combination of these two points of view has concrete applications.

\ssec{Overview}

The primary goal of this paper is to prove the Deligne--Milnor formula, conjectured by Deligne in \cite[Expos\'e XVI]{sga7ii}, in its full generality.

\sssec{}

For $S$ the spectrum of a DVR (discrete valuation ring), the conjecture calls for an equality of numbers associated to an $S$-scheme $U$ with an isolated singularity.
These numbers reflect the algebraic, topological and arithmetic information encoded by $U \to S$.

\sssec{}

In \cite[Expos\'e XVI]{sga7ii}, Deligne proved the Deligne--Milnor formula in the pure-characteristic case, hence the interest of the present paper lies mainly in the mixed characteristic situation.
Nevertheless, our proof also works in pure characteristic without modifications: 
this fact is, in our opinion, an important feature of our methods. 
Such methods, which are totally different from those of Deligne, rely heavily on ideas of \TV{} and are \virg{non-commutative} in nature:
even though the Deligne--Milnor formula is an equality of numbers, the main players of this paper are certain categories of sheaves, and the formula will then be obtained from them by performing a two-step decategorification (roughly speaking: from categories to vector spaces, and from vector spaces to numbers).

\sssec{} 

This strategy gets us somewhat close to the full Bloch conductor conjecture (see \cite{bloch87}), which is known to imply the Deligne--Milnor conjecture (\cite{orgogozo03}). This is currently work in progress and will be the topic of another work. 
However, what we do get in this paper is a proof of \emph{a generalization of Bloch conductor conjecture} in the case of pure characteristic.

\sssec{}

To summarize, in the present paper we prove the following two main results:

\begin{enumerate}

\item
A generalization of Bloch conductor conjecture in the pure characteristic case (Theorem \ref{mainthm: pure char}). 

\item
The Deligne--Milnor conjecture in full generality.
In fact, we manage to prove Bloch conductor conjecture in several new cases (Theorem \ref{mainthm: hypersurf}).

\end{enumerate}

Both results depend (logically for the former, and only morally for the latter) on a categorical version of Bloch conductor conjecture, which we state as Theorem \ref{thm: categorical Bloch}.
Even though this categorical version appears quite abstract, we believe it provides a truly conceptual explanation as to why the conductor formula conjectured by Bloch holds.
This categorical point of view is due to \TV, who stated and proved a categorical version of Bloch's formula in a particular case\footnote{They assume that the action of the inertia group on the $\ell$-adic cohomology of the geometric generic fiber is unipotent.}.

\sssec{}

Here is an even shorter summary of the contents of this paper. 
If we perceive the proof of Bloch conductor conjecture as a two-step decategorification, we have full control of the first decategorification and the second one yields the categorical generalized Bloch conductor formula.
This formula, however, is not a-priori an equality of numbers, but an equality of classes in a cohomology group we have not sufficient control of.
In the pure characteristic and hypersurface cases this issue can be bypassed.
This is how we obtain our results (1) and (2) above.

\ssec{Some basic notation}

\sssec{}

Let $v_{\up K}: \up K\to \bbZ \cup \{\oo\}$ be a complete discrete valuation field with ring of integers $\scrO_{\up K}$. 
We fix once and for all a uniformizer $\pi_{\up K}\in \scrO_{\up K}$ and denote by $k=\scrO_{\up K}/\langle \pi_{\up K} \rangle$ the residue field, which we assume algebraically closed of characteristic $p \geq 0$. 
Let $\uKs$ be a separable closure of $\uK$ and $\uI_{\uK}=\Gal(\uKs/\uK)$ the absolute Galois group of $\uK$. 
Recall the exact sequence of profinite groups
$$
1\to \uP_{\uK} \to \uI_{\uK} \to \uI_{\uK}^{\up t} \to 1.
$$
The profinite group $\uP_{\uK}$ is a pro-$p$-group (it is trivial if $p=0$) and is called the \emph{wild inertia group}, while the quotient $\uI_{\uK}^{\up t}$ is called the \emph{tame inertia group}.

\sssec{}

We will use the following notation:
$$
 s=\Spec{k} \xrightarrow{i_S} S=\Spec{\scrO_{\uK}} \xleftarrow{j_S}\eta=\Spec{\uK} \leftarrow \bar{\eta}=\Spec{\uKs}.
$$
For an $S$-scheme $Y$, we add subscripts to denote the relevant fiber products: $Y_s := s\x{S} Y$, and similarly for $Y_{\eta}$ and $Y_{\bar{\eta}}$. 
As usual, we call $Y_s$ the special fiber, $Y_\eta$ the generic fiber and $Y_{\bar{\eta}}$ the geometric generic fiber.

\sssec{}

Let $\uK \subseteq \uL (\subseteq \uKs)$ be a finite Galois extension giving rise to a totally ramified extension $\scrO_{\uK}\subseteq \scrO_{\uL}$. 
By \cite[Chapter 1, Proposition 18]{serre79}, we know that
$$
  \scrO_{\uL}\simeq \scrO_{\uK}[x]/ \langle E(x) \rangle
$$
for a suitable Eisenstein polynomial $E(x) \in \scrO_{\uK}[x]$. In this case, the element $\pi_{\uL}:= x \in \scrO_{\uL}$ is a uniformizer and $\pi_{\uK}$, when viewed as an element of $\scrO_{\uL}$, can be written as $\pi_{\uK}=u\cdot \pi_{\uL}^{e}$ for some $u \in \scrO_{\uL}^*$ and $e=\deg(E(x))$.
We set $S' := \Spec{\scrO_{\uL}} $ and often use the notation $Y' := Y \times_S S'$.
To avoid cumbersome notation, we set $A=\scrO_{\uK}$ and $A'=\scrO_{\uL}$.

\sssec{}

Let $\uI_{\uL}$ be the Galois group $\Gal(\uKs/\uL)$. This is an open normal subgroup of $\uI_{\uK}$, with finite quotient $\GLK:= \uI_{\uK}\big/\uI_{\uL}$.
We denote $\uP_{\uL}$ and $\uI_{\uL}^{\up{t}}$ the wild inertia subgroup of $\uI_{\uL}$ and its tame quotient, respectively. 
We will denote the quotient $\uP_{\uK}\big/\uP_{\uL}$ by $\PLK$: this is the $p$-Sylow subgroup of $\GLK$.

\ssec{The Deligne--Milnor conjecture}

\sssec{}

Let $q:U \to S$ be a flat regular $S$-scheme purely of relative dimension $n$, which we assume to be smooth outside an isolated singularity $x \in U_s$.

\sssec{}

The \emph{Deligne--Milnor number} is defined as
$$
  \mu_{U/S}= \lg_{\ccO_{U,x}} \Bigl ( \ul{\Ext}_U^1(\Omega^1_{U/S},\ccO_{U})_x \Bigr ). 
$$

\sssec{}

We fix once and for all a prime number $\ell$, different from the residue characteristic of $S$.
Recall that the $\oo$-functor of \emph{vanishing cycles}
$$
  \Phi_q: \Shvcon(U) \rightarrow \Shvcon^{\uI_{\uK}}(U_s)
$$
associates a (constructible) continuous $\ell$-adic $\uI_{\uK}$-representation to any (constructible) $\ell$-adic sheaf on $U$.
See \cite{sga7i,sga7ii, deligne77}.

\sssec{}

With our standing hypothesis, $\Phi_q(\Ql{,U})$ can be identified with a skyscraper sheaf at $x\in U_s$ placed in cohomological degree $n$. 
From now on we will adopt the following notation:
$$
  \Phi
  :=
  \Phi_q(\Ql{,U}).
$$
Since $\Phi$ is a continuous $\ell$-adic $\uI_{\uK}$-representation, the wild inertia subgroup acts through a finite quotient.
Therefore, there exists a finite Galois extension $\uK \subseteq \uL$ such that $\uP_{\uL}$ acts trivially on $\Phi$.
In particular, we get an induced action of $\PLK$ on $\Phi$. 

\sssec{}

Recall that the \emph{Swan character} is defined as
$$
  \sw: \GLK \to \bbZ
$$
$$
  g \mapsto 
  \begin{cases}
  \lg_{A}(\Omega^1_{A'/A}) - |\GLK| + 1, & \text{ if } g=\id;\\
  -\lg_{A}\bigt{A'/\langle \frac{g(\pi_{\uL})}{\pi_{\uL}}-1\rangle}, & \text{ if } g\neq \id.
  \end{cases}
$$
According to \cite{swan63}, this function is the character of an $\ell$-adic representation, known as the \emph{Swan representation}. It can take nonzero values only on the subgroup $\PLK$:
$$
  \sw(g)=0, \hspace{0.4cm} {\text{for } } g \in \GLK \setminus \PLK.
$$

\sssec{}
The \emph{wild dimension} of $\Phi$, also known as the \emph{Swan conductor}, is defined as
\begin{align*}
  \Sw(\Phi) & := \frac{1}{|\GLK|} \sum_{g\in \GLK}\sw(g)\cdot \Tr_{\Qell{}}(g:\Phi)\\
                            & = \frac{1}{|\GLK|} \sum_{g\in \PLK} \sw(g) \cdot \Tr_{\Qell{}}(g:\Phi).
\end{align*}
This is a well-defined integer that vanishes if and only if the action of $\uI_{\uK}$ on $\Phi_q$ is tame. 

\begin{rmk}\label{rmk: Sw always well defined}
The group $\GLK$ \emph{does not} act on $\Phi$ in general and yet the definition of the Swan conductor is well-posed: the trace $\Tr_{\Qell}(g:\Phi)$ is not defined for $g\notin \PLK$, but $\sw(g)=0$ in this case and so the meaning of the notation $\sw(g)\cdot \Tr_{\Qell}(g:\Phi)$ is just $\virg{0}$.
\end{rmk}

\sssec{}

Denote by $\chi (\Phi)$ the $\Qell$-adic Euler characteristic of $\Phi$ and define the \emph{total dimension} of $\Phi$ as
$$
  \dimtot (\Phi) := 
  \chi(\Phi) + \Sw (\Phi).
$$
The following formula was proposed by Deligne.
\begin{unconj}[Deligne--Milnor formula, \cite{sga7ii}]
With the above notation, we have:
$$
  \mu_{U/S}= (-1)^n \dimtot (\Phi).
$$
\end{unconj}

\begin{rmk}
This conjecture has been proved in the following cases:
\begin{itemize}
    \item when the relative dimension $n$ is $0$ in \cite[Exposé XVI]{sga7ii}. In fact, in this case the formula is equivalent to the \emph{conductor discriminant formula} of E.~Artin (\cite{artin31});
    \item when $x\in U_s$ is an ordinary quadratic singularity in \cite[Exposé XVI]{sga7ii}. This was proved using the Picard--Lefschetz formula in \'etale cohomology (\cite[Exposé XV]{sga7ii});
    \item when $S$ is of pure characteristic $p\geq 0$ in \cite[Exposé XVI]{sga7ii};
    \item when the relative dimension $n$ is $1$, by combining \cite{bloch87} with \cite{orgogozo03};
    \item when the action of $\uI_{\uK}$ on $\Phi$ is unipotent in \cite{beraldopippi22}.
\end{itemize}

\end{rmk}

\ssec{Bloch conductor conjecture}

In his seminal paper \cite{bloch87}, Bloch proposed a generalization of the Deligne--Milnor conjecture that we now recall.

\sssec{} 

Let us keep the notation of the previous subsection. 
Let $p:X\to S$ be a flat proper $S$-scheme. Suppose moreover that $X$ is regular and that $p_{\eta}:X_{\eta}\to \eta$ is smooth. Denote by $d$ the relative dimension of $X$ over $S$.

\sssec{} 

Bloch defined a \emph{localized Chern class} adapting the \emph{graph construction} of Mac-Pherson and Fulton (\cite{fulton98}).
This is an element 
$$
  c_{d+1,X_s}^{X}(\Omega^1_{X/S}) \in \CH_{d+1}(X_s\to X)
$$
in the bivariant Chow group of $X/S$. From this, one obtains a class
$$
  c_{d+1,X_s}^{X}(\Omega^1_{X/S})\cap [X] \in \CH_0(X_s),
$$
which yields a \emph{localized intersection number} upon taking the degree:
$$
  (\Delta_X,\Delta_X)_S:=(-1)^{d+1} \on{deg} \Bigt{ c_{d+1,X_s}^{X}(\Omega^1_{X/S})\cap [X]}  \in \CH_0(s)\simeq \bbZ.
$$
This number is called \emph{Bloch intersection number} and, in the sequel, it will be often denoted by $\Bl(X/S)$.

\begin{rmk}
The class $c_{d+1,X_s}^{X}(\Omega^1_{X/S})\cap [X]$ is supported on the singular locus of $p:X\to S$. Hence, $\Bl(X/S)$ is well-defined with the weaker assumption that the singular locus of $p:X\to S$ is proper.
\end{rmk}

\begin{rmk}
Kato--Saito reinterpreted $\Bl(X/S)$ in terms of the stable Tors of the diagonal $\delta_{X/S}: X \to X\times_ S X$.
In \cite{beraldopippi22}, we used this reformulation to categorify $\Bl(X/S)$: we expressed it using matrix factorizations.
\end{rmk}

\sssec{}

The number $\Bl(X/S) = (\Delta_X,\Delta_X)_S$ will appear on the left-hand side of Bloch conductor formula. To formulate the right-hand side, we need the following notation:
\begin{itemize}
\item 
the $\Qell$-adic Euler characteristic of the special fiber
$$
  \chi(X_s;\Qell) := \Tr_{\Qell}\bigl (\id:\uH^*_{\et}(X_s,\Qell)\bigr ) 
  := \sum_{j\in \bbZ}(-1)^j  \dim_{\Qell}\bigl ( \uH^j_{\et}(X_s,\Qell)\bigr ); 
$$
\item 
the $\Qell$-adic Euler characteristic of the geometric generic fiber
$$
  \chi(X_{\bar{\eta}};\Qell)
  :=
  \Tr_{\Qell}\bigl (\id:\uH^*_{\et}(X_{\bar{\eta}},\Qell)\bigr ) 
  := \sum_{j\in \bbZ}(-1)^j \dim_{\Qell}\bigl ( \uH^j_{\et}(X_{\bar{\eta}},\Qell)\bigr ); 
$$
\item for a finite Galois extension $\uK \subseteq \uL (\subseteq \uKs)$ with the property that the wild inertia group $\uP_{\uL}$ acts trivially on $\uH(X_{\bar{\eta}},\Ql{,X_{\bar{\eta}}})$ (this extension always exists, see \cite{sga7i}), the Swan conductor of the geometric fiber is defined as
\begin{align*}
     \Sw(X_{\eta}/\eta;\Qell) & = \frac{1}{|\GLK|} \sum_{g\in \GLK}\sw(g) \Tr_{\Qell{}}\bigl (g:\uH^*_{\et}(X_{\bar{\eta}},\Qell) \bigr ) \\
                            & = \frac{1}{|\GLK|}\sum_{g\in \GLK}\sw(g) \Bigt{ \sum_{j\in \bbZ}(-1)^{j} \Tr_{\Qell{}}\bigl (g:\uH^j_{\et}(X_{\bar{\eta}},\Qell) \bigr )}.
\end{align*}
For the same reason as in Remark \ref{rmk: Sw always well defined}, this is well defined even if $\GLK$ does not act on $\uH_{\et}^*(X_{\bareta},\Qell)$.
With a similar formula, one defines the Swan conductor $\Sw(V)$ of any finite dimensional $\ell$-adic $\IK$-representation $V$.

\item 
finally, the total dimension of $X/S$ is defined as
$$
\dimtot(X/S; \Qell)
:=
 \chi(X_{\bar{\eta}};\Qell) 
-
\chi(X_s;\Qell)
+
\Sw(X_{\eta}/\eta;\Qell)
=
\chi(\Phi)+\Sw(\Phi),
$$
where $\Phi:= \uH^*_{\et}\bigt{X_s, \Phi_p(\Ql{,X})}$.
With a similar formula, one defines the total dimension of any finite-dimensional $\ell$-adic $\IK$-representation $V$.
\end{itemize}

\sssec{}

In \cite{bloch87}, Bloch proposed the following
\begin{unconj}[Bloch conductor formula]
With the same notation as above,
$$
  \Bl(X/S)
  = 
  - \dimtot(X/S;\Qell).
$$
\end{unconj}

\begin{rmk}
This conjecture has been proved in the following cases:
\begin{itemize}
    \item the case $d=1$ was proven by Bloch in \cite{bloch87};
    \item the case in pure characteristic zero follows from the work of Kapranov, see \cite{kapranov};
    \item when the reduced special fiber $(X_s)_{\on{red}}$ is a simple normal crossing divisor, the conjecture was proven by Kato and Saito in \cite{katosaito04};
    \item the pure-characteristic case was proven by Saito in \cite{saito21}.
    \item the unipotent case, under the additional hypothesis that $X$ is a hypersurface in a smooth $S$-scheme, has been proven in \cite{beraldopippi22}.
\end{itemize}
\end{rmk}

\sssec{}

In \cite{orgogozo03}, Orgogozo suggests that a similar formula ought to hold when only the singular locus of $X\to S$ (and not necessarily $X$ itself) is proper over $S$. 
In the same paper, he proves that the Deligne--Milnor formula is a particular case of Bloch conductor formula. For future reference, we record this statement here, where as before
$$
  \Phi:= \uH^*_{\et}\bigt{X_s,\Phi_p(\Ql{,X})}
  \simeq
  \uH^*_{\et,c}\bigt{X_s,\Phi_p(\Ql{,X})}.
$$

\begin{unconj}[Generalized Bloch conductor formula]
Let $S$ be as above.
Assume that $p:X\to S$ is a flat, geometrically connected, generically smooth regular $S$-scheme whose singular locus is proper over $S$. Then the following formula holds:
$$
  \Bl(X/S)
  =
  -\dimtot (\Phi).
$$

\end{unconj}

\sssec{}

It is convenient to introduce the \emph{Artin conductor} of an $\ell$-adic $\IK$-representation as well.
Recall that the \emph{Artin character}
$$
  \ar: \GLK \to \bbZ
$$
of $\uK \subseteq \uL$ is defined by
$$
  g \mapsto 
  \begin{cases}
  \lg_{A}(\Omega^1_{A'/A}) & \text{ if } g=\id;\\
  -\lg_{A}\bigt{A'/\langle g(\pi_{\uL})-\pi_{\uL}\rangle} & \text{ if } g\neq \id.
  \end{cases}
$$
This function is the character of an $\ell$-adic representation, the \emph{Artin representation} (\cite{serre60}): in fact, the Artin representation is the sum of the Swan representation and the augmentation representation (the quotient of the regular representation by the trivial representations):
$$
  \sw 
  =
  \ar
  -
  \on{reg}_{\GLK}
  +
  \on{triv}_{\GLK}.
$$

\sssec{} 

Let $V$ be an object in $\Shvcon^\IK(s)$. 
Assume that $\IL$ acts unipotently on $V$, so that for every endomorphism $f:V \to V$ we have
$$
  \Tr_{\Qell}\bigt{f:V}
  =
  \Tr_{\Qell^{\IL}}\bigt{f:V^\IL},
$$
where the trace on the right hand side is computed in the stable symmetric monoidal $\oo$-category $\Mod_{\Qell^{\IL}}$.
We define the \emph{Artin conductor} of $V$ as
\begin{align}\label{eqn: def Ar}
  \Ar(V) & := \frac{1}{|\GLK|}\sum_{g\in \GLK}\ar(g) \Tr_{\Qell^{\IL}}(g:V^{\IL}) \\
  \nonumber
         & = \chi(V) + \Sw(V) - \Tr_{\Qell^{\IK}}(\id:V^{\IK}).
\end{align}
In particular,
\begin{equation}\label{eqn: dimtot = Ar + Tr_(Qell^I)(id:V^I)}
  \dimtot(V) = \Ar(V) + \Tr_{\Qell^{\IK}}(\id:V^{\IK}).
\end{equation}

\begin{rmk}
Our definition of the Artin conductor \emph{is different} from other definitions appearing in the literature, where it is defined as the number
$$
  \Sw(V)+\Tr_{\Qell}(\id:V)-\Tr_{\Qell}(\id:V^\IK),
$$
with $V^\IK$ standing for usual fixed points (as opposed to the derived fixed points).
\end{rmk}

\begin{rmk}
The definition of the Artin conductor as in \eqref{eqn: def Ar} involves the $\ell$-adic number $\Tr_{\Qell^{\IK}}(\id:V^\IK)$, which is inherently derived. This complication does not appear in the definition of the Swan conductor because, as explained in Remark \ref{rmk: Sw always well defined}, there only the action of $\PLK$ needs to be taken into account. On the contrary, for the Artin conductor, we need the entire $\GLK$, which does not act on $V$.
\end{rmk}

\sssec{}

With the definition of the Artin conductor at hand, we immediately obtain the following alternative expression for the generalized Bloch conductor formula:
\begin{equation}\label{eqn:Bloch-formula-with-Artin}
  \Bl(X/S)
  =
 - \Tr_{\Qell^\IK}\bigt{\id: \Phi^\IK}  -\Ar(\Phi) .    
\end{equation}
This is the shape of the formula that is best suited for our methods: for instance, compare this formula with that of Theorem \ref{thm: categorical Bloch}.

\ssec{Categorical Bloch conductor formula}

Let us now outline our methods, which naturally yield a categorical version of Bloch conductor conjecture. In Section \ref{ssec:intro-classical conjectures} below, we will explain how this categorical version relates to the classical version.

\sssec{} 

We tackle the generalized Bloch conductor conjecture (thus in particular the Deligne--Milnor conjecture) by working with categories of sheaves on some schemes attached to $X/S$. 
More precisely, we will work with differential graded categories (dg-categories, from now on) of \emph{singularities}, see Section \ref{ssec: singularity categories} for the definitions. 
The dg-category of singularities of a regular scheme is trivial, hence we must look for singular (that is, non-regular) schemes in our situation. 

\sssec{} 

One such singular scheme is the special fiber $X_s := s \times_S X$. 
Following \cite{toenvezzosi22}, we define the dg-category
$$
\sT:= \Sing(X_s)
$$
and observe that it is equipped with the convolution action of the monoidal dg-category
$$
\sB := \Sing(s \times_S s).
$$
Here $s \times_S s$ is the derived self-intersection of the closed point of $S$.
This is an object of \emph{derived algebraic geometry}: precisely, it is a quasi-smooth derived groupoid.\footnote{In the pure-characteristic case, $G$ is actually a derived group scheme, due to the presence of a retraction $s \hto S \to s$. This retraction partially accounts for the second categorification in the pure-characteristic being considerably easier than in the mixed-characteristic case.} 
An important result due to \TV{} is that $\sT$ is a dualizable $\sB$-module, see Corollary \ref{cor: T dualizable B module}.

\sssec{} 

Another singular scheme arising naturally is the base-change $X' := X \times_S S'$ for some sufficiently large totally ramified extension $S'\to S$, see Section \ref{ssec: the C module U}. 
Thus we consider
$$
\sU := \Sing(X')
$$
and, parallel to the case of $\sT$, we notice that $\sU$ is endowed with an action of the monoidal dg-category
$$
\sC := \Sing(S' \times_S S').
$$
This dg-category is monoidal as $S' \times_S S'$ is a groupoid over $S'$. Notice that this groupoid is classical since the map $S' \to S$ is flat.

\sssec{}

A key fact is that the monoidal dg-categories $\sB$ and $\sC$ are \emph{Morita equivalent}, that is, they have equivalent $\infty$-categories of modules (Corollary \ref{cor:Morita}). Parallel to the case of Morita-equivalent rings (which have isomorphic centers), the monoidal dg-categories $\sB$ and $\sC$ have equivalent \emph{Drinfeld centers} (or \emph{Hochschild cohomologies}). 
In Theorem \ref{thm:HH-are-equal}, we will show that $\sB$ and $\sC$ have also isomorphic \emph{Hochschild homologies}:
\begin{equation}\label{eqn:intro-HH=HH}
 \HH_*(\sB/A)
 \simeq
 \HH_*(\sC/A).    
\end{equation}

\sssec{}

The categorical Bloch conductor formula is an equality of classes in
\begin{equation} \label{eqn:intro-ambiente per catBCC}
\uH^0_{\et}\Bigt{S, \rl_S \bigt{\HH_*(\sB/A)}}.   
\end{equation}
Here $\rl_S$ is a decategorification procedure,
called $\ell$-adic realization: roughly speaking, it associates an $\ell$-adic sheaf to any $A$-linear dg-category.
By the above discussion, we can express $\rl_S \bigt{\HH_*(\sB/A)}$ alternatively as $\rl_S \bigt{\HH_*(\sC/A)}$.
To produce classes in \eqref{eqn:intro-ambiente per catBCC}, we will use the formalism of categorical traces.

\sssec{}

The dualizability of $\sT$ over $\sB$ allows us to take the \emph{categorical trace} of any $\sB$-linear endofunctor $F: \sT \to \sT$. By construction, this trace in an object
$$
\Tr_\sB(F : \sT)
\in 
\HH_*(\sB/A).
$$
We will be mostly interested in $\Tr_\sB(\id:\sT)$, the trace of the identity. Indeed, the \emph{categorical Bloch intersection class} is an element
$$
\Bl^{\cat}(X/S)
\in 
\uH^0_\et\bigt{S, \rl_S \bigt{\HH_*(\sB/A)}}
$$
defined purely in terms of the dg-functor
$\Tr_{\sB}(\id:\sT)$.
This construction is due to \TV{} (\cite{toenvezzosi22}); see Section \ref{sssec:cat Bloch class} for a quick summary.

\sssec{}

It turns out that the $\sB$-module $\sT$ and the $\sC$-module $\sU$ correspond to each other under the above Morita equivalence. We obtain that $\sU$ is a dualizable $\sC$-module (Proposition \ref{prop: U dualizable C module}) and that 
$$
  \Tr_{\sB}(\id:\sT)
  \simeq
  \Tr_{\sC}(\id:\sU).
$$
Observe that this formula makes sense thanks to the equivalence \eqref{eqn:intro-HH=HH}. In particular, we can alternatively express $\Bl^{\cat}(X/S)$ purely in terms of $\Tr_{\sC}(\id : \sU)$.

\sssec{} 

A fundamental feature of the $\infty$-functor $\rl_S$ is its \emph{lax-monoidality}. It follows formally that $\rl_S(\sB)$ is an algebra object that acts on $\rl_S(\sT)$. Likewise, $\rl_S(\sC)$ is an algebra object acting on $\rl_S(\sU)$.
In Section \ref{ssec: the r(B)-module r(T)}, we will show that $\rl_S(\sT)$ is a dualizable $\rl_S(\sB)$-module. Similarly, in Section \ref{ssec: the r(C)-module r(U)} we will show that $\rl_S(\sU)$ is a dualizable left $\rl_S(\sC)$-module. %
Hence, we can take the traces of the identity in the two cases:
$$
\Tr_{\rl_S(\sB)}\bigt{\id: \rl_S(\sT)}
\in 
\uH^0_\et\Bigt{
S,
\HH_*\bigt{\rl_S(\sB)/\rl_S(A)}
} 
$$
$$
\Tr_{\rl_S(\sC)}\bigt{\id: \rl_S(\sU)}
\in 
\uH^0_\et\Bigt{
S,
\HH_*\bigt{\rl_S(\sC)/\rl_S(A)}
}.
$$
By the lax-monoidal structure of $\rl_S$ we obtain arrows
$$
\mu_\sB: 
\HH_*\bigt{\rl_S(\sB)/\rl_S(A)}
\to 
\rl_S\bigt{
\HH_*(\sB/A)
} 
$$
$$
\mu_\sC: 
\HH_*\bigt{\rl_S(\sC)/\rl_S(A)}
\to 
\rl_S\bigt{
\HH_*(\sC/A)} .
$$

\sssec{}

We use $\mu_\sB$ and $\mu_{\sC}$ to view the two traces above as elements of $\uH^0_{\et}\bigt{S, \rl_S \bigt{\HH_*(\sB/A)}}$: we can thus compare them with each other and with the categorical Bloch intersection class. 
In Section \ref{sec: categorical BCC}, we will prove that
\begin{equation}\label{eqn:intro-Blcat as a sum, first formula}
\Bl^{\cat}(X/S)
=
\mu_{\sB}
\Bigt{
\Tr_{\rl_S(\sB)}\bigt{\id: \rl_S(\sT)}
}
+ 
\mu_{\sC}
\Bigt{
\Tr_{\rl_S(\sC)}\bigt{\id: \rl_S(\sU)}
}.
\end{equation}

\sssec{} 

We will actually show that these two traces define elements
\begin{equation}\label{eqn:intro-traces of rl}
\Tr_{\rl_S(\sB)}\bigt{\id: \rl_S(\sT)}
\in 
\Qell
\end{equation}
\begin{equation}
\Tr_{\rl_S(\sC)}\bigt{\id: \rl_S(\sU)}
\in 
\HH_0\bigt{\Qell(\GLK)/\Qell}
\end{equation}
and that we have canonical arrows
$$
\idcat: \Qell
\hto
\uH^0_{\et} \Bigt{S,\HH_*\bigt{\rl_S(\sB)/\rl_S(A)}}
\xto{\mu_{\sB}}
\uH^0_\et\Bigt{
S,
\rl_S\bigt{\HH_*(\sB/A)}
}, 
$$
$$
-\arcat:
\HH_0\bigt{\Qell(\GLK)/\Qell}
\hto 
\uH^0_{\et} \bigt{S,\bigt{\rl_S(\sC)/\rl_S(A)}}
\xto {\mu_{\sC}}
\uH^0_\et\Bigt{
S,
\rl_S\bigt{\HH_*(\sC/A)}
}.
$$
See Section \ref{ssec:intro-classical conjectures} below for the reason of the notations $\idcat$ and $\arcat$.

\sssec{}

By \cite[Theorem 4.39]{brtv18}, 
we know that the class $\mu_{\sB}\Bigt{\Tr_{\rl_S(\sB)}\bigt{\id: \rl_S(\sT)}}$ is equal to 
$$
  \idcat \Bigt{\Tr_{\rl_S(\sB)}\bigt{\id:\rl_S(\sT)}}
  =
  -\idcat \Bigt{\Tr_{\Qell^{\IK}}\bigt{\id:\Phi^\IK}}.
$$ 
We define the \emph{categorical Artin conductor} as the opposite of the class $\mu_{\sC}\Bigt{\Tr_{\rl_S(\sC)}\bigt{\id: \rl_S(\sU)}}$:
\begin{align*}
\Ar^{\cat}(\Phi) & :=
-\arcat \Bigt{
  \Tr_{\rl_S(\sC)}\bigt{\id:\rl_S(\sU)}
  }\\
  & =
  -\arcat \Bigt{
  \frac{1}{|\GLK|}\sum_{g\in \GLK} \Tr_{\Qell^\IL}(g^{-1}:\Phi^\IL) \cdot \langle g \rangle 
  }
  \in
  \uH^0_\et\Bigt{S, \rl_S\bigt{\HH_*(\sC/A)}}.
\end{align*}
The latter equality will be proved in Section \ref{sec: categorical BCC}.

\medskip

These two terms sum up to give the opposite of what we call \emph{categorical total dimension} of $X/S$, to be denoted $\dimtot^{\cat} (\Phi)$.

\begin{mainthm}[{Categorical generalized Bloch conductor formula}]\label{thm: categorical Bloch}
With the same hypotheses as in the statement of the generalized Bloch conductor conjecture, we have
\begin{align*}
\Bl^{\cat}(X/S) & = -\idcat \bigt{\Tr_{\Qell^\IK}(\id:\Phi^\IK)}- \Ar^{\cat}(\Phi)
\\
                & =: -\dimtot^{\cat} (\Phi).
\end{align*}
\end{mainthm}
\begin{rem}

We have the following interesting extreme cases.
\begin{itemize}
    \item 
In the unipotent case considered by \TV{}, the second term on the right-hand side vanishes: in other words, $\Bl^{\cat}(X/S)$ comes entirely from $\rl_S(\sT)$. This is the version of Bloch conductor conjecture proven in \cite[Theorem 5.2.2]{toenvezzosi22}.
    \item 
Symmetrically, when $\Phi^{\IK} \simeq 0$, the first term on the right-hand side vanishes: in this case, $\Bl^{\cat}(X/S)$ comes entirely from $\rl_S(\sU)$.
\end{itemize}
\end{rem}

\ssec{Connection with the \virg{classical} conjectures} \label{ssec:intro-classical conjectures}

\sssec{} 

We believe that the categorical generalized Bloch conductor formula \emph{is} the generalized Bloch conductor formula. To prove this is the case, one needs to have control of the dg-category
$$
  \HH_*(\sB/A)
  \simeq
  \HH_*(\sC/A)
$$
and of certain dg-functors mapping to it.
More precisely, \TV{} conjectured that
$$
\uH^0_\et\Bigt{S, \rl_S\bigt{\HH_*(\sB/A)}} 
\simeq 
\Qell;
$$
we believe that, under this isomorphism, the maps
$$
\idcat:
\Qell 
\to
\uH^0_\et\Bigt{S, \rl_S \bigt{\HH_*(\sB/A)}}
$$
$$
-\arcat:
\HH_0\bigt{\Qell(\GLK)/\Qell} 
\to
\uH^0_\et\Bigt{S, \rl_S \bigt{\HH_*(\sB/A)}}
$$
correspond to the identity and the negative of the Artin character, respectively. We are not able to show this yet.
Nevertheless, our theorems below provide some evidence corroborating these expectations.

\sssec{}
In the pure-characteristic case, a major simplification occurs: the monoidal dg-category $\sB$ is symmetric monoidal. 
This implies the existence of a retraction $\HH_*(\sB/A)\to \sB$ which allows to prove that the above guesses are correct when we post-compose with the map
$$
  \uH^0_\et\Bigt{S, \rl_S \bigt{\HH_*(\sB/A)}}
  \to
  \uH^0_\et\bigt{S, \rl_S(\sB)}
  \simeq
  \Qell.
$$
We thus obtain the generalized Bloch conductor conjecture in pure characteristic.

\begin{mainthm}\label{mainthm: pure char}
Let $S$ be of pure characteristic. 
Then the generalized Bloch conductor conjecture holds true.
\end{mainthm}

\sssec{}

In mixed characteristic, the monoidal dg-category $\sB$ is not symmetric monoidal, hence we cannot appeal to the simplification used in the pure-characteristic case.
Nevertheless, under some additional hypothesis, our proof of the categorical version can be modified so as to bypass the dg-category $\HH_*(\sB/A)$.

\begin{mainthm}\label{mainthm: hypersurf}
Assume that $X$ is a hypersurface in a smooth $S$-scheme.
Then the generalized Bloch conductor conjecture holds true.
\end{mainthm}

\noindent
Arguing as in \cite[]{beraldopippi22}, we see that the  Deligne--Milnor conjecture is a particular case of Theorem \ref{mainthm: hypersurf}. 
Hence:

\begin{uncor}[Deligne--Milnor formula]
The Deligne--Milnor conjecture holds true.
\end{uncor}

\begin{rmk}
Our proof goes through if, instead of the identity morphism, we consider any other $\sB$-linear endomorphism of $\sT$ (for instance, one induced by an $S$-endomorphism of $X$.)
This provides more general versions of the main theorems above, whose formulation is left to the interested reader.
\end{rmk}

\ssec{Structure of the paper}

The rest of this paper is organized as follows.
\begin{itemize}
    \item In Section \ref{sec: preliminaries}, we fix some notation and recall some background in derived and non-commutative algebraic geometry. In particular, we review the basics on the theory of dualizable modules over a monoidal dg-category.
    
    \item In Section \ref{sec: nctraces}, we introduce and study the dg-categories of singularities that play a main role in this paper, namely the monoidal dg-categories $\sB$ and $\sC$ and their modules $\sT$ and $\sU$. Then we introduce a $(\sB,\sC)$-bimodule $\sE$.
    
    \item In Section \ref{sec:morita}, we prove that the monoidal dg-categories $\sB$ and $\sC$ are Morita equivalent by means of the bimodule $\sE$. As a consequence, we prove that $\sU$ is a dualizable $\sC$-module.
    Finally we prove that $\sB$ and $\sC$ have isomorphic Hochschild homologies over $A$: this fact is crucial for the proofs of the main theorems. 
    
    \item In Section \ref{sec: realizations}, we study the decategorification of $\sT/\sB$ (resp. $\sU/\sC$) provided by $\ell$-adic realization.
    The main result of this section is that the decategorification of $\sT$ (resp. $\sU$) is dualizable over that of $\sB$ (resp. $\sC$).
    Furthermore, we show that the duality datum of the left $\sB$-module $\sT$ (equivalently, the duality datum of the left $\sC$-module $\sU$) induces a duality datum of the decategorification of $\sT$ (resp. $\sU$) as a module over the decategorification of $\sB$ (resp. $\sC$).
    
    \item In Section \ref{sec: categorical BCC}, we prove Theorem \ref{thm: categorical Bloch}, that is, the categorical generalized Bloch conductor conjecture.
    Finally, we deduce Theorem \ref{mainthm: pure char} from the categorical version.
    
    \item In Section \ref{sec: connection to the classical conjectures}, we adapt our proof of Theorem \ref{thm: categorical Bloch} to obtain Theorem \ref{mainthm: hypersurf}.
 
    \item In Appendix \ref{appendix: traces over group algebras}, we recall some trace formulas for dualizable modules over a group algebra.
\end{itemize}

\sec*{Acknowledgements}

We are deeply grateful to Bertrand To\"en and Gabriele Vezzosi: this work would simply not exist without their insights and their ideas, which they have always generously shared with us.

We are indebted to John Milnor, Pierre Deligne, Spencer Bloch, Kazuya Kato and Takeshi Saito who influenced us through their works.

We wish to thank Benjamin Hennion, Valerio Melani, Mauro Porta and Marco Robalo for useful discussions over the years.

Last but not least, we are thankful to a fortune cookie whose ticket provided essential motivational support at a stage of the project when it was most needed. 

This project has received funding from the PEPS JCJC 2024.
\section{Preliminaries and notation}\label{sec: preliminaries}

In this section we recall some constructions, notions and facts that we will need in this paper.

\ssec{Conventions}

\sssec{}

We will employ the language of higher category theory of \cite{luriehtt} and \cite{lurieha}. 
All functors will be implicitly derived. For instance, we will write $-\otimes -$ instead of $-\otimes^{\bbL} -$ for the derived tensor product.

\sssec{}

All schemes will be assumed to be of finite type over a regular noetherian base of finite Krull dimension (often a complete and strict henselian trait).

\sssec{}

In the body of the paper, $S'\to S$ will always correspond to a totally ramified extension of DVR's $A=\scrO_{\uK} \subseteq \scrO_{\uL}=A'$ such that the inertia group $\IL$ acts unipotently on the $\ell$-adic representation under exam.
For reasons that will become clear later, we will need to assume that the extension is not trivial, i.e. that $S'\to S$ is not an isomorphism.

\sssec{} 

Whenever an $\ell$-adic sheaf $M$ is endowed with the trivial action of $\IK$, we will not distinguish between $\IK$-fixed points and $\IL$-fixed points and simply write $M^{\up{I}}$.

\sssec{}

For a dg-category $T$ we will denote $\HK(T)$ the spectrum of the homotopy-invariant non-connective algebraic K-theory of $T$.

\sssec{}
Whenever we have a stable symmetric monoidal $\oo$-category $\ccC^{\otimes}$, whenever we have a morphism
the following notation will be employed for the cohomology class associated to a morphism $f:1_{\ccC}\to M$:
$$
  [f]
  \in
  \uH^0(M)=\pi_0 \bigt{\Hom(1_{\ccC},M)}.
$$
Moreover, we will use the following convention. Let $A$ be an associative algebra and assume that $M$ is a dualizable left $A$-module with duality datum $(\coev,\ev)$. For every $A$-linear endomorphism $g:M\to M$ we will denote
$$
  \Tr_A(g:M)
  \in
  \HH_0(A)
  :=
  \pi_0\Bigt{\Hom_{\ccC}\bigt{1_{\ccC},\HH_*(A)}}
$$
the cohomology class associated to the morphism
$$
  1_{\ccC}
  \xto{\coev}
  M^{\vee}\otimes_{A}M
  \xto{\id \otimes g}
  M^{\vee}\otimes_{A}M
  \xto{\ev^{\HH}}
  \HH_*(A).
$$
If $A$ is commutative, we will also denote $\Tr_A(g:M)$ the associated element in $\uH^0(A)$.

\sssec{}
For a quasi-compact quasi-separated morphism $q:Z\to T$ (where $T$ is either $s$ or $S$) we will denote by $\Ql{,Z}$ and $\oml{,Z}:=q^!\Ql{,T}$ the constant and dualizing $\ell$-adic sheaves respectively.

\ssec{Derived algebraic geometry}

Derived algebraic geometry has been developed by \TV{} (\cite{toenvezzosi08}) and Lurie (\cite{luriedag1}).
For a gentle introduction to the theory of derived algebraic geometry we refer the reader to \cite{toen14}.
For the intimately related (but different) theory of spectral algebraic geometry the reader can consult \cite{luriesag}.

\sssec{}

Derived algebraic geometry is a homotopical version of algebraic geometry where rings are replaced by simplicial rings.
Nontrivial elements in higher homotopy groups of simplicial rings retain nontrivial homotopical information.

\sssec{}

Simplicial rings are regarded up to quasi-isomorphism. 
In other words they constitute the objects of an $\oo$-category of derived rings.
For a simplicial ring $R$, the set $\pi_0(R)$ inherits a ring structure.
Moreover, $\pi_i(R)$ is a $\pi_0(R)$-module for every $i\in \bbZ$. The same holds true if we work with a sheaf of simplicial rings instead.

\sssec{}

Consider a pair $(X,\ccO_X)$, where $X$ is a topological space and $\ccO_X$ a sheaf of simplicial rings.
Such a pair is a derived scheme if
\begin{itemize}
\item $(X,\pi_0(\ccO_X))$ is a scheme;
\item $\pi_i(\ccO_X)$ is a quasi-coherent $\pi_0(\ccO_X)$-module for every $i\in \bbZ$.
\end{itemize}
A derived scheme is said to be bounded if $\pi_{i}(\ccO_X)\simeq 0$ for all but finitely many $i\in \bbZ$.
All derived schemes will be implicitly supposed to be bounded.

\sssec{}

Derived schemes are naturally regarded as objects in an $\oo$-category $\dSch$.
There is an adjunction 
\begin{equation*}
  \begin{tikzpicture}[scale=1.5]
    \node (a) at (0,1) {$\Sch$};
    \node (b) at (2,1) {$\dSch$.};
    \path[->,font=\scriptsize,>=angle 90]
         ([yshift= 1.5pt]a.east) edge node[above] {$\iota$} ([yshift= 1.5pt]b.west);
    \path[->,font=\scriptsize,>=angle 90]
         ([yshift= -1.5pt]b.west) edge node[below] {$\pi_0$} ([yshift= -1.5pt]a.east);
  \end{tikzpicture}
\end{equation*}
between the $\oo$-category of derived schemes and the (usual) category of schemes $\Sch$. The functor $\iota: \Sch \to \dSch$ is fully faithful.

\sssec{}

The schemes with a non-trivial derived structure that appear in the present paper are of a precise nature, namely they arise as (derived) fiber products of diagrams of classical schemes. 
The classical scheme associated to such a fiber product coincides with the usual fiber product of schemes.

\ssec{Morita theory of dg-categories}

We refer to \cite{keller06,toen08} for more details on the Morita theory of dg-categories.

\sssec{}

Let $R$ be a commutative ring.
An $R$-linear dg-category is a category enriched over the category of $R$-cochain complexes.

\sssec{}

If $\sT$ is an $R$-linear dg-category, one can consider its \emph{homotopy category} $\mathsf{hT}$: this is an $R$-linear category with the same objects as $\sT$ and with morphisms given by
$$
  \Hom_{\mathsf{hT}}(x,y)
  :=
  \uH^0\bigt{\Hom_{\sT}(x,y)}.
$$
The composition law in $\mathsf{hT}$ is defined in the obvious way.
The assignment $\sT \rightsquigarrow \mathsf{hT}$ is functorial in $\sT$.

\sssec{}

A dg-functor between two $R$-linear dg-categories is a functor compatible with the enrichment.
A dg-functor is \emph{quasi-fully faithful} if it induces quasi-isomorphisms on the hom-complexes.
It is \emph{quasi-essentially surjective} if it induces an essentially-surjective functor at the level of homotopy categories.
A \emph{quasi-equivalence} is a dg-functor that is both quasi-fully faithful and quasi-essentially surjective.

\sssec{}

A dg-functor is a \emph{Morita equivalence} if it induces a quasi-equivalence at the level of dg-categories of dg-modules (\cite{toen07}).
The class of Morita equivalences contains that of quasi-equivalences. 
Concretely, a dg-functor $F:\sT\to \sU$ is a Morita equivalence if it is quasi-fully faithful and if the essential image of $F$ generates $\sU$ under finite colimits and retracts.

\sssec{}

We will denote by $\dgCat_R$ the associated $\oo$-category, i.e. the $\oo$-localization of the category of dg-categories with respect to Morita equivalences.

\sssec{}

In \cite{toen07}, To\"en shows that $\dgCat_R$ admits a theory of localizations: for a dg-category $\sT$ and a set of morphisms $W$ in $\sT$, there exists a morphism $\sT \to L_W \sT$ in $\dgCat_R$ such that for every dg-category $\sU$ the induced $\oo$-functor
$$
  \Fun_{\dgCat_R}(L_W\sT,\sU)
  \to
  \Fun_{\dgCat_R}(\sT,\sU)
$$
is fully-faithful and the essential image consists of those dg-functors $\sT\to \sU$ such that the image of $W$ lands in the set of equivalences of $\sU$.

\sssec{}

In particular, one can use the above construction to construct dg-quotients as follows.
Consider a quasi-fully faithful embedding $\sU \hto \sT$ and let $W$ be the collection of the morphisms in $\sT$ whose cone belongs to $\sU$. 
Then the dg-quotient associated to $\sU \hto \sT$ is defined as
$$
  \sT/\sU
  :=
  L_{W} \sT.
$$

\sssec{}
A \emph{localization sequence} in $\dgCat_R$ is a diagram
$$
  \sU
  \to
  \sT
  \to
  \sT'
$$
such that the composition is zero, the arrow $\sU \to \sT$ is quasi-fully faithful, and $\sT'\simeq \sT/\sU$.

\sssec{}

The derived tensor product of dg-categories, constructed in \cite{toen07}, endows $\dgCat_R$ with the structure of a symmetric monoidal $\infty$-category. A \emph{monoidal $R$-linear dg-category} is a unital associative algebra object in $\dgCat_R$.

\ssec{Duality for modules over a monoidal dg-category}

In this subsection we recall some general notions about dualizable modules over a monoidal dg-category.
For further details we refer the reader to \cite[Section 4.6.2]{lurieha} and \cite[Section 2]{toenvezzosi22}.

\sssec{}

For a monoidal ($R$-linear) dg-category $\sA$, we denote by $\sA^{\rev}$ the monoidal dg-category obtained from $\sA$ by \emph{reversing} the monoidal structure.

\sssec{}

We denote by $\Mod_{\sA}$ the $\oo$-category of \emph{left} $\sA$-modules.
When we need to consider the $\oo$-category of \emph{right} $\sA$-modules, we employ the notation $\Mod_{\sA^\rev}$.

If we are given a pair of monoidal $S$-linear dg-categories $\sA_1$ and $\sA_2$, we will denote by $\biMod_{(\sA_1,\sA_2)}$ the $\oo$-category of $(\sA_1,\sA_2)$-bimodules. 

The $\oo$-category of $(\sA_1,\sA_2)$-bimodules can be identified with the $\oo$-category of left $\sA_1\otimes_R \sA_2^\rev$-modules. In particular, we have the following identifications:
$$
\Mod_{\sA} 
\simeq
\biMod_{(\sA,R)}, 
\hspace{0.5cm}
\Mod_{\sA^\rev} 
\simeq
\biMod_{(R,\sA)}.
$$

\sssec{} 

Whenever we are given three monoidal dg-categories $\sA_1, \sA_2$ and $\sA_3$, there is a relative tensor product 
$$
-\otimes_{\sA_2}-:
\biMod_{(\sA_1,\sA_2)} 
\x{} 
\biMod_{(\sA_2,\sA_3)} 
\to
\biMod_{(\sA_1,\sA_3)}. 
$$
See \cite[Section 4.4]{lurieha}.%
The following particular case will be of interest in the sequel. Define the monoidal dg-category
$$
\sA^\env
:=
\sA^\rev 
\otimes_R
\sA
$$
and notice that $\sA$ carries a natural $\sA^\env$ left module structure.
Then, for a left $\sA$-module $\sM$ and a right $\sA$-module $\sN$, we have
$$
\sN \otimes_\sA \sM
\simeq
(\sN \otimes_R \sM)
\otimes_{\sA^\env}
\sA.
$$

\sssec{} 

Denote by $\dgCAT_R$ the $\oo$-category of compactly generated presentable $R$-linear dg-categories with continuous $R$-linear dg-functors.
Similarly to $\dgcat_R$, this $\infty$-category is symmetric monoidal when equipped with the tensor product defined in \cite[Section 4.8.1]{lurieha}, \cite{toen07}.

\sssec{}

The procedure of ind-completion yields a symmetric-monoidal (non-full) faithful embedding 
$$
  \widehat{(-)}: 
  \dgcat_R 
  \to
  \dgCAT_R,
$$
which is compatible with relative tensor products:
$$
\widehat{\sN}\otimes_{\widehat{\sA}}\widehat{\sM}
\simeq 
\widehat{\sN \otimes_\sA \sM}.
$$
\sssec{} 

After \cite{toen07}, it is known that $\dgCAT_R$ is a rigid symmetric monoidal $\oo$-category. 
Moreover, for $\sT \in \dgCat_R$, the dual of $\widehat{\sT}$ in $\dgCAT_R$ is $\widehat{\sT^{\op}}\simeq \widehat{\sT}^{\op}$.

For $\sM$ a left $\sA$-module, we have that  $\widehat{\sM}$ is a left $\widehat{\sA}$-module while $\widehat{\sM}^{\op}$ is a right $\widehat{\sA}$-module.

Thus, we have two dg-functors in $\dgCAT_R$
$$
\mu: 
\widehat{\sA}\otimes_{\widehat{R}} \widehat{\sM} 
\to
\widehat{\sM},
\hspace{0.5cm}
\mu^{\op}:
\widehat{\sM}^{\op} \otimes_{\widehat{R}} \widehat{\sA}
\to 
\widehat{\sM}^{\op}.
$$

\begin{defn}\label{defn: cotensored A-module}
A left $\sA$-module $\sM$ is \emph{cotensored} over $\sA$ if the morphism $\mu^{\op}$ above lies in the subcategory $\dgCat_R\subseteq \dgCAT_R$.
\end{defn}

In other words, a left $\sA$-module $\sM$ is cotensored over $\sA$ if $\sM^{\op}$ is a right $\sA$-module.

\begin{rmk}\label{rmk: monoidal catogory generated by its unit object = every module is cotensored}
As underlined in \cite[Remark 2.1.4]{toenvezzosi22}, if the monoidal dg-category $\sA$ is Karoubi-generated by its unit object, then
$$
  \Mod_{\sA} 
  \to 
  \Mod_{\widehat{\sA}}
$$
is essentially surjective.

In particular, every $\sA$-module $\sM$ is cotensored over $\sA$.
\end{rmk}
\sssec{}

Suppose that the right adjoint $\mu^{\up{r}}$ to $\mu: \widehat{\sA}\otimes_{\widehat{R}} \widehat{\sM} \to \widehat{\sM}$ also belongs to $\dgCAT_R$.
By duality we obtain a dg-functor
$$
h:
\widehat{\sM}^{\op} \otimes_{\widehat{R}} \widehat{\sM}
\to 
\widehat{\sA}.
$$
\begin{defn}
We say that  $\sM$ is \emph{proper} over $\sA$ if the dg-functor $h:\widehat{\sM}^{\op} \otimes_{\widehat{R}} \widehat{\sM} \to \widehat{\sA}$ above lies in the subcategory $\dgCat_R\subseteq \dgCAT_R$.
\end{defn}

\begin{rmk}
By its very definition
$$
h:
\widehat{\sM}^{\op} \otimes_{\widehat{R}} \widehat{\sM} 
\to 
\widehat{\sA},
$$
is the dg-functor
$$
(x,y)\mapsto \Hom_{\widehat{\sM}}(x,y).
$$
Therefore, $\sM$ is proper over $\sA$ if it has a natural \emph{enrichment} over $\sA$.
\end{rmk}

\sssec{}

As proved in \cite[Proposition 2.4.6]{toenvezzosi22}, for every $\sM \in \Mod_{\sA}$, $\widehat{\sM}$ is a dualizable $\widehat{\sA}$-module (in the sense of \cite[Section 2.4]{toenvezzosi22}, \cite[Section 4.6.1]{lurieha}). 
Its right dual can be identified with $\widehat{\sM^\op}$.
In particular, we have a uniquely defined (up to a contractible space of choices) pair of dg-functors
$$
\widehat{\sM} \otimes_{\widehat{R}} \widehat{\sM}^{\op} 
\to
\widehat{\sA},
$$
$$
\widehat{R}
\to
\widehat{\sM}^\op \otimes_{\widehat{\sA}} \widehat{\sM} 
$$
providing a duality datum for the dualizable object $\widehat{\sM}$ in $\Mod_{\widehat{\sA}}$.

Notice that the first displayed functor is $(\widehat{\sA}^{\env})^\rev$-linear, while the second one is just $\widehat{R}$-linear.
\begin{defn}
We say that $\sM\in \Mod_{\sA}$ is \emph{saturated over $\sA$} if it is cotensored over $\sA$ and if the two morphisms above lie in the images of the $\oo$-functors $\widehat{(-)}:\Mod_{(\sA^{\env})^\rev} \to \Mod_{(\widehat{\sA}^{\env})^\rev}$ and $\widehat{(-)}:\dgCat_R\to \dgCAT_R$.
\end{defn}
\begin{notat}
If $\sM$ is saturared over $\sA$, we will denote the two relevant dg-functors as
$$
\ev_{\sM/\sA}:
\sM \otimes_R \sM^\op 
\to
\sA,
$$
$$
\coev_{\sM/\sA}:
R
\to
\sM^\op \otimes_{\sA} \sM 
$$
and refer to them as the \emph{evaluation} and \emph{coevaluation} respectively.

If no confusion arises, we will denote them simply as $\ev$ and $\coev$.
\end{notat}

\sssec{}

For $\sM$ a saturated $\sA$-module, we will write
$$
  \ev^{\HH}_{\sM/\sA}:
  \sM^\op \otimes_\sA \sM
  \simeq
  (\sM^\op \otimes_R \sM)
  \otimes_{\sA^\rev} 
  \sA
  \to
  \sA
  \otimes_{\sA^\rev}
  \sA
  =:
  \HH_*(\sA/R)
$$
for the map
$$
\bigt{
  \sM \otimes_R \sM^\op 
  \xto{\up{swap}}
  \sM^\op \otimes_R \sM
  \xto{\ev_{\sM/\sA}}
  \sA
}
\otimes_{\sA^\rev}
\sA.
$$
We will write $\ev^{\HH}$ if no confusion can arise.
\sssec{}

For every $\sM$ saturated over $\sA$, and for every $\sA$-linear endomorphism $\phi$ of $\sM$, one has the following $R$-linear dg-functor
$$
  \Perf(R)
  \xto{\coev}
  \sM^\op \otimes_\sA \sM
  \xto{\id \otimes \phi}
  \sM^\op \otimes_\sA \sM
  \xto{\ev^{\HH}}
  \HH_*(\sA/R)
$$
which can be identified with an object
$$
  \Tr_\sA(\phi:\sM)
  \in
  \HH_*(\sA/R).
$$

\ssec{Singularity categories}\label{ssec: singularity categories}
\sssec{}
Just as for usual schemes, one can attach certain dg-categories to derived schemes.

For $X=\Spec{A}$ an affine derived scheme, corresponding to a simplicial ring $A$, one puts
$$
  \QCoh(X):=\Mod_{N(A)},
$$
where $N(A)$ is the dg-algebra obtained from $A$ via the nerve construction.

When $X$ is an arbitrary derived scheme, one puts
$$
  \QCoh(X)
  :=
  \lim_{\Spec{A}\subseteq X}\Mod_{N(A)},
$$
where the limit runs over the category of Zariski open subsets of $X$.

\sssec{}

Inside $\QCoh(X)$ sit two particularly important dg-categories, namely the dg-category of perfect complexes $\Perf(X)$ and the dg-category $\Coh(X)$ of objects in $\QCoh(X)$ with bounded coherent cohomology.

\sssec{}

The singularity category $\Sing(Z)$ is a non-commutative invariant of $Z$ which is related to its singularities.
For example, it follows from the homological characterization of regularity (\cite{auslanderbuchsbaum56,auslanderbuchsbaum57,serre56}) that a noetherian scheme of finite Krull dimension is regular if and only if its singularity category is trivial.

\sssec{}
Let $Z$ be a bounded and noetherian derived scheme. 
Then we can consider the full embedding
$$
  \Perf(Z)\subseteq \Coh(Z).
$$
The \emph{singularity category of $Z$} is defined as the dg-quotient
$$
  \Sing(Z):=\Coh(Z)/\Perf(Z).
$$

\sssec{}

Let $Z$ be an hypersurface $V(f)$ in a regular scheme $Y$. 
Then $\Sing(Z)$ is equivalent to the dg-category of matrix factorizations $\MF(Y,f)$. 
This is a dg-category whose objects are quadruplets $(E,F,\phi,\psi)$, where $E$ and $F$ are projective $\ccO_Y$-modules of finite type and
$$
  \phi: 
  E 
  \to 
  F,
  \hspace{0.5cm}
  \psi:
  F 
  \to 
  E
$$
are $\ccO_Y$-linear morphisms such that
$$
  \psi \circ \phi
  =
  f\cdot \id_E,
  \hspace{0.5cm}
  \phi \circ \psi
  =
  f\cdot \id_F.
$$
We refer to \cite{buchweitz21, orlov04} for more details.

More generally, we can consider a global section $\sigma$ of a line bundle $L$ on $Y$. 
In this case, we will denote $K(Y,L,\sigma)$ the Koszul (derived) scheme $Y\times_{\sigma,L,0}Y$\footnote{When the line bundle $L$ is trivial we will simply write $K(Y,\sigma)$.}.
Then $\Sing \bigt{K(Y,L,\sigma)}$ is equivalent to the dg-category of twisted matrix factorizations $\MF(Y,L,\sigma)$.
This is a dg-category whose objects are quadruplets $(E,F,\phi,\psi)$, where $E$ and $F$ are vector bundles on $Y$and
$$
  \phi: 
  E 
  \to 
  F,
  \hspace{0.5cm}
  \psi:
  F 
  \to 
  E \otimes L
$$
are $\ccO_Y$-linear morphisms such that
$$
  \psi \circ \phi
  =
  \id_E\otimes \tau,
  \hspace{0.5cm}
  (\phi \otimes L) \circ \psi
  =
  \id_F \otimes \tau.
$$
\begin{rmk}
Notice that the above mentioned equivalence holds even when $\sigma=0$, in which case $\MF(Y,L,0)$ is equivalent to the dg-category of $L$-twisted $2$-periodic complexes on $Y$.
\end{rmk}
\sec{The dg-categories of singularities attached to an arithmetic scheme}\label{sec: nctraces}

In this section, we introduce several dg-categories of sheaves and study some relations among them.

\ssec{The monoidal dg-category \texorpdfstring{$\sB$}{B}}

In this section, we introduce the scheme $G := s \times_S s$, which is a groupoid over $s$, and the dg-categories $\sB^+ := \Coh(G)$ and $\sB :=\Sing(G)$. We provide strict models for them and show that they are endowed with natural convolution monoidal structures.

\sssec{}

We describe the groupoid structure on $G$ via the Hopf algebroid structure on its associated dg-algebra.

Consider the $A$-dg-algebra
$$
  K_A
  :=
  A[\varepsilon],
$$
where $\varepsilon$ is a free variable in cohomological degree $-1$ such that $d(\varepsilon)=\pi_{\uK}$.

Consider the dg-algebra
$$
  K_A^2
  :=
  K_A\otimes_A K_A
  \simeq
  A[\varepsilon_1,\varepsilon_2].
$$
Here, both $\varepsilon_1$ and $\varepsilon_2$ are free variables sitting in cohomological degree $-1$ whose differential is $\pi_{\uK}$.

Notice that $K_A$ is quasi-isomorphic to $k$ and that $K_A^2$ is quasi-isomorphic to the nerve of the simplicial algebra $k\otimes_Ak$. Thus, $K_A^2$ is the Hopf algebroid associated to the derived groupoid scheme
$$
  G =  s\times_S s.
$$

\sssec{}

As it is well-known, $K_A^2$ is a Hopf algebroid over $A$. 
The source and target morphisms are
$$
  1\otimes \id :
  K_A
  \to
  K_A^2,
  \hspace{0.5cm}
  \id \otimes 1 :
  K_A
  \to
  K_A^2.
$$
The unit of the Hopf algebroid corresponds to the multiplication of the $A$-dg-algebra $K_A$
$$
  K_A^2
  \to 
  K_A,
  \hspace{0.5cm}
  \varepsilon_1,\varepsilon_2
  \mapsto
  \varepsilon
$$
and the composition to
$$
  \id\otimes 1 \otimes \id:
  K_A^2
  \to
  K_A^2\otimes_{K_A}K_A^2.
$$
Finally, the antipode corresponds to 
$$
  K_A^2
  \to
  K_A^2,
  \hspace{0.5cm}
  \varepsilon_1
  \mapsto 
  \varepsilon_2,
  \hspace{0.2cm}
  \varepsilon_2
  \mapsto
  \varepsilon_1.
$$

\sssec{}

Consider the $A$-linear dg-category $\dgMod(K_A^2)$ of $K_A^2$-dg-modules.
\begin{rmk}
The datum of a $K_A^2$-dg-module is equivalent to the datum of an $A$-dg-module $(E,d)$ together with two $A$-linear morphisms
$$
  h_1,h_2:
  E
  \to
  E[-1]
$$
such that 
\begin{itemize}
    \item $h_1^2=0$, $h_2^2=0$;
    \item $[h_1,h_2]=0$;
    \item $[d,h_1]=\pi_{\uK} \cdot \id_E = [d,h_2]$.
\end{itemize}
A morphism between two such objects is a morphism in the underlying dg-category of $A$-dg-modules which is strictly compatible with the given morphisms of degree $-1$. 
See \cite[Section 2]{brtv18} and \cite{pippi22b} for more details.
\end{rmk}

\sssec{}

The Hopf algebroid structure on $K_A^2$ endows $\dgMod(K_A^2)$ with an associative and unital (but not necessarely symmetric) monoidal structure
$$
  -\odot -:
  \dgMod(K_A^2) \otimes_A \dgMod(K_A^2)
  \to
  \dgMod(K_A^2)
$$
$$
  (M,N)
  \mapsto
  M\odot N
  :=
  M\otimes_{K_A}N,
$$
where $K_A$ acts on $M$ via $\varepsilon_2$ and on $N$ via $\varepsilon_1$.
The unit of this monoidal structure is $K_A$, regarded as a $K_A^2$-dg-module via the diagonal structure.

\sssec{}

We endow $\dgMod(K_A^2)$ with the projective model structure.
As already mentioned in \cite[\S 4.1.1]{toenvezzosi22}, this is compatible with the monoidal structure. 
In particular, the monoidal structure preserves cofibrant objects.

\sssec{}

Let $\cofdgMod(K_A^2)\subseteq \dgMod(K_A^2)$ denote the full subcategory spanned by those dg-modules which are cofibrant and strictly perfect over $A$, together with the monoidal unit $K_A$ (which is not cofibrant).

\sssec{}

Let $W_{\up{qi}}$ denote the set of quasi-isomorphisms in $\cofdgMod(K_A^2)$.
Put
$$
  \sB^+
  :=
  L_{W_{\up{qi}}}\bigt{\cofdgMod(K_A^2)}.
$$
This is a monoidal $A$-linear dg-category (\cite[Appendix A]{toenvezzosi22}). 
Its underlying $A$-linear dg-category is $\Coh(G)$.

\sssec{}

Consider the set $W_{\up{pe}}$ consisting of those morphism in $\cofdgMod(K_A^2)$ whose cofiber is a perfect $K_A^2$-dg-module.

\begin{lem}\label{lem: compatibility odot with Wqi and Wpe}
The monoidal structure on $\cofdgMod(K_A^2)$ is compatible with $W_{\up{pe}}$, i.e. $\odot$ sends $W_{\up{pe}}\otimes \id \cup \id \otimes W_{\up{pe}}$ into $W_{\up{pe}}$.
\end{lem}
\begin{proof}
This has been already observed in \cite{toenvezzosi22}. 
We include the details for the sake of completeness.
The claim is equivalent to the following: if $M, N\in \cofdgMod(K_A^2)$, then $M\odot N$ is a perfect $K_A^2$-dg-module as soon as one among $M$ and $N$ is. 
Suppose that $M$ is a perfect $K_A^2$-dg-module.
We may assume that $M=K_A^2$. 
In this case,
$$
  K_A^2\odot N 
  \simeq 
  K_A\otimes_A N
$$
is the external tensor product over $A$ of two perfect $K_A$-modules (since $K_A\simeq k$ is regular), whence perfect.

The case when $N$ is perfect as a $K_A^2$-dg-module is totally analogous.
\end{proof}

\sssec{}
The lemma above, joint with \cite[Appendix A]{toenvezzosi22}, implies that
$$
  \sB
  :=
  L_{W_{\up{pe}}}\bigt{\cofdgMod(K_A^2)}
$$
is a monoidal dg-category, whose underlying dg-category is equivalent to $\Sing(G)$.
Since $W_{\up{qi}}\subseteq W_{\up{pe}}$, the localization morphism $\sB^+\to \sB$ is monoidal.

\ssec{The \texorpdfstring{$\sB$}{B}-module \texorpdfstring{$\sT$}{T}}

Let $X \to S$ be a regular $S$-scheme of finite type.
As usual, we denote by $X_s := X \times_S s$ the special fiber. 
The goal of this section is to provide a strict model for the dg-category $\sT := \Sing(X_s)$, and equip it with the structure of left $\sB$-module.

\sssec{}

Consider an $A$-algebra $A\to R$ of finite type and assume that $R$ is regular.
Let
$$
  K_R
  :=
  R[\varepsilon],
$$
where $\varepsilon$ is a free variable in cohomological degree $-1$ such that $d(\varepsilon)=\pi_{\uK}$.
Notice that $K_R=K_A\otimes_AR$.

\sssec{}

Following \cite[\S 4.1.2]{toenvezzosi22}, we consider the $A$-linear dg-category $\cofdgMod(K_R)$ of cofibrant (with respect to the projective model structure) $K_R$-dg-modules with underlying perfect $R$-module.
By \cite[Lemma 2.33]{brtv18}, $\cofdgMod(K_R)$ is a strict model for $\Coh(R\otimes_Ak)$ in the sense that
$$
  \Coh(R\otimes_Ak)\simeq L_{W_{\up{qi}}}\bigt{\cofdgMod(K_R)},
$$
where $W_{\up{qi}}$ stands for the set of quasi-isomorphisms in $\cofdgMod(K_R)$.

\sssec{}

We consider the dg-functor
$$
  -\odot-:
  \cofdgMod(K_A^2)\otimes_A \cofdgMod(K_R)
  \to
  \cofdgMod(K_R)
$$
$$
  (M,E)
  \mapsto
  M\odot N
  :=
  M\otimes_{K_A}E,
$$
where $M$ has the $K_A$-module structure induced by $1\otimes \id:K_A\to K_A^2$ and $E$ that induced by $K_A\to K_R$.

\sssec{}\label{left B module structure}

By \cite[Proposition 4.1.5]{toenvezzosi22} this induces a left $\sB^+$-modules structure on 
$$
\Coh(k\otimes_AR)\simeq L_{W_{\up{qi}}}\cofdgMod(K_R)
$$
and a left $\sB$-module structure on 
$$
\Sing(k\otimes_AR)\simeq L_{W_{\up{pe}}}\cofdgMod(K_R).
$$

\sssec{}

For a regular $S$-scheme of finite type $X\to S$, we set
$$
  \sT^+
  :=
  \lim_{\Spec{R}\subseteq X} L_{W_{\up{qi}}}\cofdgMod(K_R)
  \in
  \Mod_{\sB^+}.
$$
Since the forgetful functor $\Mod_{\sB^+} \to \dgCat_A$ is a right adjoint, it commutes with limits.
If follows that the $A$-linear dg-category underlying $\sT^+$ is equivalent to $\Coh(X_s)$.
Similarly, we define
$$
  \sT
  :=
  \lim_{\Spec{R}\subseteq X} L_{W_{\up{pe}}}\cofdgMod(K_R)
  \in
  \Mod_{\sB}.
$$
This definition endows $\Sing(X_s)$ with a left $\sB$-module structure.

\begin{rmk}{\cite[Proposition 4.1.7]{toenvezzosi22}}
Since $\sB^+$ (resp. $\sB$) is Karoubi-generated by its unit object, $\sT^+$ (resp. $\sT$) is cotensored over $\sB^+$ (resp. $\sB$).
\end{rmk}

\sssec{}

One of the main contributions of {\TV} is the following \emph{K\"unneth formula for singularity categories}.

\begin{thm}[{\cite[Theorem 4.2.1]{toenvezzosi22}}]\label{Thm: Kunneth for Sing}\label{thm:Kunneth}
Let $Y$ and $Z$ be two schemes over $S$. 
Suppose that both $Y$ and $Z$ are regular $S$-schemes of finite type with smooth generic fibers.
Consider the dg-categories $\Sing(Y_s)$ and $\Sing(Z_s)$, both equipped with the convolution action of $\sB$. There is a canonical equivalence
$$
\ffF:
\Sing(Y_s)^\op
\otimes_{\sB}
\Sing(Z_s)
\xto{\;\; \simeq \;\;}
\Sing(Y \times_S Z)
$$
induced by the dg-functor
$$
  (Y_s\times_sZ_s \to Y\times_S Z)_*(\ul\Hom_{Y_s}(-,\ccO)\boxtimes_s -):
  \Coh(Y_s)^\op \otimes_A \Coh(Z_s)
  \to
  \Coh(Y\times_S Z).
$$

\end{thm}
\begin{rmk}
Notice that in \emph{loc. cit.} the authors suppose that both $Y$ and $Z$ are flat over $S$. 
This hypothesis is superfluous, provided that fiber products are taken the derived sense (as we do in this paper).
\end{rmk}

As a consequence, we obtain
\begin{cor}[{\cite[Proposition 4.3.1, Remark 4.3.2]{toenvezzosi22}}]\label{cor: T dualizable B module}
Let $p:X\to S$ a regular $S$-scheme of finite type with smooth geometric fiber.
Assume that the singular locus of $p_s:X_s\to s$ is proper.
Then $\sT$ is a saturated $\sB$-module.
\end{cor}

\sssec{}

More precisely, consider the $\sB^\env$-linear dg-functor
$$
  \ev_{\sT/\sB}:
  \sT \otimes_A \sT^\op
  \to
  \sB
$$
induced by 
$$
  \Coh(X_s)\otimes_A \Coh(X_s)^\op
  \to
  \Coh(G)
$$
$$
  (x,y)
  \mapsto
  (X_s\times_S X_s \to s\times_Ss)_*(y\boxtimes_S \ul\Hom(x,\ccO)),
$$
and the functor (here $\on{proj}:\Coh(X\times_SX)\to \Sing(X\times_SX)$ denotes the canonical dg-functor)
$$
  \coev_{\sT/\sB}
  :=
  \ffF^{-1}\circ \on{proj} \circ \delta_{X/S*}\circ p^*:
  \Perf(S)
  \to
  \sT^\op \otimes_{\sB}\sT.
$$
It is proven in \cite{beraldopippi22} that $\ev_{\sT/\sB}$ and $\coev_{\sT/\sB}$ provide a duality datum for $\sT$ over $\sB$.

\ssec{The monoidal dg-category \texorpdfstring{$\sC$}{C}}

In this section, we study the groupoid scheme
$$
  H
  :=
  S'\times_S S'
$$
and the associated monoidal dg-categories $\sC^+ :=\Coh(H)$ and $\sC := \Sing(H)$.

\sssec{}

Recall our convention that $S' \to S$ is the map of affine schemes corresponding to a totally ramified extension $A \subseteq A'$ of discrete valuation rings. By \cite[Chapter 1, Proposition 18]{serre79}, we can write $A' =A[x]/\langle E(x) \rangle$ for an Eisenstein polynomial $E(x) \in A[x]$.

\sssec{}

We now describe the groupoid structure on $H$ via the corresponding Hopf algebroid. Consider the commutative $A$-algebra
$$
  A''
  :=
  A[x_1,x_2]/\langle E(x_1),E(x_2) \rangle.
$$
This is a Hopf algebroid over $A$, where the source and target maps are given by the inclusions
$$
  A'
  \to 
  A'', 
  \hspace{0.5cm} 
  x 
  \mapsto 
  x_i 
  \hspace{0.2cm} 
  (i=1,2),
$$
the unit is given by the multiplication
$$
  A'' 
  \to
  A',
  \hspace{0.5cm} 
  x_1,x_2 
  \mapsto 
  x, 
$$
the composition corresponds to
$$
  \id \otimes 1 \otimes \id:
  A''=A'\otimes_AA' 
  \to
  A'\otimes_A A' \otimes_A A',
$$
and the antipode to
$$
  A''
  \to
  A'',
  \hspace{0.5cm}
  x_1
  \mapsto 
  x_2, \;\;
  x_2
  \mapsto 
  x_1.
$$

\sssec{}

Consider the $A$-linear dg-category $\dgMod(A'')$ of dg-modules over $A''$.
The Hopf algebroid structure on $A''$ induces a (non-symmetric) monoidal structure on $\dgMod(A'')$ given by
$$
  M\circledcirc N
  :=
  M\otimes_{A'}N,
$$
where $A'$ acts on $M$ (resp. $N$) via $A'\xto{1\otimes \id}A''$ (resp. $A'\xto{\id\otimes 1}A''$).
The $A''$-structure on $M\otimes_{A'} N$ is given by $\id \otimes 1 \otimes \id: A''=A'\otimes_AA' \to A'\otimes_A A' \otimes_A A'$.
This structure is clearly associative and unital, the unit object being $A'$.

\sssec{}

We can endow $\dgMod(A'')$ with the projective model structure.
Moreover, this is compatible with the monoidal structure above, i.e. it defines a monoidal model category in the sense of \cite{hovey99}. 
In fact, we just need that the monoidal tensor product preserves cofibrant objects.

\sssec{}

Let $\cofdgMod(A'')\subseteq \dgMod(A'')$ be the (full) subcategory of cofibrant objects which are strictly perfect over $A$ and by the unit object $A'$. 
Notice that the unit is not cofibrant, so that including it is a non-redundant request.

\sssec{}

Consider the set of quasi-isomorphisms $W_{\up{qi}}$ in $\cofdgMod(A'')$.
We define
$$
  \sC^+
  :=
  L_{W_{\up{qi}}}\bigt{\cofdgMod(A'')}.
$$
This is a monoidal $A$-linear dg-category (\cite[Appendix A]{toenvezzosi22}), whose underlying dg-category is $\Coh(H)$ (\cite[Lemma 2.33]{brtv18}).

\sssec{}

Let $W_{\up{pe}}$ be the set of morphisms in $\cofdgMod(A'')$ whose cone is a perfect $A''$-dg-module.

\begin{lem}\label{lem: compatibility circledcirc with Wqi and Wpe}
The monoidal structure on $\cofdgMod(A'')$ is compatible with $W_{\up{pe}}$, i.e. the monoidal tensor product $\circledcirc$ sends $W_{\up{pe}} \otimes \id \cup \id \otimes W_{\up{pe}} $ into $W_{\up{pe}}$.
\end{lem}
\begin{proof}
Let $u:M_1\to M_2 \in W_{\up{pe}}$ and $N \in \cofdgMod(A'')$. 
Since $\circledcirc$ is compatible with homotopy colimits, $u\circledcirc \id: M_1\circledcirc N \to M_2\circledcirc N \in W_{\up{pe}}$ if and only if $\coFib(u)\circledcirc N$ is a perfect $A''$-dg-module.
Since the category of perfect $A''$-dg-modules is Karoubi generated by $A''$, we need to show that $A''\circledcirc N$ is perfect for every $N$. This is obvious, since
$$
  A''\circledcirc N = A''\otimes_{A'}N 
$$
and $N$ is a perfect $A'$-dg-module as $A'$ is regular.
\end{proof}

\sssec{}

As a consequence of this lemma and of \cite[Appendix A]{toenvezzosi22}, the dg-localization
$$
  \sC
  :=
  L_{W_{\up{pe}}}\bigt{\cofdgMod(A'')}
$$
is monoidal. By \cite[Lemma 2.33]{brtv18}, it is equivalent to $\Sing(H)$.
Since $W_{\up{qi}}\subseteq W_{\up{pe}}$, the canonical map $\sC^+ \to \sC$ is monoidal.

\ssec{The \texorpdfstring{$\sC$}{C}-module \texorpdfstring{$\sU$}{U}}\label{ssec: the C module U}

Let $X\to S$ be a regular $S$-scheme of finite type. Let $X' := X \times_S S'$: this is a classical quasi-smooth scheme. In this section, we provide a strict model for the dg-category $\sU := \Sing(X')$ and endow it with the structure of left $\sC$-module.

\sssec{}

Let $A\to R$ be a regular $A$-algebra of finite type. Define
$$
  R'
  :=
  R[x]/\langle E(x) \rangle.
$$
Notice that $R'\simeq R\otimes_AA'$.

\sssec{}

We consider the $A$-linear dg-category $\cofdgMod(R')$ of cofibrant $R'$-dg-modules with an underlying perfect $R$-dg-module.
We have that
$$
  \Coh(R')\simeq L_{W_{\up{qi}}}\bigt{\cofdgMod(R')},
$$
with $W_{\up{qi}}$ the set of quasi-isomorphisms in $\cofdgMod(R')$.

\sssec{}

Consider the following dg-functor
$$
   \cofdgMod(A'') \otimes_A \cofdgMod(R')
  \to
  \cofdgMod(R')
$$
$$
  (N,E)
  \mapsto
  N\circledcirc E
  :=
  N\otimes_{A'}E,
$$
where $E$ has the $A'$-dg-module structure induced by $A'\to R'$ and $N$ that induced by $1 \otimes \id :A'\to A''$.

\begin{lem}
The above dg-functor induces a $\cofdgMod(A'')$-module structure on $\cofdgMod(R')$.
\end{lem}
\begin{proof}
We start by proving that there is a canonical isomorphism
$$
  (N_1\otimes_{A'} N_2)\otimes_{A'}E
  \simeq 
  N_1\otimes_{A'} (N_2\otimes_{A'}E).
$$
Notice that there is such a canonical isomorphism at the level of the underlying $A$-dg-modules.
Indeed both the left and the right hand sides are canonically isomorphic to the $A$-dg-module
$$
  N_1\otimes_{A'}N_2\otimes_{A'}E.
$$
We conclude by noticing that the canonical isomorphisms
$$
  (N_1\otimes_{A'} N_2)\otimes_{A'}E
  \simeq 
  N_1\otimes_{A'} N_2\otimes_{A'}E
  \simeq
  N_1\otimes_{A'} (N_2\otimes_{A'}E).
$$
are compatible with the $R'$-dg-modules structures, which in the middle is given by
$$
  (r\otimes x) \cdot (n_1\otimes n_2 \otimes e) 
  =
  (x\cdot n_1)\otimes n_2 \otimes (r \cdot e).
$$

Next, we show that there is a canonical isomorphism 
$$
  A'\circledcirc E
  \simeq 
  E
$$
This is obvious from the definition of $\circledcirc$ and from the $A''$-module structure on $A'$.
\end{proof}

\begin{lem}\label{lem: left C module structure}
There is a canonical left $\sC^+$-module structure on $\Coh(R')\simeq L_{W_{\up{qi}}}\bigt{\cofdgMod(R')}$ and a canonical left $\sC$-module structure on $\Sing(R')\simeq L_{W_{\up{pe}}}\bigt{\cofdgMod(R')}$.
\end{lem}
\begin{proof}
By abuse of notation, we will denote by $W_{\up{qi}}$ (resp. $W_{\up{pe}}$) both the set of quasi-isomorphisms (resp. the set of morphisms with perfect cone) in $\cofdgMod(A'')$ and $\cofdgMod(R')$.
By the previous lemma and \cite[Appendix A]{toenvezzosi22}, it suffices to show that the (pseudo) left $\cofdgMod(A'')$-module structure on $\cofdgMod(R')$ is compatible with $W_{\up{qi}}$ and $W_{\up{pe}}$.
In other words, we need to show that $\circledcirc$ sends $W_{\up{qi}}\otimes \id \cup \id \otimes W_{\up{qi}}$ into $W_{\up{qi}}$ and $W_{\up{pe}}\otimes \id \cup \id \otimes W_{\up{pe}}$ into $W_{\up{pe}}$.

The case of quasi-isomorphisms follows immediately from the observation that we are dealing with cofibrant objects (and the unit of $\cofdgMod(A'')$).

As for the case of $W_{\up{pe}}$, we may equivalently prove that $N\circledcirc E$ is a perfect $R'$-dg module as soon as one among $E\in \cofdgMod(R')$ and $N\in \cofdgMod(A'')$ is such.
This reduces to prove that $N\circledcirc R'$ and $A''\circledcirc E$ are perfect $R'$-dg-modules, which is obvious.
\end{proof}

\sssec{} 

For a regular $S$-scheme $X\to S$ of finite type, we define
$$
  \sU^+
  :=
  \lim_{\Spec{R}\subseteq X}L_{W_{\up{qi}}}\bigt{\cofdgMod(R')}
  \in 
  \Mod_{\sC^+}.
$$
The forgetful functor $\Mod_{\sC^+} \to \dgCat_A$ being a right adjoint, the underlying $A$-linear dg-category of $\sU^+$ is equivalent to $\Coh(X')$.

Similarly, we put
$$
  \sU
  :=
  \lim_{\Spec{R}\subseteq X}L_{W_{\up{pe}}}\bigt{\cofdgMod(R')}
  \in 
  \Mod_{\sC}.
$$
For the same reason as above, the $A$-dg-category underlying the right $\sC$-module $\sU$ is equivalent to $\Sing(X')$.

\begin{lem}
The monoidal dg-category $\sC$ is Karoubi-generated by its unit.
In particular, the left $\sC$-module $\sU$ is cotensored over $\sC$.
\end{lem}
\begin{proof}
By Remark \ref{rmk: monoidal catogory generated by its unit object = every module is cotensored}, the second statement follows from the first one.
Therefore, we just need to show that $A'=1_\sC$ generates $\sC$.

By Theorem \ref{Thm: Kunneth for Sing} we have that $\sC$ is Karoubi-generated by the image of the dg-functor
$$
  \ffF:
  \Sing(s')^\op\otimes_A \Sing(s')
  \to
  \Sing(S'\times_S S'),
$$
whence - as $\Sing(s')$ is Karoubi-generated by $k$ - by the $A''$-module $k=\ffF(k,k)$.
Since 
$$
  k
  \simeq 
  [
  A'\xto{\pi_{\uL}}A'
  ]
$$
the claim follows.
\end{proof}

\begin{rmk}
Notice that $\sC^+$ is not Karoubi-generated by its unit.
Nevertheless, the left $\sC^+$-module $\sU^+$ is cotensored over $\sC^+$.
This is a consequence of the fact that the left $\sC^+$-module structure on $\sU^+$ corresponds, under Grothendieck duality, to a right $\sC^+$-module structure on $\sU^{\op}$. 
Therefore, one checks directly that the condition of Definition \ref{defn: cotensored A-module} holds.
Of course, this implies that the $\sC$-module $\sU$ is cotensored over $\sC$.
\end{rmk}

\ssec{Connecting the two sides: the \texorpdfstring{$(\sB,\sC)$}{(B,C)}-bimodule \texorpdfstring{$\sE$}{E}}

In the previous sections, we introduced two pairs of algebra/module categories: $(\sB,\sT)$ and $(\sC, \sU)$. We now relate them.

\sssec{}

We will relate the two groupoids $G = s \times_S s$ and $H = S' \times_S S'$ by means of the scheme
$$
s' := s\times_S S'.
$$
As above, this is a classical quasi-smooth scheme: explicitly, $s' = \Spec{k[x]/x^e}$ where $e$ equals the degree of the Eisenstein polynomial defining $S'$. In particular, $s'$ is a thickening of $s$.

\sssec{}

We now define
$$
\sE ^+ := \Coh(s')
$$
and
$$
\sE := \Sing(s').
$$
It is clear that $\sE^+$ is a $(\sB^+, \sC^+)$-bimodule. Below, we make this statement explicit using the strict models discussed in the previous sections. We will then show that the bimodule structure descends to the singularity categories.

\sssec{}

Let $\dgMod(k')$ be the $A$-linear dg-category of $k'=k\otimes_AA'$-dg-modules, endowed with the projective model structure.
Denote $\cofdgMod(k')\subseteq \dgMod(k')$ the full subcategory of cofibrant objects which are strictly perfect over $A$.
This is a strict model for $\Coh(s')$, that is,
$$
  \Coh(s')
  \simeq
  L_{W_{\up{qi}}}\bigt{\cofdgMod(k')}
$$
where $W_{\up{qi}}$ is the set of quasi-isomorphisms (\cite[Lemma 2.33]{brtv18}).

\begin{lem}\label{lem: E^+ (B^+,C^+)-bimod}
The dg-category $\cofdgMod(k')$ is naturally a  $\bigt{\cofdgMod(K_A^2),\cofdgMod(A'')}$-bimodule. 
Moreover, $\sE^+$ inherits a $(\sB^+, \sC^+)$-bimodule structure.
\end{lem}
\begin{proof}
Let $M \in \cofdgMod(K_A^2)$, $E\in \cofdgMod(K_{A'})$ and $N \in \cofdgMod(A'')$. 
We need to show that there is a canonical isomorphism
$$
  (M\odot E)\circledcirc N
  \simeq
  M\odot (E\circledcirc N).
$$
Unraveling the definitions, we see that both sides of the displayed equation are canonically isomorphic to the $A$-dg module
$$
  M\otimes_{K_A}E\otimes_{A'}N,
$$
endowed with the structure of a dg-module over $K_{A'}\simeq K_A\otimes_AA'\simeq A[\varepsilon,x]$ via
$$
  \varepsilon \cdot (m\otimes e \otimes n)
  =
  (\varepsilon_1 \cdot m)\otimes e \otimes n,
$$
$$
  x \cdot (m\otimes e \otimes n)
  =
  m\otimes e \otimes (x_2 \cdot n).
$$
The final statement follows from \cite[Appendix A]{toenvezzosi22} and from the observation that
$$
  \cofdgMod(K_A^2)\otimes_A \cofdgMod(k')_A \otimes \cofdgMod(A'') 
  \to 
  \cofdgMod(k')
$$
sends $(W_{\up{qi}}\otimes \id \otimes \id) \cup (\id \otimes W_{\up{qi}} \otimes \id) \cup (\id \otimes \id \otimes W_{\up{qi}})$ into $W_{\up{qi}}$, which is true since we are working with cofibrant dg-modules and units.
\end{proof}

\begin{lem}
The above actions descend to make $\sE$ into a $(\sB,\sC)$-bimodule.
\end{lem}
\begin{proof}
It suffices to show that the bimodule structure
$$
  \cofdgMod(K_A^2)\otimes_A \cofdgMod(k') \otimes_A \cofdgMod(A'') 
  \to 
  \cofdgMod(k')
$$
sends $(W_{\up{pe}}\otimes \id \otimes \id) \cup (\id \otimes W_{\up{pe}} \otimes \id) \cup (\id \otimes \id \otimes W_{\up{pe}})$ into $W_{\up{pe}}$. Equivalently, that for $M \in \cofdgMod(K_A^2)$, $E \in \cofdgMod(k')$ and $N \in \cofdgMod(A'')$
$$
  M \odot E \circledcirc N
$$
is a perfect $k'$-dg-module as soon as one among $M,E$ and $N$ is so.

This is the case if $M$ or $N$ is perfect by by Lemma \ref{lem: E^+ (B^+,C^+)-bimod}, Section \ref{left B module structure} and Lemma \ref{lem: left C module structure}.
It remains to treat the situation when $E$ is perfect. 
Without loss of generality, we can assume that $E=k'$.
In this case,
$$
M\odot k'\circledcirc N
\simeq 
M\otimes_{K_A}k'\otimes_{A'}N
\simeq
M\otimes_A N.
$$
It is easier to see this using the geometric language.
The functor
$$
  -\otimes_A -:
  \cofdgMod(K_A^2)\otimes_A \cofdgMod(A'')
  \to 
  \cofdgMod(k')
$$
corresponds to the $\oo$-functor
$$
  \pr{14*}(\pr{12}^*\otimes \pr{34}^*):
  \Coh(G)\otimes_A \Coh(H)
  \to
  \Coh(s'),
$$
where
\begin{equation*} 
  \begin{tikzpicture}[scale=1.5]
    \node (Cu) at (2,1) {$s\x{S}s\x{S}S'\x{S}S'$ };
    \node (Ld) at (1,0) {$s\x{S}s$};
    \node (Ru) at (5,1) {$ s\x{S}S'$};
    \node (Rd) at (3,0) {$S'\x{S}S'.$};
    \draw[->] (Cu) to node[left] { $\pr{12}\;\;$ } (Ld);
    \draw[->] (Cu) to node[right] { $\;\; \pr{34}$ } (Rd);
    \draw[->] (Cu) to node[above] {$\pr{14}$} (Ru);
  \end{tikzpicture}
\end{equation*}
We need to show that this factors through $\Perf(s')\subseteq \Coh(s')$.

By decomposing 
$$
  \pr{14}\sim \pr{13}\circ \pr{124}:
  s\x{S}s\x{S}S'\x{S}S'
  \to 
  s\x{S}s\x{S}S' 
  \to 
  s\x{S}S'
$$ 
and 
$$
  \pr{12}\sim \pr{12}\circ \pr{124}:
  s\x{S}s\x{S}S'\x{S}S'
  \to 
  s\x{S}s\x{S}S' 
  \to 
  s\x{S}s,
$$ 
by the projection formula we get that

\begin{align*}
    \pr{14*}(\pr{12}^*M\otimes \pr{34}^*N) & \simeq \pr{13*}\pr{124*}(\pr{124}^*\pr{12}^*M\otimes \pr{34}^*N) \\
     & \simeq \pr{13*}(\pr{12}^*M\otimes \pr{124*}\pr{34}^*N).
\end{align*}

Applying base-change for the cartesian square
\begin{equation*} 
  \begin{tikzpicture}[scale=1.5]
    \node (Lu) at (0,1) {$s\x{S}s\x{S}S'\x{S}S'$ };
    \node (Ld) at (0,0) {$s\x{S}s\x{S}S'$};
    \node (Ru) at (3,1) {$S'\x{S}S'$};
    \node (Rd) at (3,0) {$ S'$};
    \draw[->] (Lu) to node[right] { $\pr{124}$ } (Ld);
    \draw[->] (Ru) to node[right] { $\pr{2}$ } (Rd);
    \draw[->] (Lu) to node[above]{$\pr{34}$} (Ru);
    \draw[->] (Ld) to node[above]{$\pr{3}$} (Rd);
  \end{tikzpicture}
\end{equation*}
we can continue the chain of equivalences:
$$
  \pr{13*}(\pr{12}^*M\otimes \pr{124*}\pr{34}^*N)
  \simeq 
  \pr{13*}(\pr{12}^*M\otimes \pr{3}^*\pr{2*}N).
$$
Since $S'$ is regular, it suffices to show that  $\pr{13*}\pr{12}^*M \in \Perf(s')$ for any $M\in \Coh(G)$. This is an immediate consequence of base-change applied to the fiber square
\begin{equation*} 
  \begin{tikzpicture}[scale=1.5]
    \node (Lu) at (0,1) {$s\x{S}s\x{S}S'$ };
    \node (Ld) at (0,0) {$s\x{S}s$};
    \node (Ru) at (3,1) {$s\x{S}S'$};
    \node (Rd) at (3,0) {$s$};
    \draw[->] (Lu) to node[right] { $\pr{12}$ } (Ld);
    \draw[->] (Ru) to node[right] { $\pr{1}$ } (Rd);
    \draw[->] (Lu) to node[above]{$\pr{13}$} (Ru);
    \draw[->] (Ld) to node[above]{$\pr{1}$} (Rd);
  \end{tikzpicture}
\end{equation*}
and of the regularity of $s$.
\end{proof}

\sec{The fundamental Morita equivalence} \label{sec:morita}

In the previous section, we introduced the monoidal dg-categories $\sB$ and $\sC$, together with a $(\sB,\sC)$-bimodule $\sE$. This structure yields an adjunction
\begin{equation} \label{adj:Bmod vs Cmod}
\begin{tikzpicture}[scale=1.5]
\node (a) at (0,1) {$\Mod_{\sB}$};
\node (b) at (2.5,1) {$\Mod_{\sC}$};
\path[->,font=\scriptsize,>=angle 90]
([yshift= 1.5pt]a.east) edge node[above] {$\sE^\op \otimes_{\sB} -$ } ([yshift= 1.5pt]b.west);
\path[->,font=\scriptsize,>=angle 90]
([yshift= -1.5pt]b.west) edge node[below] {$\Hom_{\sC} (\sE^{\op}, -)$ } ([yshift= -1.5pt]a.east);
\end{tikzpicture}
\end{equation}
of $\infty$-categories. 
Assume from now on that $S' \to S$ is not an isomorphism.
The first goal of this section is to prove that this adjunction is a pair of mutually inverse equivalences. In other words, the algebras $\sB$ and $\sC$ are Morita equivalent by means of $\sE$.
The second goal of this section is to prove that the Hochschild homologies $\HH_*(\sB/A)$ and $\HH_*(\sC/A)$ are naturally equivalent as dg-categories.

\ssec{The algebra of \texorpdfstring{$\sB$}{B}-linear endomorphisms of \texorpdfstring{$\sE$}{E}}

Observe that Theorem \ref{Thm: Kunneth for Sing}, applied to the case $Y=Z=S'$, yields an equivalence
\begin{equation}\label{eq: weak End_B(E)=C}
  \Sing(s')^{\op} \otimes_{\sB}\Sing(s') 
  \simeq
  \Sing(S''),
\end{equation}
that is, an equivalence $\sE^{\op} \otimes_{\sB} \sE \simeq \sC$. In this section, we show that $\sE^{\op} \otimes_{\sB} \sE$ is a monoidal ($A$-linear) dg-category and that this equivalence is monoidal.

\sssec{}

We start by considering the $A$-linear dg-functor
$$
  \hB:
  \cofdgMod(K_{A'})\otimes_A \cofdgMod(K_{A'})^{\op}
  \to
  \cofdgMod(K_A^2)
$$
$$
  (M,N)
  \mapsto
  M\otimes_{A'}N^\vee.
$$
Here, we use the notation
$$
  N^\vee
  :=
  \ul \Hom_{K_{A'}}(N,K_{A'}).
$$
Clearly, the $K_A^2$-dg-module structure on $M\otimes_{A'} N^\vee$ is given by
$$
  \varepsilon_1 
  \cdot
  (m\otimes n)
  =
  \varepsilon \cdot m \otimes n,
$$
$$
  \varepsilon_2 
  \cdot
  (m\otimes n)
  =
  m\otimes \varepsilon \cdot n.
$$
Notice that this dg-functor is well-defined. 
Indeed, since $K_{A'}$ is cofibrant over $A'$, the tensor product  $M\otimes_{A'}N^\vee$ is cofibrant over $K_{A'}^2$. 
As $A'$ is cofibrant over $A$, we deduce $K_{A'}^2$ is cofibrant over $K_A^2$, which implies that $M\otimes_{A'} N^\vee$ is cofibrant over $K_A^2$.

\sssec{}

The dg-functor above is a strict model for the evaluation of $\sE$ over $\sB$.
The following observation is unsurprising.
\begin{lem}\label{lem: spadesuit pseudo associative}
The dg-functor
$$
\hB:
  \cofdgMod(K_{A'})\otimes_A \cofdgMod(K_{A'})^{\op}
  \to
  \cofdgMod(K_A^2)
$$
is $\cofdgMod(K_A^2)$-bilinear.
\end{lem}

\begin{proof}
Let $F_1,F_2 \in \cofdgMod(K_{A'})$ and $M \in \cofdgMod(K_A^2)$. 
Then there is a canononical isomorphism
$$
  \hB(F_1,F_2)\odot M 
  =
  (F_1\otimes_{A'}F_2^\vee)\otimes_{K_A}M
  \simeq
  F_1 \otimes_{A'}(F_2^\vee \otimes_{K_A}M)
  =
  \hB(F_1,F_2\odot M).
$$
Similarly, there is a canonical isomorphism
$$
  M\odot \hB(F_1,F_2) \simeq \hB(M\odot F_1,F_2).
$$
\end{proof}

\sssec{}

We introduce a non-unital monoidal structure on $ \cofdgMod(K_{A'})^{\op}\otimes_A \cofdgMod(K_{A'})$.
Consider the dg-functor
$$
  -\spadesuit -:
  \bigt{
  \cofdgMod(K_{A'})^{\op}\otimes_A \cofdgMod(K_{A'})
  }^{\otimes_A2}
  \to
  \cofdgMod(K_{A'})^{\op}\otimes_A \cofdgMod(K_{A'})
$$
obtained by composing the canonical isomorphism
$$
  \bigt{
  \cofdgMod(K_{A'})^{\op}\otimes_A \cofdgMod(K_{A'})
  }
  \otimes_A
  \bigt{
  \cofdgMod(K_{A'})^{\op}\otimes_A \cofdgMod(K_{A'})
  }
  \simeq 
$$
$$
  \cofdgMod(K_{A'})^{\op}\otimes_A
  \bigt{
  \cofdgMod(K_{A'})
  \otimes_A
  \cofdgMod(K_{A'})^{\op}
  }
  \otimes_A \cofdgMod(K_{A'})
$$
with
$$
  \cofdgMod(K_{A'})^{\op}\otimes_A
  \Bigt{
  \cofdgMod(K_{A'})
  \otimes_A
  \cofdgMod(K_{A'})^{\op}
  \xto{\hB}
  \cofdgMod(K_A^2)
  }
  \otimes_A \cofdgMod(K_{A'})
$$
and
$$
  \id \otimes (- \odot -):
  \cofdgMod(K_{A'})^{\op}\otimes_A \cofdgMod(K_A^2)\otimes_A \cofdgMod(K_{A'})
  \to
  \cofdgMod(K_{A'})^{\op}\otimes_A \cofdgMod(K_{A'}).
$$

\begin{lem}\label{lem: spadesuit is associative}
$\bigt{ \cofdgMod(K_{A'})^{\op}\otimes_A \cofdgMod(K_{A'}), -\spadesuit -}$ is a non-unital monoidal dg-category.
\end{lem}
\begin{proof}
We need to show that the square
\begin{equation*} 
  \begin{tikzpicture}[scale=1.5]
    \node (Lu) at (0,1) {$ \bigt{\cofdgMod(K_{A'})^{\op}\otimes_A \cofdgMod(K_{A'})}^{\otimes_A 3}$  };
    \node (Ld) at (0,0) {$\bigt{\cofdgMod(K_{A'})^{\op}\otimes_A \cofdgMod(K_{A'})}^{\otimes_A 2}$};
    \node (Ru) at (6,1) {$\bigt{\cofdgMod(K_{A'})^{\op}\otimes_A \cofdgMod(K_{A'})}^{\otimes_A 2}$};
    \node (Rd) at (6,0) {$\cofdgMod(K_{A'})^{\op}\otimes_A \cofdgMod(K_{A'})$};
    \draw[->] (Lu) to node[right] { $\id \otimes (- \spadesuit -)$ } (Ld);
    \draw[->] (Ru) to node[right] { $-\spadesuit -$ } (Rd);
    \draw[->] (Lu) to node[above]{$(-\spadesuit -)\otimes \id$} (Ru);
    \draw[->] (Ld) to node[above]{$-\spadesuit -$} (Rd);
  \end{tikzpicture}
\end{equation*}
commutes up to canonical isomorphism.
Let $M_i \in \cofdgMod(K_{A'})^\op$, $N_i \in \cofdgMod(K_{A'})$, $i=1,2,3$.
Then we have that
\begin{align*}
    (M_1,N_1)\spadesuit \bigt{(M_2,N_2)\spadesuit (M_3,N_3)} & = (M_1,N_1)\spadesuit \bigt{M_2, \hB(N_2,M_3)\odot N_3}  \\
                                                             & = \Bigt{M_1, \hB(N_1,M_2)\odot \bigt{\hB(N_2,M_3)\odot N_3}}\\
                                                             & \simeq \Bigt{M_1, \bigt{\hB(N_1,M_2)\odot \hB(N_2,M_3)}\odot N_3} && \odot \text{ is assoc.}\\
                                                             & \simeq \Bigt{M_1, \hB\bigt{\hB(N_1,M_2)\odot N_2, M_3}\odot N_3}  && \text{ Lemma \ref{lem: spadesuit pseudo associative}}\\
                                                             & = \bigt{M_1, \hB(N_1,M_2)\odot N_2}\spadesuit (M_3,N_3) \\
                                                             & = \bigt{(M_1,N_1)\spadesuit (M_2,N_2)} \spadesuit (M_3,N_3).
\end{align*}
\end{proof}

\begin{cor}\label{cor: compatibility spadesuit with Wqi and Wpe}
The dg-categories $\sE^{+,\op}\otimes_A \sE^+$ and $\sE^{\op}\otimes_A \sE$ have canonical structures of non-unital monoidal dg-categories.
\end{cor}
\begin{proof}
By the lemma above and \cite[Appendix A]{toenvezzosi22}, it suffices to show that the operation $\spadesuit$ is compatible with $W_{\up{qi}}$ and $W_{\up{pe}}$. 
This is clear because $\spadesuit$ is defined as a composition of dg-functors with these properties: for the evaluation $\hB$ see \cite{beraldopippi22} and for $\odot$ see Lemma \ref{lem: compatibility odot with Wqi and Wpe}.
\end{proof}

\begin{cor}
The dg-categories $\sE^{+,\op}\otimes_{\sB^+} \sE^+$ and $\sE^{\op}\otimes_{\sB} \sE$ have canonical structures of non-unital monoidal dg-categories. 
\end{cor}

\sssec{}

We now investigate the behaviour of the equivalence \eqref{eq: weak End_B(E)=C} with respect to the monoidal structures.
Consider the dg-functor
$$
  \ffF^s:
  \cofdgMod(K_{A'})^{\op}\otimes_A \cofdgMod(K_{A'})
  \to 
  \cofdgMod(A'')
$$
$$
  (M,N)
  \mapsto
  M^\vee \otimes_{K_A}N,
$$
where the $A''$-dg-module structure is the obvious one.
Notice that this dg-functor is well-defined and that it provides a strict model for $\ffF$.
\begin{prop}
The dg-functor above is monoidal: there is a functorial equivalence
$$
  \ffF^s(M_1,N_1)\circledcirc \ffF^s(M_2,N_2)
  \simeq
  \ffF^s\bigt{(M_1,N_2)\spadesuit (M_2,N_2)}.
$$
\end{prop}
\begin{proof}
Unraveling the definitions, we have that 
\begin{align*}
    \ffF^s(M_1,N_1)\circledcirc \ffF^s(M_2,N_2) & = (M_1^\vee \otimes_{K_A}N_1)\otimes_{A'}(M_2^\vee \otimes_{K_A}N_2) \\
                                                & \simeq M_1^\vee \otimes_{K_A} (N_1\otimes_{A'} M_2^\vee) \otimes_{K_A} N_2 \\
                                                & \simeq \bigt{M_1^\vee \otimes_{K_A} (N_1\otimes_{A'} M_2^\vee)}\otimes_{K_A} N_2 \\
                                                & \simeq M_1^\vee \otimes_{K_A} \bigt{(N_1\otimes_{A'} M_2^\vee) \otimes_{K_A} N_2} \\
                                                & = M_1^\vee \otimes_{K_A}\bigt{\hB(N_1,M_2)\odot N_2} \\
                                                & = \ffF^s\bigt{(M_1,N_1)\spadesuit (M_2,N_2)}.
\end{align*}
\end{proof}

To summarize, we have obtained the following result, which is half of the Morita equivalence we are after.

\begin{cor}\label{cor: End_B(E)=C}
The dg-functor
$$
  \ffF:
   \sE^{\op}\otimes_{\sB}\sE 
  \to
  \sC
$$
is a monoidal equivalence.
\end{cor}

\ssec{The algebra of \texorpdfstring{$\sC$}{C}-linear endomorphisms of \texorpdfstring{$\sE$}{E}}

Next, we tackle the second half of the Morita equivalence: we will prove that the dg-category $\sE \otimes_{\sC} \sE^{\op}$ is monoidal and monoidally equivalent to $\sB$.

\sssec{}

Consider the $A$-linear dg-functor
$$
  \hC:
  \cofdgMod(K_{A'})^{\op}\otimes_{A} \cofdgMod(K_{A'})
  \to
  \cofdgMod(A'')
$$
$$
  (M,N)
  \mapsto
  M^\vee \otimes_{K_A}N,
 $$
for $M,N \in \cofdgMod(K_{A'})$.
As above, we have adopted the notation
$$
  M^\vee
  =
  \ul \Hom_{K_{A'}}(M,K_{A'}).
$$
Notice that $M^\vee \otimes_{K_A}N$ is naturally a $K_{A'}\otimes_{K_A}K_{A'}$-dg module. 
Therefore, it is naturally an $A''$-dg module via the morphism $A''\to K_{A'}\otimes_{K_A}K_{A'}$.
Notice that this dg-functor is well-defined, i.e. it preserves cofibrant dg-modules and unit objects.
Analogously to the previous section, we have the following result. 

\begin{lem}
The dg-functor
$$
  \hC:
  \cofdgMod(K_{A'})^{\op}\otimes_{A} \cofdgMod(K_{A'})
  \to
  \cofdgMod(A'')
$$
is $\cofdgMod(A'')$-bilinear.
\end{lem}
\begin{proof}
Let $F_1,F_2 \in \cofdgMod(K_{A'})$ and $M\in \cofdgMod(A'')$.
We have that 
$$
    \hC(F_1,F_2)\circledcirc M 
    = 
    (F_1^\vee \otimes_{K_A}F_2)\otimes_{A'}M
    \simeq 
    F_1^\vee \otimes_{K_A}(F_2 \otimes_{A'}M)
    = 
    \hC(F_1,F_2\circledcirc M).
$$
Similarly, one sees that 
$$
   M\circledcirc \hC(F_1,F_2)\simeq \hC(M\circledcirc F_1,F_2).
$$
\end{proof}

\sssec{}

Similarly to the previous section, we introduce a non-unital pseudo monoidal operation on $\cofdgMod(K_{A'})\otimes_A \cofdgMod(K_{A'})^{\op}$.
Consider the $A$-linear dg-functor
$$
  - \clubsuit -:
  \bigt{
  \cofdgMod(K_{A'})\otimes_A \cofdgMod(K_{A'})^{\op}
  }^{\otimes_A2}
  \to
  \cofdgMod(K_{A'})\otimes_A \cofdgMod(K_{A'})^{\op}
$$
obtained by composing the canonical isomorphism
$$
  \bigt{
  \cofdgMod(K_{A'})\otimes_A \cofdgMod(K_{A'})^{\op}
  }
  \otimes_A
  \bigt{
  \cofdgMod(K_{A'})\otimes_A \cofdgMod(K_{A'})^{\op}
  }
  \simeq 
$$
$$
  \cofdgMod(K_{A'})\otimes_A 
  \bigt{
  \cofdgMod(K_{A'})^{\op}
  \otimes_A
  \cofdgMod(K_{A'})
  }
  \otimes_A \cofdgMod(K_{A'})^{\op}
$$
with
$$
  \cofdgMod(K_{A'})\otimes_A 
  \Bigt{
  \cofdgMod(K_{A'})^{\op}
  \otimes_A
  \cofdgMod(K_{A'}) 
  \xto{\hC}
  \cofdgMod(A'')
  }
  \otimes_A \cofdgMod(K_{A'})^{\op}
$$
and
$$
  \id \otimes (- \circledcirc -):
  \cofdgMod(K_{A'})\otimes_A \cofdgMod(A'')\otimes_A \cofdgMod(K_{A'})^{\op}
  \to
  \cofdgMod(K_{A'})\otimes_A \cofdgMod(K_{A'})^{\op}.
$$

\begin{lem}\label{lem: clubsuit is associative}
$\bigt{ \cofdgMod(K_{A'})\otimes_A \cofdgMod(K_{A'})^{\op}, -\clubsuit -}$ is a non-unital monoidal dg-category.
\end{lem}
\begin{proof}
This is analogous to the proof of Lemma \ref{lem: spadesuit is associative}.
\end{proof}

\begin{cor}
The dg-categories $\sE^+ \otimes_A \sE^{+,\op}$ and $\sE \otimes_A \sE^{\op}$ have canonical structures of non-unital monoidal dg-categories induced by $\clubsuit$.
\end{cor}
\begin{proof}
As for Corollary \ref{cor: compatibility spadesuit with Wqi and Wpe}, it suffices to show that the operation $\clubsuit$ is compatible with $W_{\up{qi}}$ and $W_{\up{pe}}$. 
As observed for $\spadesuit$, this is clear because $\clubsuit$ is defined as a composition of dg-functors with these properties: for $\hC$, this is analogous to the case considered in \cite{beraldopippi22} and, for $\circledcirc$, it follows from Lemma \ref{lem: compatibility circledcirc with Wqi and Wpe}.
\end{proof}

\begin{cor}
The dg-categories $\sE^{+}\otimes_{\sC^+} \sE^{+,\op}$ and $\sE \otimes_{\sC} \sE^{\op}$ have canonical structures of non-unital monoidal dg-categories.
\end{cor}

\sssec{}

Next, we explore the relationship between the dg-categories $\sE \otimes_{\sC}\sE^{\op}$ and $\sB$, viewed as non-unital monoidal dg-categories. Consider the dg-functor
$$
  \ffG^s:
  \cofdgMod(K_{A'})\otimes_A \cofdgMod(K_{A'})^{\op}
  \to 
  \cofdgMod(K_A^2)
$$
$$
  (M,N)
  \mapsto
  M \otimes_{A'}N^\vee,
$$
where the $K_A^2$-dg-module structure is the obvious one.
Notice that this dg-functor is well-defined.

\begin{prop}
The dg-functor above is monoidal, i.e. there is a canonical equivalence
$$
  \ffG^s(M_1,N_1)\odot \ffG^s(M_2,N_2)
  \simeq
  \ffG^s\bigt{(M_1,N_2)\clubsuit (M_2,N_2)}.
$$
\end{prop}
\begin{proof}
Unraveling the definitions, we have that 
\begin{align*}
    \ffG^s(M_1,N_1)\odot \ffG^s(M_2,N_2) & = (M_1 \otimes_{A'}N_1^\vee)\otimes_{K_A}(M_2 \otimes_{A'}N_2^\vee) \\
                                         & \simeq M_1 \otimes_{A'} (N_1^\vee\otimes_{K_A} M_2) \otimes_{A'} N_2^\vee \\
                                         & \simeq \bigt{M_1 \otimes_{A'} (N_1^\vee\otimes_{K_A}M_2)}\otimes_{A'} N_2^\vee \\
                                         & \simeq M_1 \otimes_{A'} \bigt{(N_1^\vee\otimes_{K_A} M_2) \otimes_{A'} N_2^\vee} \\
                                         & = M_1 \otimes_{A'}\bigt{\hC(N_1,M_2)\circledcirc N_2^\vee} \\
                                         & = \ffG^s\bigt{(M_1,N_1)\clubsuit (M_2,N_2)}.
\end{align*}
\end{proof}

\sssec{}

We can now prove the following theorem, companion to Corollary \ref{cor: End_B(E)=C}.

\begin{thm}\label{thm: End_C(E)=B}
The dg-functor
$$ 
 \ffG: 
 \sE\otimes_{\sC}\sE^{\op} 
  \to
  \sB
$$
induced by $\ffG^s$ is a monoidal equivalence of dg-categories.

\end{thm}
\begin{proof}
We already know from the above proposition that $\ffG$ is monoidal. Since $\sB$ is unital, by \cite[Remark 2.1.3.8]{lurieha}, it suffices to prove that $\ffG$ is an equivalence of plain dg-categories.
For the computations that we will perform below, it is convenient to consult the following diagram:
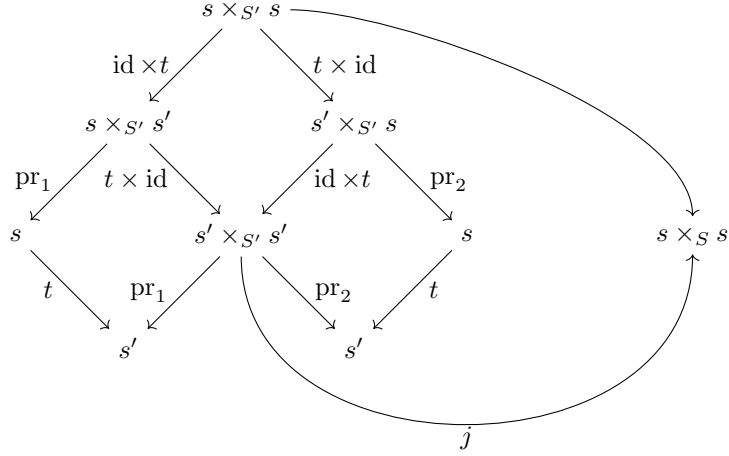
\begin{equation*}
  \begin{tikzpicture}[scale=1.5]
    \node (LLc) at (0,2) {$s$};
    \node (Lu) at (1,3) {$s\times_{S'}s'$};
    \node (Ld) at (1,1) {$s'$}; 
    \node (Cuu) at (2,4) {$s\x{S'}s$};
    \node (Cc) at (2,2) {$s'\x{S'}s'$};
    \node (Ru) at (3,3) {$s'\times_{S'}s$};
    \node (Rd) at (3,1) {$s'$};
    \node (RRc) at (4,2) {$s$};
    \node (RRRc) at (6,2) {$s\times_S s$};
    \draw[ ->] (Cuu) to node[left] {$\id \times t\;$} (Lu);
    \draw[ ->] (Cuu) to node [right]  {$\;t \times \id$} (Ru);
    \draw[ ->] (Cuu.east) .. controls +(right:9mm) and +(up:9mm) .. (RRRc.north);
    \draw[ ->] (Lu) to node [left] {$\pr{1}\;$} (LLc);
    \draw[ ->] (Lu) to node [left] {$t\times \id \;$} (Cc);
    \draw[ ->] (Ru) to node [right] {$\; \pr{2}$} (RRc);
    \draw[ ->] (Ru) to node[right] {$\; \id \times t$} (Cc);
    \draw[ ->] (LLc) to node[left] {$t\;$} (Ld);
    \draw[ ->] (Cc) to node[left] {$\pr{1}\;$} (Ld);
    \draw[ ->] (Cc) to node[right] {$\; \pr{2}$} (Rd);
    \node at (4,0.2) {$j$};
    \draw[ ->] (Cc.south) .. controls +(down:20mm) and +(down:20mm) .. (RRRc.south);
    \draw[->] (RRc) to node[right] {$\;t$} (Rd);
    
  \end{tikzpicture}
\end{equation*}
Observe that the dg-functor
$$
\ffG:
\Sing(s')\otimes_{\sC}\Sing(s')^{\op}
\to
\Sing(s\times_S s)
$$
is induced by
$$
j_*\circ (-\boxtimes_{S'}(-)^\vee):
\Sing(s')\otimes_A \Sing(s')^{\op}
\to
\Sing(s\times_Ss).
$$
We first prove that $\sB$ is Karoubi-generated by the essential image of $\ffG$. To see this, notice that $\Sing(s')$ is Karoubi-generated by $t_*k$ and a routine computation yields 
$$
  \ffG(t_*k,t_*k)\simeq 1_{\sB}\oplus 1_{\sB}[1].
$$
Then the claim follows from the fact that $\sB$ is Karoubi-generated by its unit.

It remains to show that $\ffG$ is fully-faithful. To begin, we observe that
$$
  \Hom_{\sE}(t_*k,t_*k)
  \simeq
  k[u,u^{-1}]\oplus k[u,u^{-1}][1].
$$
Also notice that 
$$
\Hom_{\sC^+}(\Delta_{S'},\Delta_{S'})
= 
\HH^*(A'/A).
$$
In concrete terms, $\Hom_{\sC^+}(\Delta_{S'},\Delta_{S'})$ is equivalent to the cochain complex
$$
  0\to A' \xto{0} A' \xto{E'(\pi_{\uL})} A' \xto{0} A' \xto{E'(\pi_{\uL})} \dots
$$
with first nontrivial entry in degree $0$. It follows that 
$$
  \Hom_{\sE}(t_*k,t_*k) 
  \simeq
  \coFib \bigt{\Hom_{\sC}(\Delta_{S'},\Delta_{S'})\xto{\pi_{\uL}}\Hom_{\sC}(\Delta_{S'},\Delta_{S'})}.
$$
Let $\End_{\sC}(\Delta_{S'}):= \Hom_{\sC}(\Delta_{S'},\Delta_{S'})$.
We can then compute
\begin{align*}
  \Hom_{\sE \otimes_{\sC}\sE^{\op}}\bigt{(t_*k,t_*k),(t_*k,t_*k)} & = \Hom_{\sE}(t_*k,t_*k) \otimes_{\End_{\sC}(\Delta_{S'})}\Hom_{\sE}(t_*k,t_*k) \\
               & \simeq \coFib \bigt{\End_{\sC}(\Delta_{S'})\xto{\pi_{\uL}}\End_{\sC}(\Delta_{S'})}\otimes_{\End_{\sC}(\Delta_{S'})}\Hom_{\sE}(t_*k,t_*k) \\
                                                                  & \simeq \coFib \bigt{ \Hom_{\sE}(t_*k,t_*k)\xto{0}\Hom_{\sE}(t_*k,t_*k)} \\
                                                                  & \simeq \Hom_{\sE}(t_*k,t_*k) \oplus \Hom_{\sE}(t_*k,t_*k)[1].
\end{align*}
A simple computation yields
$$
\Hom_{\sB}(1_{\sB}\oplus 1_{\sB}[1],1_{\sB}\oplus 1_{\sB}[1])
\simeq 
  k[u,u^{-1}] ^{\oplus 2 }\oplus k[u,u^{-1}] ^{\oplus 2}[1]
\simeq
\Hom_{\sE}(t_*k,t_*k) \oplus \Hom_{\sE}(t_*k,t_*k)[1].
$$
Tracing through the above,  we have proven that $\ffG$ induces an equivalence
$$
  \Hom_{\sE \otimes_{\sC}\sE^{\op}}\bigt{(t_*k,t_*k),(t_*k,t_*k)} 
  \xto{ \simeq }
  \Hom_{\sB}(1_{\sB}\oplus 1_{\sB}[1],1_{\sB}\oplus 1_{\sB}[1]),
$$
which implies the claimed fully faithfulness.
\end{proof}

\ssec{The Morita equivalence}

\sssec{}

We can rewrite the right adjoint in \eqref{adj:Bmod vs Cmod} as a tensor product. For this, it suffices to notice that $\sE$ is dualizable as a $(\sB,\sC)$-bimodule, with dual equal to $\sE^\op$.
We check this directly, by writing a duality datum.

\begin{prop}
The $(\sB,\sC)$-bimodule $\sE$ admits $\sE^\op$ as a dual.
\end{prop}
\begin{proof}
One verifies that the maps
$$
  \ffF^{-1}:\sC\to \sE^{\op}\otimes_{\sB}\sE,
$$
$$
  \ffG: \sE \otimes_{\sC} \sE^{\op}\to \sB
$$
provide a duality datum:
consider the composition
$$
  \sE 
  \xto{\simeq}
  \sE \otimes_{\sC}\sC
  \xto{\id \otimes \ffF}
  \sE \otimes_{\sC} \sE^{\op} \otimes_{\sB}\sE
  \xto{\ffG \otimes \id}
  \sB \otimes_{\sB}\sE
  \xto{\simeq}
  \sE.
$$
This corresponds, under the canonical equivalence $\sE \simeq \sE \otimes_{\sC}\sC$, to the morphism induced by acting on $\sE$ by the unit object $1_{\sC}$. This is clearly homotopic to the identity map.
The other triangular identity works in the same manner.
\end{proof}

\sssec{}

Hence, the adjunction (\ref{adj:Bmod vs Cmod}) can be rewritten as
\begin{equation} \label{adj:Morita-tensor}
\begin{tikzpicture}[scale=1.5]
\node (a) at (0,1) {$\Mod_\sB$};
\node (b) at (2.5,1) {$\Mod_\sC$.};
\path[->,font=\scriptsize,>=angle 90]
([yshift= 1.5pt]a.east) edge node[above] {$\sE^\op \otimes_{\sB} -$ } ([yshift= 1.5pt]b.west);
\path[->,font=\scriptsize,>=angle 90]
([yshift= -1.5pt]b.west) edge node[below] {
$
\sE \otimes_{\sC} -
$ } ([yshift= -1.5pt]a.east);
\end{tikzpicture}
\end{equation}

\begin{cor} \label{cor:Morita}
The bimodule $\sE$ yields a Morita equivalence between the monoidal dg-categories $\sB$ and $\sC$. 
In other words, the adjoint functors \eqref{adj:Morita-tensor} are mutually inverse equivalences.
\end{cor}
\begin{proof}
By Corollary \ref{cor: End_B(E)=C}, we have
$$
\sE^{\op} \otimes_{\sB} (\sE \otimes_{\sC} - )  \simeq 
(\sE^{\op}\otimes_{\sB} \sE) \otimes_{\sC} -
\simeq 
\sC \otimes_{\sC} -
\simeq \id.
$$
Similarly, Theorem \ref{thm: End_C(E)=B} yields
$$
  \sE \otimes_{\sC} (\sE^{\op} \otimes_{\sB} - ) 
 \simeq 
 (\sE\otimes_{\sC} \sE^{\op}) \otimes_{\sB} - 
 \simeq 
 \sB \otimes_{\sB} - 
 \simeq 
 \id.
$$
\end{proof}

\sssec{} 

An immediate consequence is the following version of Theorem \ref{Thm: Kunneth for Sing} for the other side of the Morita equivalence.

\begin{cor}\label{cor: Kunneth for Sing (C side)}
Let $Y$ and $Z$ be regular $S$-schemes of finite type with smooth generic fibers.
Consider the dg-categories $\Sing(Y')$ and $\Sing(Z')$, both equipped with the convolution action of $\sC$. There is a canonical equivalence
$$
\Sing(Y')^\op
\otimes_{\sC}
\Sing(Z')
\xto{\;\; \simeq \;\;}
\Sing(Y \times_S Z).
$$
\end{cor}

\begin{proof}
By Theorem \ref{Thm: Kunneth for Sing} we know that
$$
  \Sing(Y')\simeq E^{\op}\otimes_{\sB}\Sing(Y_s), \hspace{0.5cm} \Sing(Z')\simeq E^{\op}\otimes_{\sB}\Sing(Z_s).
$$
Moreover, since Grothendieck duality exchanges left and right $\sB$-module structures and it is compatible with the equivalences above, we have that 
$$
  \Sing(Y')^{\op}\simeq \Sing(Y_s)^{\op}\otimes_{\sB}\sE.
$$
Therefore, 
\begin{align*}
  \Sing(Y')^{\op}\otimes_{\sC}\Sing(Z') 
 & \simeq 
  \bigt{\Sing(Y_s)^{\op}\otimes_{\sB}\sE}\otimes_{\sC} \bigt{\sE^{\op}\otimes_{\sB}\Sing(Z_s)} 
  \\
& \simeq 
\Sing(Y_s)^{\op}\otimes_{\sB}(\sE \otimes_{\sC} \sE^{\op})\otimes_{\sB}\Sing(Z_s)
\\
& \simeq 
\Sing(Y_s)^{\op}\otimes_{\sB}\sB \otimes_{\sB}\Sing(Z_s) 
\\
& \simeq 
\Sing(Y_s)^{\op}\otimes_{\sB}\Sing(Z_s)       
\\
& \simeq \Sing(Y\times_S Z),                                   
\end{align*}
where the third equivalence follows from Theorem \ref{thm: End_C(E)=B} and the last one from Theorem \ref{Thm: Kunneth for Sing}.
\end{proof}

\ssec{Duality datum for \texorpdfstring{$\sU$}{U} over \texorpdfstring{$\sC$}{C}}\label{ssec: duality datum for U over C}

Here, we show that $\sU$ is dualizable as a left $\sC$-module. We prove this by combining the Morita equivalence proven above with the known dualizability of $\sT$ as a left $\sB$-module.

\sssec{}

Recall that $\sT$ is a dualizable left $\sB$-module, with dual the right $\sB$-module $\sT^{\op}$.

\begin{prop}\label{prop: U dualizable C module}
The left $\sC$-module $\sU$ and the right $\sC$-module $\sU^\op$ are mutually dual. 
\end{prop}

\begin{proof}
The K\"unneth formula of Theorem \ref{thm:Kunneth} yields a $(\sC,A)$-linear equivalence
$$
  \sE^\op \otimes_\sB \sT
  \simeq
  \Sing(S' \times_S X)
  = 
  \sU
$$
and, consequently, an $(A,\sC)$-linear equivalence
$$
\sT^\op \otimes_\sB \sE
\simeq
\sU^\op.
$$
Then the assertion follows formally from the Morita equivalence, Corollary \ref{cor:Morita}.
In order to see this, fix a duality datum $\bigt{\coev_{\sT/\sB}:A\to \sT^{\op}\otimes_{\sB}\sT, \ev_{\sT/\sB}:\sT\otimes_A \sT^{\op}\to \sB}$ for $\sT$ over $\sB$. 
We define
$$
  \coev_{\sU/\sC}:A
  \to 
  \sU^{\op}\otimes_{\sC}\sU
$$
to be the map that corresponds to $\coev_{\sT/\sB}$ under the equivalences
$$
  \sU^{\op}\otimes_{\sC}\sU
  \simeq 
  (\sT^{\op}\otimes_{\sB}\sE)\otimes_{\sC}(\sE^{\op}\otimes_{\sB}\sT)
  \simeq 
  \sT^{\op}\otimes_{\sB}\sT.
$$
We define
$$
  \ev_{\sU/\sC}:\sU \otimes_A \sU^{\op}\to \sU
$$
as the composition
$$
  \sU \otimes_A \sU^{\op}
  \simeq 
  \sE^{\op}\otimes_{\sB}\sT \otimes_A \sT^{\op}\otimes_{\sB}\sE^{\op}
  \xto{\id \otimes \ev_{\sT/\sB} \otimes \id}
  \sE^{\op}\otimes_{\sB}\sB\otimes_{\sB}\sE^{\op}
  \simeq
  \sC.
$$
The pair $(\coev_{\sU/\sC},\ev_{\sU/\sC})$ verifies the first triangular identity because of the following commutative diagram

\begin{equation*}
  \begin{tikzpicture}[scale=1.5]
    \node (LLu) at (0,1) {$\sU$};
    \node (Lu) at (1.4,1) {$\sU \otimes_AA$};
    \node (Cu) at (4.9,1) {$\sU \otimes_A \sU^{\op}\otimes_{\sC}\sU$};
    \node (Ru) at (8.4,1) {$\sC \otimes_{\sC}\sU$};
    \node (RRu) at (9.8,1) {$\sU$};
    \draw[->] (LLu) to node[above] {$\simeq$} (Lu);
    \draw[->] (Lu) to node[above] {$\id \otimes \coev_{\sU/\sC}$} (Cu);
    \draw[->] (Cu) to node[above] {$\ev_{\sU/\sC}\otimes \id$} (Ru);
    \draw[->] (Ru) to node[above] {$\simeq$} (RRu);
    \node (LLd) at (0,0) {$\sE\otimes_{\sB}\sT$};
    \node (Ld) at (1.4,0) {$\sE\otimes_{\sB}\sT \otimes_AA$};
    \node (Cd) at (4.9,0) {$\sE\otimes_{\sB}\sT \otimes_A \sT^{\op}\otimes_{\sB}\sT$};
    \node (Rd) at (8.4,0) {$\sE\otimes_{\sB}\sB \otimes_{\sB}\sT$};
    \node (RRd) at (9.8,0) {$\sE\otimes_{\sB}\sT.$};
    \draw[->] (LLd) to node[above] {$\simeq$} (Ld);
    \draw[->] (Ld) to node[above] {$\id \otimes \coev_{\sT/\sB}$} (Cd);
    \draw[->] (Cd) to node[above] {$\id \otimes \ev_{\sT/\sB}\otimes \id$} (Rd);
    \draw[->] (Rd) to node[above] {$\simeq$} (RRd);
    \draw[->] (LLu) to node[left] {$\simeq$} (LLd);
    \draw[->] (Lu) to node[left] {$\simeq$} (Ld);
    \draw[->] (Cu) to node[left] {$\simeq$} (Cd);
    \draw[->] (Ru) to node[left] {$\simeq$} (Rd);
    \draw[->] (RRu) to node[left] {$\simeq$} (RRd);
  \end{tikzpicture}
\end{equation*}
The composition of the arrows at the bottom is homotopic to $\sE\otimes_{\sB}\id_{\sT}\sim \id_{\sE\otimes_{\sB}\sE}$.
The verification that the second triangular identity holds too is similar and left to the reader.
\end{proof}

\begin{rmk}
The above proposition is actually a particular case of the following statement: let $\sM$ be a dualizable $\sB$-module with dual $\sM^\vee$. Then $\sE^{\op}\otimes_{\sB}\sM$ is a dualizable $\sC$-module with dual $\sM^{\vee}\otimes_{\sB}\sE$.
The proof provided above works in this case as well.
\end{rmk}

\ssec{Hochschild (co)homologies}

In this section, we establish relations between the Hochschild homology of $\sB$ and that of $\sC$.

\sssec{}

Given an $A$-linear monoidal dg-category $\sM$, its \emph{Drinfeld center} (or \emph{Hochschild cohomology}) is defined to be the dg-category (actually, an $\bbE_2$-monoidal dg-category)
$$
\HH^*(\sM/A)
:=
\Hom_{\sM \otimes_A \sM^\rev} (\sM,\sM).
$$
Similarly, the \emph{Hochschild homology} of $\sM$ is the tensor product
$$
\HH_*(\sM/A)
:=
\sM 
\otimes_{\sM \otimes_A \sM^\rev} 
\sM.
$$

\sssec{}

In the above section, we constructed a Morita equivalence between $\sB$ and $\sC$ by means of the bimodule $\sE$. 

\sssec{}
Observe the dg-category $\sE \otimes_A \sE^\op$ is a $(\sCenv,\sBenv)$-bimodule.
Using this, the Morita equivalence $\sB\mmod \simeq \sC\mmod$ of Theorem \ref{cor:Morita} can be \virg{doubled} as follows.
\begin{cor}
The functor
$$
(\sE \otimes_A \sE^\op) \otimes_{\sBenv} -:
\sBenv
\mmod
\longto
\sCenv \mmod
$$
is an equivalence of $\infty$-categories.
\end{cor}
\begin{proof}
The proof is similar to that of Corollary \ref{cor:Morita} and it is left as an exercise to the reader.
\end{proof}

\sssec{}

Observe that $\sB$ is naturally a left $\sBenv$-module: when we wish to emphasize this structure, we write $\sB^L$. Similarly, $\sB$ is naturally a right $\sBenv$-module: to emphasize this structure we use the symbol $\sB^R$.
We will use parallel notations for $\sC$.

\sssec{}
The above functor sends the left $\sBenv$-module $\sB^L$ to the left $\sCenv$-module
\begin{equation} \label{eqn:double-morita-for-BL}
(\sE \otimes_A \sE^\op) \otimes_{\sBenv} \sB^L
\simeq
\sE \otimes_\sB \sE^\op
\simeq \sC^L.
\end{equation}
Similarly, we have an equivalence
\begin{equation} \label{eqn:second-double-morita-for-BL}
\sC^R
\otimes_{\sCenv} 
(\sE \otimes_A \sE^\op) 
\simeq
\sE^\op \otimes_\sC \sE
\simeq
\sB^R
\end{equation}
of right $\sBenv$-modules.

\sssec{}

This immediately implies that the Drinfeld centers of $\sB$ and $\sC$ are equivalent.

\begin{cor}
There is a canonical equivalence
$$
\HH^*(\sB/A)
\simeq
\HH^*(\sC/A).
$$
\end{cor}
\begin{proof}
By definition, we have
$$
\HH^*(\sB/A)
=
\Hom_{\sB^{\env}}(\sB,\sB),
\hspace{0.5cm}
\HH^*(\sC/A)
=
\Hom_{\sC^{\env}}(\sC,\sC).
$$
Since $\sC\simeq (\sE\otimes_A \sE^{\op})\otimes_{\sB^{\env}}\sB$, the equivalence above fulfills the desired equivalence
$$
  \Hom_{\sB^{\env}}(\sB,\sB)
  \simeq 
  \Hom_{\sC^{\env}}(\sC,\sC).
$$
\end{proof}
However, since we will need to work with traces, we are more interested in the Hochschild homologies. The main result of this section is the following fact.

\begin{thm} \label{thm:HH-are-equal}
There is a canonical equivalence
$$
\HH_*(\sB/A)
\simeq
\HH_*(\sC/A).
$$
\end{thm}

\begin{proof}
With the above conventions, the Hochschild homologies of $\sB/A$ and $\sC/A$ are given by  
$$
\HH_*(\sB/A)
:=
\sB^R
\otimes_{\sBenv} 
\sB^L
$$
$$
\HH_*(\sC/A)
:=
\sC^R
\otimes_{\sCenv} 
\sC^L.
$$
We wish to construct an equivalence between them.

Combining the two equivalences (\ref{eqn:double-morita-for-BL}) and (\ref{eqn:second-double-morita-for-BL}) above, we formally obtain
\begin{eqnarray}
\nonumber
\HH_*(\sC/A)
& = &
\sC^R
\otimes_{\sCenv} 
\sC^L 
\\
\nonumber
& \simeq &
\sC^R
\otimes_{\sCenv} 
(\sE \otimes_A \sE^\op) \otimes_{\sBenv} \sB^L
\\
\nonumber
& \simeq &
\sB^R
\otimes_{\sBenv} 
\sB^L 
\\
\nonumber
& = &
 \HH_*(\sB/A),
\end{eqnarray}
as claimed.
\end{proof}

\ssec{The main diagram}
\sssec{}

The Hochschild homology $\HH_*(\sB/A)$ appeared in \cite{toenvezzosi22, beraldopippi22} because of the following commutative square, which is a standard construction in the theory of traces of modules over non-commutative algebras:

\begin{equation} 
  \begin{tikzpicture}[scale=1.5]
    \node (Lu) at (0,1) {$\sT \otimes_\sA \sT^\op$};
    \node (Ld) at (0,0) {$\sB$};
    \node (Ru) at (3,1) {$ \sT^\op \otimes_\sB \sT$};
    \node (Rd) at (3,0) {$ \HH_*(\sB/A).$};
    \draw[->] (Lu) to node[left] { $\ev_{\sT/\sB}$ } (Ld);
    \draw[->] (Ru) to node[right] { $\ev^{\HH}_{\sT/\sB}$ } (Rd);
    \draw[->] (Lu) to node[above] {$\on{proj}\circ \on{swap}$} (Ru);
    \draw[->] (Ld) to node[above]{$\ev^{\HH}_{\sB/\sB}$} (Rd);
  \end{tikzpicture}
\end{equation}
%
Recall that the right vertical arrow is written as
$$
\ev^{\HH}_{\sT/\sB}:
\sT^\op 
\otimes_\sB 
\sT
\simeq
\bigt{ \sT \otimes_A \sT^\op  }
\otimes_{\sB \otimes_A \sB^\rev}
\sB
\xto{\ev_{\sT/\sB} \otimes \id}
\sB
\otimes_{\sB \otimes_A \sB^\rev}
\sB
= 
\HH_*(\sB/A).
$$

\sssec{}

Let us now draw the same diagram, but with the pair $(\sB,\sT)$ replaced by the pair $(\sC,\sU)$. For later convenience, we perform a symmetry along the middle vertical axis:
\begin{equation} 
  \begin{tikzpicture}[scale=1.5]
    \node (Lu) at (0,1) {$ \sU^\op \otimes_\sC \sU$};
    \node (Ld) at (0,0) {$ \HH_*(\sC/A)$};
    \node (Ru) at (3,1) {$\sU \otimes_\sA \sU^\op$};
    \node (Rd) at (3,0) {$\sC.$};
    \draw[->] (Lu) to node[left] { $\ev^{\HH}_{\sU/\sC}$ } (Ld);
    \draw[->] (Ru) to node[right] { $\ev_{\sU/\sC}$ } (Rd);
    \draw[<-] (Lu) to node[above] {$\on{proj}\circ \on{swap}$} (Ru);
    \draw[<-] (Ld) to node[above]{$\ev^{\HH}_{\sC/\sC}$} (Rd);
  \end{tikzpicture}
\end{equation}


\sssec{}

It turns out that we can juxtapose these diagrams. Indeed, in Theorem \ref{thm:HH-are-equal} we proved that $\HH_*(\sB/A) \simeq \HH_*(\sC/A)$, while the Corollary \ref{cor: Kunneth for Sing (C side)} yields
$$
\sU^\op 
\otimes_\sC 
\sU
\simeq
\sT^\op 
\otimes_\sB 
\sT.
$$
These equivalences are compatible with the $\HH$-evaluations functors, that is, the diagram
\begin{equation} \label{diag:HH-evaluations}
  \begin{tikzpicture}[scale=1.5]
    \node (Lu) at (0,1) {$ \sT^\op \otimes_\sB \sT$};
    \node (Ld) at (0,0) {$ \HH_*(\sB/A)$};
    \node (Ru) at (3,1) {$\sU^{\op} \otimes_\sC \sU$};
    \node (Rd) at (3,0) {$\HH_*(\sC/A)$};
    \draw[->] (Lu) to node[left] { $\ev^{\HH}_{\sT/\sC}$ } (Ld);
    \draw[->] (Ru) to node[right] { $\ev^{\HH}_{\sU/\sC}$ } (Rd);
    \draw[<->] (Lu) to node[above] {$\simeq$} (Ru);
    \draw[<->] (Ld) to node[above]{$\simeq$} (Rd);
  \end{tikzpicture}
\end{equation}
is commutative.

\sssec{}

The glued commutative diagram now looks as follows:
\begin{equation}  \label{diag:big-comm-diagram}
  \begin{tikzpicture}[scale=1.5]
    \node (Lu) at (-4,1) {$\sT \otimes_\sA \sT^\op$};
    \node (Ld) at (-4,0) {$\sB $};
    \node (Cu) at (0,1) {$\sT^\op \otimes_\sB \sT \simeq \sU^\op \otimes_\sC \sU$ };
    \node (Cd) at (0,0) {$\HH_*(\sB/A) \simeq \HH_*(\sC/A)$};
    \node (Ru) at (4,1) {$\sU \otimes_\sA \sU^\op$};
    \node (Rd) at (4,0) {$ \sC$.};
    \draw[->] (Lu) to node[right] { $\ev_{\sT/\sB}$ } (Ld);
    \draw[->] (Cu) to node[right] { $\ev^{\HH}_{\sT/\sB}\simeq \ev^{\HH}_{\sU/\sC}$ } (Cd);
    \draw[->] (Ru) to node[right] { $\ev_{\sU/\sC}$ } (Rd);
    \draw[->] (Ru) to node[above]{$\on{proj}\circ \on{swap}$} (Cu);
    \draw[->] (Rd) to node[above]{$\ev^{\HH}_{\sC/\sC}$} (Cd);
    \draw[->] (Lu) to node[above]{$\on{proj}\circ \on{swap}$} (Cu);
    \draw[->] (Ld) to node[above]{$\ev^{\HH}_{\sB/\sB}$} (Cd);
  \end{tikzpicture}
\end{equation}
\section{\texorpdfstring{$\ell$}{l}-adic realizations of \texorpdfstring{$\sT$}{T} and \texorpdfstring{$\sU$}{U}}\label{sec: realizations}

In this section, we look at the decategorifications of $\sT$ and $\sU$. 
More precisely, we show that the duality datum of the left $\sB$-module $\sT$ (equivalently, of the left $\sC$-module $\sU$) induces duality data for the left $\rl_S(\sB)$-module $\rl_S(\sT)$ and for the left $\rl_S(\sC)$-module $\rl_S(\sU)$.
Furthermore, we prove that the decategorification of the tensor product splits as the sums of the tensor products of the decategorifications: 
$$
  \rl_S(\sT\otimes_{\sB}\sT)
  \simeq
  \rl_S(\sT)\otimes _{\rl_S(\sB)}\rl_S(\sT)\oplus  \rl_S(\sU)\otimes_{\rl_S(\sC)}\rl_S(\sU).
$$

\ssec{Non-commutative motives and \texorpdfstring{$\ell$}{l}-adic realizations of dg-categories}

\sssec{}

Let $\SH_S$ denote the stable homotopy category of $S$-schemes, see \cite{morelvoevodsky99,robalo15}.
This is a presentable stable symmetric monoidal $\oo$-category endowed with an $\oo$-functor
$$
  \Sigma^{\oo}_+:
  \Sch_S^{\up{sm}}
  \to 
  \SH_S,
$$ 
universal among those symmetric monoidal $\oo$-functors satisfying Nisnevich descent, $\bbA^1$-invariance and such that 
$$
  (\bbP^1,\oo)
  :=
  \coFib (\Sigma^{\oo}_+S \xto{\oo} \Sigma^{\oo}_+\bbP^1)
$$
is an invertible object.

\sssec{}

We will denote $\BU_S$ the object in $\SH_S$ which represents non-connective homotopy-invariant algebraic K-theory:
for every smooth $S$-scheme $Y$,
$$
  \BU_S(Y)
  =
  \HK(Y).
$$

\sssec{}

In \cite[Section 3.2]{brtv18}, Blanc--Robalo--{\TV} construct an $\oo$-functor
$$
  \Mv_S:
  \dgCat_A
  \to
  \Mod_{\BU_S}(\SH_S),
$$
called the \emph{motivic realization of ($A$-linear) dg-categories}.
The main features of $\Mv_S$ are the following:
\begin{itemize}
\item 
it is lax-symmetric monoidal;
\item it is a localizing invariant: it sends localization sequences in $\dgCat_A$ to fiber-cofiber sequences in $\Mod_{\BU_S}(\SH_S)$;
\item it commutes with filtered colimits;
\item for $q:Y\to S$ a quasi-compact quasi-separated $S$-scheme,
      $$
        \Mv_S(\Perf(Y))
        \simeq
        q_*\BU_Y;
      $$
\item  
 it is unital: as a particular instance of the above item, we have
      $$
        \Mv_S(\Perf(S))
        \simeq
        \BU_S.
      $$
\end{itemize}

\sssec{}

Let $\Shvcon(S)$ be the stable $\oo$-category of constructible $\Qell$-adic sheaves on $S_\et$ and let 
$$
  \Shv(S)
  :=
  \Ind \bigt{\Shvcon(S)} 
$$
be its ind-completion.
By \cite{robalo15}, there exists a symmetric monoidal $\oo$-functor
$$
  \rho^\ell_S:
  \SH_S
  \to
  \Shv(S)
$$
and it follows  from \cite{riou10} that
$$
  \rho^\ell_S(\BU_S)
  \simeq
  \Ql{,S}(\beta)
  :=
  \Ql{,S}[\beta, \beta^{-1}],
$$
where $\Ql{,S}[\beta]:=\Sym_{\Ql{,S}}\bigt{\Ql{,S}(1)[2]}$.

\sssec{}

Following \cite{brtv18}, we refer to the composition
$$
  \rl_S:
  \dgCat_S
  \xto{\Mv_S}
  \Mod_{\BU_S}(\SH_S)
  \xto{\rho^\ell_S}
  \Mod_{\Ql{,S}(\beta)}\bigt{\Shv(S)}
$$ 
as the \emph{$\ell$-adic realization of ($A$-linear) dg-categories}.
This $\infty$-functor inherits features similar to those of $\Mv_S$.
In particular, it is lax-symmetric monoidal and unital, it is a localizing invariant and commutes with filtered colimits.
We will usually employ the notation $\rl_S(A)$ instead of $\rl_S(\Perf(S))$.

\begin{rem}\label{rmk: rl over general bases}
As in \cite[Remark 2.2.2]{toenvezzosi22},
both $\Mv_S$ and $\rl_S$ can be defined for $S$ an arbitrary base scheme. 
\end{rem}

\sssec{}

By \cite[Section 2.3]{toenvezzosi22}, the motivic and $\ell$-adic realizations of dg-categories are related by the \emph{non-commutative $\ell$-adic Chern character}, that is, a lax-monoidal natural transformation
$$
  \chern:
  \HK
  \to
  |\rl_S|,
$$
where $|-|$ denotes the Dold--Kan functor.

\ssec{Inertia-invariant vanishing cycles}

We review the computation of the $\ell$-adic realization of $\sT$, which was performed in \cite{brtv18}. Before doing this, we need to set some notation for the vanishing cohomology of $p:X\to S$.

\sssec{}

Let $\scrV_{X/S}$ be the $2$-periodized sheaf of vanishing cycles
$$
  \scrV_{X/S}
  :=
  \Phi_{p}\bigt{\Ql{,X}(\beta)}
  \simeq
  \Phi_{p}\bigt{\Ql{,X}}\otimes_{\Ql{,X_s}}\Ql{,X_s}(\beta).
$$
This is an ind-constructible $\Qell$-adic sheaf on $X_s$ which is equipped with a continuous action of $\uI_{\uK}$.
It is supported on the singular locus of $X/S$.
To shorten the formulas below, we will adopt the following notation:
$$
  \nu
  :=
  p_{s*}(\scrV_{X/S})
  \simeq
  \uH^*_{\et} \Bigt{X_s, \Phi_p\bigt{\Ql{,X}(\beta)}}.
$$

\sssec{}

Consider $\nu^{\IK}$ and $\nu^{\IL}$, i.e. the fixed points of vanishing cohomology with respect to the two inertia groups $\IK$ and $\IL$.  
Notice that $\nu^{\IL}$ is endowed with an $\GLK$-action and 
$$
\nu^{\IK}
\simeq
\bigt{\nu^{\IL}}^{\GLK}.
$$
We define
$$
\nuIquot
:=
\coFib
\bigt{
\nu^{\IK} \to \nu^{\IL}
}
$$
to be the cone of the natural map.

\sssec{}

The fiber-cofiber sequence
\begin{equation}\label{eqn: nu^IK to nu^IL to nuIquot}
\nu^{\IK}
\to
\nu^{\IL}
\to
\nuIquot
\end{equation}
is split, as it is obtained from the split sequence
$$
\QellI
\to
\QellIG
\to
\QellIGred
$$
of $\GLK$-modules upon tensoring up with $\nu^{\IL}$.
In particular, we have
$$
\nuIquot
\simeq
\bbQ_\ell^{\uI}(\GLK)
\usotimes_{\bbQ_\ell^{\uI}[\GLK]}
\nu^{\IL},
$$
and thus
\begin{equation} \label{eqn:auxiliar-for-nu}
\mathcal{F}
\usotimes_{\bbQ_\ell^{\uI}(\GLK)}
\nuIquot
\simeq
\mathcal{F}
\usotimes_{\bbQ_\ell^{\uI}[\GLK]}
\nu^{\IL}
\end{equation}
for any sheaf $\ccF$ equipped with an action of $\bbQ_\ell^{\uI}(\GLK)$.

\sssec{}

The sequence \eqref{eqn: nu^IK to nu^IL to nuIquot} plays a crucial role in this work. 
The following theorem relies its first term with the $\ell$-adic realization of the singularity category of the special fiber.

\begin{thm}[{\cite[Theorem 4.39]{brtv18}}]\label{thm: mainthm brtv18}
We have
\begin{equation}\label{equiv: main equivalence brtv18}
\rl_{X_s} (\sT) \simeq \scrV_{X/S}^{\IK}[-1].
\end{equation}
\end{thm}

\begin{proof}
For future reference, we recall the constructions involved in the proof. 
Let $V$ be a flat $S$-scheme of finite type, which we assume generically smooth.
We will use the following notation
$$
  s\times_SV=V_s
  \xto{i_V}
  V
  \xleftarrow{j_V}
  V_{\eta}=\eta \times_SV.
$$

By definition, we have a fiber-cofiber sequence
$$
  \Ql{,V_s}^\IK
  \to
  \bigt{\Psi_{V/S}(\Ql{,V_\eta})}^\IK
  \to
  \bigt{\Phi_{V/S}(\Ql{,V})}^\IK,
$$
where $\Psi_{V/S}$ denotes the \emph{nearby cycles} $\oo$-functor.
Moreover, consider the split exact sequence
$$
  \Ql{,V_s}
  \to
  \Ql{,V_s}^\IK
  \to
  \Ql{,V_s}(-1)[-1].
$$
Taking into account the canonical equivalence
$$
  \bigt{\Psi_{V/S}(\Ql{,V_\eta})}^\IK
  \simeq
  i_V^* j_{V*}\Ql{,V_\eta}
$$
and the localization sequence
$$
  i_V^!\Ql{,V}
  \to
  \Ql{,V_s}
  \to
  i_V^* j_{V*}\Ql{,V_\eta},
$$
applying the octahedron to the triangle
\begin{equation*}
  \begin{tikzpicture}[scale=1.5]
    \node (UL) at (0,1) {$\Ql{,V_s}$};
    \node (UR) at (2,1) {$\Ql{,V_s}^\IK$};
    \node (C) at (1,0) {$i_V^* j_{V*}\Ql{,V_\eta}$};
    \draw[->] (UL) to node[above] {${}$} (UR);
    \draw[->] (UL) to node[above] {${}$} (C);
    \draw[->] (UR) to node[above] {${}$} (C);
  \end{tikzpicture}
\end{equation*}
we get a canonical fiber-cofiber sequence
$$
  \Ql{,V_s}(-1)[-1]
  \to
  i_V^!\Ql{,V}[1]
  \to
  \bigt{\Phi_{V/S}(\Ql{,V})}^\IK.
$$
Therefore, at the level of $\Ql{,V_s}(\beta)$-modules, we have a canonical fiber-cofiber sequence
$$
  \Ql{,V_s}(\beta)
  \to
  i_V^!\Ql{,V}(\beta)
  \to
  \nu_{V/S}^\IK[-1].
$$
The morphism $\Ql{,V_s}(\beta) \to i_V^!\Ql{,V}(\beta)$ identifies with 
$$
  i_V^*\rl_{V}\bigt{\Perf(V_s)\to \Perf(V)_{V_s}}.
$$
Since we have a commutative diagram
\begin{equation*}
  \begin{tikzpicture}[scale=1.5]
    \node (UL) at (0,1) {$\Perf(V_s)$};
    \node (UR) at (2,1) {$\Perf(V)_{V_s}$};
    \node (LL) at (0,0) {$\Coh(V_s)$};
    \node (LR) at (2,0) {$\Coh(V)_{V_s}$};
    \draw[->] (UL) to node[above] {${}$} (UR);
    \draw[->] (UL) to node[above] {${}$} (LL);
    \draw[->] (UR) to node[right] {${}$} (LR);
    \draw[->] (LL) to node[above] {${}$} (LR);
  \end{tikzpicture}
\end{equation*}
we obtain a canonical fiber-cofiber sequence
$$
  \scrV_{V/S}^\IK[-1]
  \xto{\alpha_V}
  \rl_{V_s}\bigt{\Sing(V_s)}
  \to
  i_V^*\rl_{V}\bigt{\Sing(V)},
$$
where we used the fact that
\begin{align*}
  i_V^*\rl_{V}\bigt{\Perf(V_s)\to \Coh(V)_{V_s}} & \sim \bigt{\Ql{,V_s}(\beta)\to i_V^!\oml{,V}(\beta)}\\
                                              & \sim \bigt{\Ql{,V_s}(\beta)\to \oml{,V_s}(\beta)}\\
                                              & \sim \rl_{V_s}\bigt{\Perf(V_s)\to \Coh(V_s)}\\
\end{align*}
and that $V/S$ is generically smooth, so that
$\Sing(V)_{V_s}
  \simeq
\Sing(V)$.
Since $X$ is assumed to be regular, $\Sing(X)\simeq 0$ so that $\alpha_X$ is an equivalence.
\end{proof}

\begin{rem}
With a similar argument, one gets
$$
  \rl_S (\sT) 
  \simeq 
  (i_S)_*\nu^\IK[-1],
$$
as well as an isomorphism
$$
  \rl_S(\sB)
  \simeq
  (i_S)_*\Ql{}^\IK(\beta)
$$
of commutative algebra objects. 
It is proven in \cite[Theorem 4.39]{brtv18} that the natural $\rl_S(\sB)$-module structure on $ \rl_{S} (\sT)$ agrees with the natural one on the right-hand sides of the above isomorphisms. 

\end{rem}

\ssec{K\"unneth formula for inertia-invariant vanishing cycles}
We now recall the construction of the \emph{K\"unneth morphism} defined by {\TV}, which is necessary for the computation of the $\ell$-adic realization of the dg-category of singularities of a fiber product.

We closely follow \cite[Section 3]{toenvezzosi22} and just add the minor changes needed to keep track of the Tate twists.

\sssec{}

Recall that the functor that computes (homotopy) fixed points
$$
  (-)^\IK:
  \Shv^\IK(X_s)
  \to
  \Shv(X_s)
$$
is compatible with computing stalks in the following sense: for any geometric point $x:\Spec{\Omega}\to X_s$ and any $E\in \Shv^\IK(X_s)$, there is a canonical isomorphism
$$
  (E^{\IK})_x
  \simeq
  (E_x)^\IK.
$$

\sssec{}

Consider the $\ell$-adic dualizing complex $\oml{,X_s}$ as a trivial $\ell$-adic $\IK$-representation and define the functor
$$
  \bbD^\IK_{X_s}
  :=
  \ul \Hom(-,\oml{,X_s}):
  \Shvcon^\IK(X_s)
  \to
  \Shvcon^\IK(X_s)^\op.
$$
This is the \emph{Grothendieck--Verdier duality for $\IK$-representations}: it is indeed a duality, since it is compatible with the (conservative) forgetful dg-functor
$$
\Shvcon^\IK(X_s)
\to
\Shvcon(X_s).
$$
The compatibility between $  \bbD^\IK_{X_s}$ and taking $\IK$-fixed points works as follows.

\begin{lem}\label{lem: (-)^IK vs bbD}
For every $E\in \Shvcon^\IK(X_s)$, there is a functorial equivalence
$$
  \bbD_{X_s}(E^\IK)(-1)[-1]
  \simeq
  \bigt{\bbD^\IK_{X_s}(E)}^\IK
$$
in $\Shvcon(X_s)$.
\end{lem}

\begin{proof}
We repeat the proof of \cite[Lemma 3.2.3]{toenvezzosi22} and verify that it works without trivializing the Tate twist.
Using the lax monoidal structure of $(-)^\IK$, one produces a map
$$
  E^\IK \otimes \bigt{\bbD^\IK_{X_s}(E)}^\IK
  \to
  \bigt{E \otimes \bbD^\IK_{X_s}(E)}^\IK
  \xto{\on{evaluation}}
  \oml{,X_s}^\IK.
$$
Since $\IK$ acts trivially on $\oml{,X_s}$, we have an equivalence
$$
  \oml{,X_s}^\IK
  \simeq
  \oml{,X_s}\otimes \Ql{,X_s}^\IK
  \simeq 
  \oml{,X_s} \oplus \oml{,X_s}(-1)[-1].
$$
By adjunction, we get a map
\begin{align*}
  \bigt{\bbD^\IK_{X_s}(E)}^\IK & \to \ul \Hom\bigt{E^\IK,\oml{,X_s} \oplus \oml{,X_s}(-1)[-1]} \\
                               & \simeq \bbD_{X_s}\bigt{E^\IK} \oplus \bbD_{X_s}\bigt{E^\IK}(-1)[-1]\\
                               & \to \bbD_{X_s}\bigt{E^\IK}(-1)[-1].
\end{align*}
It suffices to check that this morphism induces an isomorphism on the stalks: for a geometric point $x$, we need to show that the arrow
$$
  \Bigt{\bigt{\bbD^\IK_{X_s}(E)}^\IK}_x 
  \simeq
  \Bigt{\bigt{\bbD^\IK_{X_s}(E)}_x}^\IK 
  \to
  \Bigt{\bbD_{X_s}\bigt{E^\IK}}_x(-1)[-1]
$$
is an isomorphism.
The left-hand side can be rewritten as
$$
  \Bigt{\bigt{\bbD^\IK_{X_s}(E)}_x}^\IK 
  \simeq
  \bigt{\ul \Hom(x^!E,\Ql{})}^\IK ,
$$
while the right-hand side can be rewritten as
$$
  \Bigt{\bbD_{X_s}\bigt{E^\IK}}_x(-1)[-1]
  \simeq
  \ul \Hom \bigt{x^!(E^\IK),\Ql{}}(-1)[-1]
  \simeq
  \ul \Hom \bigt{(x^!E)^\IK,\Ql{}}(-1)[-1].
$$
Then the assertion follows from Poincar\'e duality in Galois cohomology, induced by the fundamental class in $\uH^1(\IK,\Ql{})\simeq \Ql{}(1)$.
\end{proof}

\sssec{}

We are now ready to define the K\"unneth morphism of \TV.
As usual, for two separated $s$-schemes $p_s:X_s\to s$ and $q_s:Y_s\to s$, we denote by
$$
  -\boxtimes_s-
  :=
  p_s^*\otimes q_s^*:
  \Shvcon(X_s)\times \Shvcon(Y_s)
  \to
  \Shvcon(X_s\x{s}Y_s)
$$
the external tensor product of $\ell$-adic sheaves.
By abuse of notation, we will also denote $-\boxtimes_s-$ the external tensor product of $\ell$-adic $\IK$-representations.

\sssec{}

Let $E\in \Shvcon^\IK(X_s)$ and  $F\in \Shvcon^\IK(Y_s)$.
The lax-monoidal structure on $(-)^{\IK}$
yields a map
$$
  \bigt{\bbD^\IK_{X_s}(E)}^\IK \boxtimes_s \bigt{\bbD^\IK_{Y_s}(F)}^\IK
  \to
  \bigt{\bbD^\IK_{X_s}(E) \boxtimes_s \bbD^\IK_{Y_s}(F)}^\IK,
$$
which we compose with
$$
  \bigt{\bbD^\IK_{X_s}(E) \boxtimes_s \bbD^\IK_{Y_s}(F)}^\IK
  \to
  \bigt{\bbD^\IK_{X_s\x{s}Y_s}(E\boxtimes_sF)}^\IK.
$$
The latter map arises by duality from the external product of the evaluation maps
$$
  p_s^*\ul \Hom_{X_s}(E,\omega_{X_s}) \otimes q_s^*\ul \Hom_{Y_s}(F,\omega_{Y_s}) \otimes p_s^*E \otimes q_s^*F
  \to
  \omega_{X_s\times_sY_s}
  \simeq
  p_s^*\omega_{X_s}\otimes q_s^*\omega_{Y_s}.
$$
Applying now $\bbD_{X_s\x{s}Y_s}$, we obtain a morphism
$$
  \bbD_{X_s\x{s}Y_s}\Bigt{\bigt{\bbD^\IK_{X_s\x{s}Y_s}(E\boxtimes_sF)}^\IK}
  \to
  \bbD_{X_s\x{s}Y_s}\Bigt{\bigt{\bbD^\IK_{X_s}(E)}^\IK \boxtimes_s \bigt{\bbD^\IK_{Y_s}(F)}^\IK}.
$$
By Lemma \ref{lem: (-)^IK vs bbD}, this identifies with
$$
  (E\boxtimes_sF)^\IK(1)[1] 
  \to
  E^\IK \boxtimes_s F^\IK(2)[2]
$$
and we define the \emph{K\"unneth morphism} as the induced map
$$
  \on{k}:
  (E\boxtimes_sF)^\IK(-1)[-1]
  \to
  E^\IK \boxtimes_s F^\IK.
$$

\sssec{}

Following \TV{}, we define the $\IK$-invariant convolution product of $E$ and $F$ as
$$
  (E\boxast{s} F)^\IK
  :=
  \coFib \bigt{(E\boxtimes_sF)^\IK(-1)[-1]
  \to
  E^\IK \boxtimes_s F^\IK}.
$$
When $X_s=s=Y_s$, we will write $E\circledast F$ instead of $E \boxast{s} F$.
Notice that, since both external products and computing fixed points are compatible with $!$-pushforwards, we have
$$
  (X_s\times_sY_s \to s)_!\bigt{(E\boxast{s}F)^\IK}
  \simeq
  (p_{s!}E\circledast q_{s!}F)^\IK.
$$

\sssec{}

The $\IK$-invariant convolution product allows to compute the $\ell$-adic realization of $\sT^\op\otimes_\sB \sT$.
More precisely, we have the following result.

\begin{thm}[{\cite[Theorem 3.4.2]{toenvezzosi22}}]\label{thm: I-invariant Thom-Sebastiani}
Let $Y\to S$ and $Z\to S$ be two flat, regular, separated and generically smooth $S$-schemes of finite type.
There is a canonical equivalence
$$
  \rl_{Y\times_S Z}\bigt{\Sing(Y\times_S Z)}
  \simeq
  (Y_s\times_sZ_s\to Y\times_SZ)_*\bigt{\scrV_{Y/S}\boxast{s} \scrV_{Z/S}}^\IK(1).
$$
\end{thm}

\begin{proof}

We follow closely the proof in \emph{loc. cit.}

\sssec*{Step 1}
Since nearby cycles are compatible with external tensor products, we have that
$$
  \Psi_{Y_\eta \times_\eta Z_\eta}
  :=
  \Psi_{Y_\eta \times_\eta Z_\eta/\eta}(\Ql{,Y_\eta \times_\eta Z_\eta})
  \simeq
  \Psi_{Y_\eta/\eta}(\Ql{,Y_\eta})
  \boxtimes_s 
  \Psi_{Z_\eta/\eta}(\Ql{,Z_\eta})
  =:
  \Psi_{Y_\eta}
  \boxtimes_s
  \Psi_{Z_\eta}
$$
By \cite[Lemma 3.4.3]{toenvezzosi22}, we obtain the following fiber-cofiber sequence
$$
  \Phi_Y \boxplus_s \Phi_Z 
  \to
  \Phi_{Y\times_S Z}
  \to
  \Phi_Y \boxtimes_s \Phi_Z.
$$
Applying $(-)^\IK$, tensoring with $\Qell(\beta)$ and shifting by $-1$, we obtain the following fiber-cofiber sequence:
\begin{equation}\label{T1}
    \scrV_{Y/S}^\IK[-1] \boxplus_s \scrV_{Z/S}^\IK[-1]
    \to
    \scrV_{Y\times_SZ/S}^\IK[-1]
    \to
    (\scrV_{Y/S} \boxtimes_s \scrV_{Z/S})^\IK[-1].
\end{equation}

\sssec*{Step 2}
For a scheme $V$, we will adopt the notation 
$$
  \oml{,V}^\sg
  :=
  \rl_V\bigt{\Sing(V)}
$$
By \cite[Lemma 3.4.3]{toenvezzosi22} applied to the fiber-cofiber sequences
$$
  \Ql{,Y_s}(\beta)
  \to
  \oml{,Y_s}(\beta)
  \to
  \oml{,Y_s}^\sg
$$
and
$$
  \Ql{,Z_s}(\beta)
  \to
  \oml{,Z_s}(\beta)
  \to
  \oml{,Z_s}^\sg,
$$
we obtain the following fiber-cofiber sequence (notice that \cite[Lemma 3.4.4]{toenvezzosi22} is obviously true in the non-proper case as well)
\begin{equation}\label{T2}
  \oml{,Y_s}^\sg \boxplus_s \omega_{Z_s}^\sg
  \to 
  \oml{,Z_s}^\sg
  \to
  \oml{,Y_s}^\sg \boxtimes_s \oml{,Z_s}^\sg.
\end{equation}

\sssec*{Step 3} The morphisms $\alpha_Y, \alpha_Z$ and $\alpha_{Y\times_SZ}$ provide the following morphism from the fiber-cofiber sequence \eqref{T1} to the fiber-cofiber sequence \eqref{T2}:
\begin{equation*}
    \begin{tikzpicture}[scale=1.5]
      \node (Lu) at (-4,1) {$\scrV_{Y/S}^\IK[-1] \boxplus_s \scrV_{Z/S}^\IK[-1]$};
      \node (Cu) at (0,1) {$\scrV_{Y\times_SZ/S}^\IK[-1]$};
      \node (Ru) at (4,1) {$\bigt{\scrV_{Y/S} \boxtimes_s \scrV_{Z/S}}^\IK[-1]$};
      
      \node (Ld) at (-4,0) {$\oml{,Y_s}^\sg \boxplus_s \oml{,Z_s}^\sg$};
      \node (Cd) at (0,0) {$\oml{,Z_s}^\sg$};
      \node (Rd) at (4,0) {$\oml{,Y_s}^\sg \boxtimes_s \oml{,Z_s}^\sg.$};
      
      \draw[->] (Lu) to node[above] {${}$} (Cu);
      \draw[->] (Cu) to node[above] {${}$} (Ru);
      
      \draw[->] (Ld) to node[above] {${}$} (Cd);
      \draw[->] (Cd) to node[above] {${}$} (Rd);
      
      \draw[->] (Lu) to node[left] {$\alpha_Y \boxplus_s \alpha_Z$} (Ld);
      \draw[->] (Cu) to node[left] {$\alpha_{Y\times_SZ}$} (Cd);
      \draw[->] (Ru) to node[above] {${}$} (Rd);
    \end{tikzpicture}
\end{equation*}
Since $Y$ and $Z$ are regular, Theorem \ref{thm: mainthm brtv18} implies that the left vertical arrow is an equivalence.
It follows that the square on the right is cartesian, so that the cofibers of the central and right vertical maps are equivalent.
As shown in the proof of Theorem \ref{thm: mainthm brtv18}, the cofiber of $\alpha_{Y\times_SZ}$ is equivalent to  
$$
  (Y_s\times_sZ_s\to Y\times_SZ)^*
  \bigt{ \oml{,Y\times_SZ}^\sg}.
$$
On the other hand, the right vertical map identifies with
$$
  \bigt{\scrV_{Y/S} \boxtimes_s \scrV_{Z/S}}^\IK[-1]
  \xto{\up{k}(1)}
  \scrV_{Y/S}^\IK \boxtimes_s \scrV_{Z/S}^\IK(1),
$$
whose cofiber is defined to be $(\scrV_{Y/S} \boxast{s}\scrV_{Z/S})^\IK(1)$.
\end{proof}

\begin{cor}\label{cor: I-invariant Thom-Sebastiani with nu^I=0}
With the same hypothesis as above, assume that either $\scrV_{Y/S}^{\IK}\simeq 0$ or $\scrV_{Z/S}^{\IK}\simeq 0$.
Then there is a canonical equivalence 
$$
  \rl_{Y\times_S Z}\bigt{\Sing(Y\times_S Z)}
  \simeq\
  (Y_s\times_sZ_s\to Y\times_SZ)_*\bigt{\scrV_{Y/S}\boxtimes_s \scrV_{Z/S}}^\IK
$$
in $\Mod_{\Qell(\beta)}\bigt{\Shv^{\IK}(Y\times_S Z)}$.

\end{cor}

\begin{rem}
The assumption of the above corollary holds if $(Y_s)_{\up{red}}$ or $(Z_s)_{\up{red}}$ is regular.
\end{rem}

\ssec{The \texorpdfstring{$\rl_S(\sB)$}{r(B)}-module \texorpdfstring{$\rl_S(\sT)$}{r(T)}}\label{ssec: the r(B)-module r(T)}

In this section we observe that the $\rl_S(\sB)$-module $\rl_S(\sT)$ is dualizable.
Next, we study the canonical morphism
$$
  \rl_S(\sT)\otimes_{\rl_S(\sB)}\rl_S(\sT)
  \to
  \rl_S(\sT \otimes_{\sB}\sT).
$$

\begin{prop}\label{prop: rl(T)/Ql dualizable}
The  $\rl_S(\sB)$-module $\rl_S(\sT)$ is dualizable, and in fact self-dual.
\end{prop}

\begin{proof}

Recall that 
$$
\rl_S(\sB)
\simeq 
(i_S)_* \QellI(\beta)
$$
$$
\rl_S(\sT)
\simeq 
(i_S)_*\nu^\IK[-1]
$$
and that the action of $\rl_S(\sB)$ on $\rl_S(\sT)$ corresponds under these isomorphisms to the natural action of $\QellI(\beta)$ on $\nu^\IK[-1]$.

By \cite{deligne77}, we know that $\nu$ is a dualizable $\Qell(\beta)$-module. 
Since $\IL$ acts unipotently on $\nu$, we deduce from \cite[Lemma 5.2.5]{toenvezzosi22} that $\nu^\IL$ is a dualizable $\QellI(\beta)$-module.
Now, $\nu^\IK$ is a direct summand of $\nu^\IL$ as a $\QellI(\beta)$-module and the dualizability follows.

Let us now show that $\nu^\IK[-1]$ is self-dual. Recall that $\scrV_{X/S}^\IK[-1]$ can be written as the cofiber of the $2$-periodic $\ell$-adic fundamental class (\cite[Definition 3.3.4]{toenvezzosi22}):
$$
  \Ql{,X_s}(\beta)
  \xto{\eta_{X_s}^{2\on{-per}}}
  \omega_{X_s}(\beta)
  \to
  \scrV_{X/S}^\IK[-1].
$$
Applying $\bbD_{X_s}=\ul \Hom_{X_s}(-,\omega_{X_s})$, we obtain the fiber-cofiber sequence
$$
  \bbD_{X_s}\bigt{\scrV_{X/S}^\IK[-1]}
  \to 
  \bbD_{X_s}\bigt{\omega_{X_s}(\beta)}
  \xto{\bbD_{X_s}(\eta_{X_s}^{2\on{-per}})}
  \bbD_{X_s}\bigt{\Ql{,X_s}(\beta)}.
$$
Since $\bbD_{X_s}(\eta_{X_s}^{2\on{-per}})\sim \eta_{X_s}^{2\on{-per}}$, we deduce that 
$$
    \bbD_{X_s}\bigt{\scrV_{X/S}^\IK[-1]}
    \simeq 
    \scrV_{X/S}^\IK[-2],
$$
that is,
\begin{equation}\label{eqn: Grothendieck dual of nu^IK[-1]}
    \bbD_{X_s}\bigt{\scrV_{X/S}^\IK}
    \simeq
    \scrV_{X/S}^\IK[-3]
    \simeq
    \scrV_{X/S}^\IK(1)[-1].
\end{equation}

Since $\QellI=\Qell \oplus \Qell(-1)[-1]$, we have
 $$
   \ul \Hom_{\Qell(\beta)}\bigt{\nu^\IK[-1],\QellI(\beta)}
   \simeq
   \bbD(\nu^\IK)[1]\oplus \bbD(\nu^\IK)(-1)
 $$
 from which we deduce that
 $$
 \ul \Hom_{\QellI(\beta)}\bigt{\nu^\IK[-1],\QellI(\beta)}
 \simeq
 \bbD(\nu^\IK)(-1).
 $$
 By \eqref{eqn: Grothendieck dual of nu^IK[-1]} we get that
$$
  \ul \Hom_{\QellI(\beta)}\bigt{\nu^\IK[-1],\QellI(\beta)}
  \simeq 
  \nu^\IK[-1]
$$
and the claim of the proposition follows.
\end{proof}

\sssec{}\label{sssec: introduction pi_T}
Consider the canonical morphism
$$
  \mu_{\sT/\sB}:
  \rl_S(\sT)\otimes_{\rl_S(\sB)}\rl_S(\sT)
  \to
  \rl_S(\sT^\op\otimes_{\sB}\sT)
$$
induced by the lax-monoidal structure on $\rl_S$.

Since $\sT^{\op}$ is the $\sB$-linear dual of $\sT$ and $\rl_S(\sT)\simeq \rl_S(\sT^{\op})$ is a self-dual $\rl_S(\sB)$-module, we also have a canonical morphism
$$
  \pi_\sT:
  \rl_S(\sT^\op\otimes_{\sB}\sT)
  \simeq
  \rl_S \bigt{\End_{\sB}(\sT)}
  \to
  \End_{\rl_S(\sB)}\bigt{\rl_S(\sT)}
  \simeq
  \rl_S(\sT)\otimes_{\rl_S(\sB)}\rl_S(\sT)
$$

The next lemma shows that $\pi_\sT$ is a left inverse for $\mu_{\sT/\sB}$.

\begin{lem}\label{lem: r(T) otimes_(r(B)) r(T) direct factor of r(T otimes_B T)}
The composition
$$
  \pi_{\sT}\circ \mu_{\sT/\sB}:
  \rl_S(\sT)\otimes_{\rl_S(\sB)}\rl_S(\sT)
  \to
  \rl_S(\sT)\otimes_{\rl_S(\sB)}\rl_S(\sT)
$$
is homotopic to the identity map.
\end{lem}
\begin{proof}
We proceed in steps.

\sssec*{Step 1}
First of all, we notice that there is a canonical morphism
\begin{equation}\label{eqn: can morphism (nu circledast nu)^IK to (nu otimes nu)^IK}
  \rl_S(\sT^{\op}\otimes_\sB \sT)
  \simeq
  (i_S)_*(\nu \circledast \nu)^\IK(1)
  \to
  (i_S)_*(\nu \otimes \nu)^\IK(1).
\end{equation}
In fact, since $\IL$-acts unipotently on $\nu$, there is fiber-cofiber sequence
$$
  (\nu \otimes \nu)^\IL(1)
  \to
  (\nu \otimes \nu)^\IL
  \to
  \nu^\IL \otimes \nu^\IL (1)[1]
$$
which, upon computing $\GLK$-fixed points, yields the fiber-cofiber sequence
$$
  (\nu \otimes \nu)^\IK(1)
  \to
  (\nu \otimes \nu)^\IK
  \to
  \bigt{\nu^\IL \otimes \nu^\IL}^{\GLK} (1)[1].
$$
Consider now the commutative square
\begin{equation*}
    \begin{tikzpicture}[scale=1.5]
    \node (Lu) at (-3,1.2) {$(\nu \otimes \nu)^\IK$};
    \node (Ld) at (-3,0) {$(\nu \otimes \nu)^\IK$};
    \node (Cu) at (0,1.2) {$\nu^\IK \otimes \nu^\IK(1)[1]$};
    \node (Cd) at (0,0) {$\bigt{\nu^\IL \otimes \nu^\IL}^{\GLK} (1)[1],$};
    \draw[->] (Lu) to node[right] { $\id$ } (Ld);
    \draw[->] (Cu) to node[right] { ${}$ } (Cd);
    \draw[->] (Lu) to node[above]{${}$} (Cu);
    \draw[->] (Ld) to node[above]{${}$} (Cd);
  \end{tikzpicture}
\end{equation*}
where the right vertical arrow is provided by the lax-monoidal structure on $(-)^{\GLK}$.

The fiber of the top row is $(\nu \circledast \nu)^\IK(1)$, so that we obtain the desired morphism
$$
  (i_S)_*(\nu \circledast \nu)^\IK(1)
  \to
  (i_S)_*(\nu \otimes \nu)^\IK(1).
$$

\sssec*{Step 2}

By \cite{luzheng19}, we know that
$$
  \bbD(\nu)
  \simeq
  \nu(1)^{\tau},
$$
where $(1)^{\tau}$ denotes the Iwasawa twist.
In particular,
$$
  \nu \otimes \nu(1)^{\tau}
  \simeq
  \End(\nu).
$$

  Notice that
\begin{equation}\label{eqn: (nu otimes nu(1)tau)^IK = (nu otimes nu)^IK(1)}
  \bigt{ \nu \otimes \nu(1)^{\tau}}^\IK
  \simeq
  \bigt{ \nu \otimes \nu}^\IK(1),
\end{equation}
  where on the right hand side appears the Tate twist.
  
  In fact for every $\IK$-representation $M$ over $\Qell(\beta)$ by \cite[Lemma 2.9]{luzheng19} we have a canonical fiber-cofiber sequence
  $$
    M^\IK 
    \to 
    M
    \to 
    M(-1)^{\tau}.
  $$
  Applying the exact $\oo$-functor $(-)^\IK$ we obtain the fiber-cofiber sequence
  $$
    (M^\IK)^\IK 
    \simeq 
    M^\IK \oplus M^\IK[1]
    \xto{\pr{1}}
    M^\IK
    \to 
    \bigt{M(-1)^{\tau}}^\IK,
  $$
  where the first map is the projection onto the first factor.
  In particular, we obtain that 
  $$
  \bigt{M(-1)^{\tau}}^\IK
  \simeq
  M[2]
  \simeq
  M(-1).
  $$
  Then from the chain of equivalences
  $$
  M^\IK
  \simeq
  \Bigt{\bigt{M(1)^{\tau}}(-1)^{\tau}}^\IK
  \simeq
  \bigt{M(1)^{\tau}}^{\IK}(-1)
  $$
  we deduce that 
  $$
  \bigt{M(1)^{\tau}}^{\IK}
  \simeq 
  M^\IK(1).
  $$

\sssec*{Step 3} 

The composition of $\mu_{\sT/\sB}$ with the canonical morphism \eqref{eqn: can morphism (nu circledast nu)^IK to (nu otimes nu)^IK} identifies, using the equivalence \eqref{eqn: (nu otimes nu(1)tau)^IK = (nu otimes nu)^IK(1)} and Theorem \ref{thm: mainthm brtv18}, with the canonical morphism
$$
  \nu^\IK[-1]\otimes_{\QellI(\beta)}\nu^\IK[-1]
  \simeq
  \nu^\IK\otimes_{\QellI(\beta)}\nu^\IK(1)
  \to
  \bigt{ \nu \otimes \nu}^\IK(1)
$$
given by the lax monoidal structure of $(-)^\IK$.

By Proposition \ref{prop: rl(T)/Ql dualizable} and by \eqref{eqn: (nu otimes nu(1)tau)^IK = (nu otimes nu)^IK(1)} this identifies with a morphism
$$
  \End_{\QellI(\beta)}(\nu^\IK[-1])
  \simeq
  \End_{\QellI(\beta)}(\nu^\IK)
  \to 
  \bigt{\End_{\Qell(\beta)}(\nu)}^\IK.
$$
Observe that $\pi_{\sT}: \rl_S(\sT^\op \otimes_{\sB}\sT)\to \rl_S(\sT)\otimes_{\rl_S(\sB)}\rl_S(\sT)$ factors through 
$$
  \rl_S(\sT^\op \otimes_{\sB}\sT)
  \to
  \bigt{ \nu \otimes \nu}^\IK(1).
$$

In other words, we have a commutative diagram
\begin{equation*}
    \begin{tikzpicture}
      \node (Cu) at (0,1) {$\rl_S(\sT)\otimes_{\rl_S(\sB)}\rl_S(\sT)$};
      \node (Cc) at (0,0) {$\rl_S(\sT^\op \otimes_{\sB}\sT)$};
      \node (Rc) at (3,0) {$\bigt{ \nu \otimes \nu}^\IK(1).$};
      \node (Cd) at (0,-1) {$\rl_S(\sT)\otimes_{\rl_S(\sB)}\rl_S(\sT)$};
      \draw[->] (Cu) to node[above] {${}$} (Cc);
      \draw[->] (Cu) to node[above] {${}$} (Rc);
      \draw[->] (Cc) to node[above] {${}$} (Rc);
      \draw[->] (Cc) to node[above] {${}$} (Cd);
      \draw[->] (Rc) to node[above] {${}$} (Cd);
    \end{tikzpicture}
\end{equation*}
Therefore, it remains to show that the composition of the two diagonal arrows is the identity.

\sssec*{Step 4} The composition mentioned above has a very explicit description. 
Indeed, it corresponds to the morphism
\begin{align*}
  \End_{\QellI(\beta)}(\nu^\IK) \simeq (\nu^\IL)^{\GLK}\otimes_{\QellI(\beta)}(\nu^\IL)^{\GLK}(1) & \to \bigt{\nu^\IL \otimes_{\QellI(\beta)}\nu^\IL }^{\GLK}(1) \simeq \bigt{\End_{\QellI(\beta)}(\nu^\IL)}^{\GLK}\\
  & \to \End_{\QellI(\beta)}(\nu^\IK)
\end{align*}
which is homotopic to the identity map.
This follows from the observation that under the canonical equivalence
$$
  \bigt{\End_{\QellI(\beta)}(\nu^\IL)}^{\GLK}
  \simeq
  \End_{\QellI(\beta)}(\nu^\IK)
  \oplus
  \End_{\QellI(\beta)(\GLK)}(\nuIquot)
$$
the two morphisms displayed above are the canonical inclusion/projection of $\End_{\QellI(\beta)}(\nu^\IK)$.

\end{proof}

\ssec{The \texorpdfstring{$\rl_S(\sC)$}{r(C)}-module \texorpdfstring{$\rl_S(\sU)$}{r(U)}}\label{ssec: the r(C)-module r(U)}
In this section, we focus our attention to the decategorification of the other side of the Morita equivalence, see Corollary \ref{cor:Morita}.
In particular, we show that the left $\rl_S(\sC)$-module $\rl_S(\sU)$ identifies with the left $(i_S)_*\QellI(\beta)(\GLK)$-module $(i_S)_*\bigt{\nuIquot}$.

\sssec{}
Let $\QellI(\beta)(\GLK)$ denote the \emph{reduced} group algebra of $\GLK$ over $\QellI(\beta)$.
Equivalently, for $\Qell(\GLK)$ the usual reduced group algebra generated by $\GLK$ over $\Qell$,
$$
  \QellI(\beta)(\GLK)
  \simeq
  \QellI(\beta)\otimes \Qell(\GLK).
$$

\begin{prop} \label{prop: rl(C)}
There is an equivalence of algebras
$$
  \rl_S(\sC)
  \simeq
 (i_S)_* \bigt{ \QellI(\beta)(\GLK) }.
$$
\end{prop}

\begin{proof}
Since $\sC = \Sing(S' \times_S S')$ and $\bigt{(S')_s}_{\on{red}} = s$ is regular, Corollary \ref{cor: I-invariant Thom-Sebastiani with nu^I=0} implies that
$$
\rl_S(\sC)
\simeq 
 (i_S)_* 
\bigt{\scrV_{S'/S} \otimes \scrV_{S'/S}}^{\IK}.
$$
To simplify the right-hand side, we first compute
$$
  \scrV_{S'/S}
  \simeq
  \coFib \bigt{\Qell \to \Qell[\GLK]}(\beta)
  \simeq
  \Qell(\beta)(\GLK).
$$
Then, applying \cite[Lemma 3.4.3]{toenvezzosi22}, we obtain a fiber-cofiber sequence
$$
  \scrV_{S'/S}^\IK \oplus \scrV_{S'/S}^\IK
  \to
  \coFib \bigt{\Qell \to \Qell[\GLK]\otimes \Qell[\GLK]}^\IK(\beta)
  \to
  \bigt{\scrV_{S'/S}\otimes \scrV_{S'/S}}^\IK.
$$
However, $\scrV_{S'/S}^\IK \simeq 0$ and thus
\begin{align*}
  \bigt{
  \scrV_{S'/S} \otimes \scrV_{S'/S}
  }^{\IK} 
& \simeq 
  \coFib \bigt{\Qell \to \Qell[\GLK]\otimes \Qell[\GLK]}^{\IK}(\beta) \\
& \simeq 
\coFib \bigt{\Qell^{\IK} \to (\Qell[\GLK \times \GLK])^{\IK}}(\beta) \\
& \simeq 
\coFib \bigt{\Qell^{\IK} \to (\Qell^{\IL}[\GLK \times \GLK])^{\GLK}}(\beta) \\
& \simeq 
\coFib \bigt{\Qell^{\IK} \to \Qell^{\IK}[\GLK]}(\beta) \\
& \simeq
\QellI(\beta)(\GLK).
\end{align*}
This proves the assertion of the lemma at the level of $\Mod_{\Qell(\beta)}\bigt{\Shv(s)}$. 

\medskip

We now show that the algebra structures are compatible.
By definition, the algebra structure on $\rl_S(\sC)$ is induced by that of $\sC$ and the lax-monoidality of $\rl_S$. The algebra structure of $\QellI(\beta)(\GLK)$ is explicit, hence we can verify the assertion on the generators.

The quasi-smooth closed embedding
$$
  K(S',0)\to S'\times_SS'
$$
of \cite{beraldopippi22} yields a morphism of algebras
$$
  i_S^*\rl_S\Bigt{\Sing \bigt{K(S',0)} \to \Sing(\sC)}
  \simeq 
  \Bigt{
  \QellI(\beta)
  \to
  i_S^*\rl_S(\sC)
  }
$$
which identifies with the canonical inclusion
$$
  \QellI(\beta)
  \subseteq
  \QellI(\beta)(\GLK)
$$
under the equivalence displayed above.
This is true because the morphism above is the canonical morphism
$$
  K(S',0)\simeq S'\times_{\pi_{\uL},\bbA^1_S,\pi_{\uL}}S' \to S'\times_SS'.
$$
The pushforward along this morphism induces a \emph{monoidal} functor
$$
  \Sing \bigt{K(S',0)}\to \Sing(S'\times_SS'), 
$$
whence the statement.

As for the generators $g\in \GLK$, notice that the non-commutative $\ell$-adic Chern character induces a morphism of algebras
$$
  \uG_0(\sC^+)
  =
  \bbZ[\GLK]
  \to
  \HK_0(\sC)
  \to
  \uH^0_\et \bigt{\sC;\Qell},
$$
and the algebra structure on $\uG_0(\sC^+)$ is obviously compatible with that of $\Qell^\IK(\beta)(\GLK)$.
\end{proof}

\begin{prop}\label{prop: rl(U)}
There is a canonical equivalence
$$
  \rl_{S}(\sU)
  \simeq
  (i_S)_*\bigt{ \nuIquot }.
$$

\end{prop}

\begin{proof}
Recall that 
$$
  \sU := \Sing(X\times_{S}S').
$$
Since $\bigt{(S')_s}_{\on{red}} \simeq s$ is regular, Corollary \ref{cor: I-invariant Thom-Sebastiani with nu^I=0} implies that
$$
  \rl_S(\sU)
  \simeq
  (i_S)_*(\nu \otimes \scrV_{S'/S})^\IK.
$$
In particular, notice that $\rl_S(\sU)\simeq (i_S)_*i_S^*\rl_S(\sU)$.
Using the isomorphism $\scrV_{S'/S}\simeq \Qell(\beta)(\GLK)$ established in the previous proposition, we obtain
\begin{align*}
    i_S^*\rl_S(\sU) & \simeq \coFib \bigt{\nu \to \nu \otimes \Qell[\GLK]}^\IK \\
             & \simeq \coFib \bigt{\nu^\IK \to (\nu[\GLK])^\IK}\\
             & \simeq \coFib \bigt{\nu^\IK \to (\nu^\IL[\GLK])^{\GLK}}\\
             & \simeq \coFib \bigt{\nu^\IK \to \nu^\IL},
\end{align*}
as desired.
\end{proof}

\begin{prop}\label{prop: rl(C) acts on rl(U)}
The canonical action of $\rl_S(\sC)$ on $\rl_S(\sU)$ is induced, under the equivalences on Proposition \ref{prop: rl(C)} and Proposition \ref{prop: rl(U)}, by the canonical action of $\QellI(\beta)(\GLK)$ on $\nuIquot$.
\end{prop}
\begin{proof}
We proceed in steps.
\sssec*{Step 1} 
In \cite{beraldopippi23}, we showed that there is a localization sequence of dg-categories
$$
  \Sing(X'_s\subseteq X')
  \to
  \Sing(X'_s)
  \to
  \Sing(X')
$$
whose $\ell$-adic realization 
\begin{equation}\label{eqn: fib-cofib r(Sing(X'_s -> X')) to r(Sing(X'_s)) to r(Sing(X'))}  
  \rl_S\bigt{\Sing(X'_s\subseteq X')}
  \to
  \rl_S\bigt{\Sing(X'_s)}
  \to
  \rl_S\bigt{\Sing(X')}
\end{equation}
corresponds to the fiber-cofiber sequence
\begin{equation}\label{eqn: fib-cofib nu^IL[-1] to nu^IK[-1] to nuIquot}
  (i_S)_*\nu^\IL[-1]
  \to
  (i_S)_*\nu^\IK[-1]
  \to
  (i_S)_*(\nuIquot).
\end{equation}
Furthermore, since $X'$ sits in the cartesian square
\begin{equation*}
  \begin{tikzpicture}[scale=1.5]
    \node (UL) at (0,1) {$X'$};
    \node (UR) at (2,1) {$X\times_S \bbA^1_S$};
    \node (LL) at (0,0) {$S$};
    \node (LR) at (2,0) {$\bbA^1_S,$};
    \draw[->] (UL) to node[above] {${}$} (UR);
    \draw[->] (UL) to node[left] {$p'$} (LL);
    \draw[->] (UR) to node[right] {$p\times_S E$} (LR);
    \draw[->] (LL) to node[above] {$0$} (LR);
  \end{tikzpicture}
\end{equation*}
the sequence \eqref{eqn: fib-cofib r(Sing(X'_s -> X')) to r(Sing(X'_s)) to r(Sing(X'))} is linear for the action of $\rl_S\bigt{\MF(S,0)}$.
By the main results of \cite{brtv18, beraldopippi23}, the sequences \eqref{eqn: fib-cofib r(Sing(X'_s -> X')) to r(Sing(X'_s)) to r(Sing(X'))} and \eqref{eqn: fib-cofib nu^IL[-1] to nu^IK[-1] to nuIquot}  correspond to each other as modules over
$$
  \rl_S\bigt{\MF(S,0)}
  \simeq
  \Ql{,S}^{\on{I}}(\beta).
$$
\sssec*{Step 2}
Recall the isomorphism
$$
  i_S^*\rl_S(\sC)
  \simeq
  \QellI(\beta) \otimes_{\Qell(\beta)}\Qell(\beta)(\GLK)
$$
and observe that the algebra map 
$$
  \QellI(\beta)
  \to
  \QellI(\beta) \otimes_{\Qell(\beta)}\Qell(\beta)(\GLK)
$$
is induced by the dg-functor
$$
  \Sing(S'\times_{\pi_{\uL},\bbA_S^1,\pi_{\uL}}S')
  \to
  \Sing(S'\times_SS')
$$
mentioned in the proof of Proposition \ref{prop: rl(C)}. 
It follows that the equivalence
$$
  i_S^*\rl_S(\sU)
  \simeq
  \nuIquot
$$
is compatible with the $\QellI(\beta)$-module structure under the equivalence $i_S^*\rl_S(\sC)\simeq \QellI(\beta)(\GLK)$.

Therefore, it suffices to show that the equivalence $i_S^*\rl_S(\sU)\simeq \nuIquot$ is compatible with the $\Qell(\beta)[\GLK]$-module structures.

\sssec*{Step 3}
Notice that the canonical morphism
$$
  \Qell(\beta)[\GLK]
  \to
  \QellI(\beta)(\GLK)
$$
corresponds to the natural map
$$
  \rl_S \bigt{\Perf(S'\times \GLK)}
  \to
    \rl_S \bigt{ \Sing(S'\times_S S')}
$$
induced by the pushforward along the canonical surjection
$$
  S'\times \GLK 
  \to
  S'\times_S S'.
$$
Therefore, it remains to show that the equivalence $i_S^*\rl_S(\sU)\simeq \nuIquot$ is compatible with the natural action of $\GLK$ on the two sides under the equivalence
$$
  i_S^*\rl_S\bigt{\Perf(S'\times \GLK)}
  \simeq
  \Qell(\beta)[\GLK].
$$

\sssec*{Step 4}

Using \eqref{eqn: fib-cofib r(Sing(X'_s -> X')) to r(Sing(X'_s)) to r(Sing(X'))} and the equivalences
$$
  \rl_S \bigt{\Sing(X'_s\subseteq X')}
  \simeq
  (i_S)_*\nu^\IL[-1],
$$
$$
  \rl_S\bigt{\Sing(X'_s)}
  \simeq
  \rl_S\bigt{\Sing(X_s)}
  \simeq
  (i_S)_*\nu^\IK[-1],
$$
$$
  \rl_S\bigt{\Sing(X')}
  \simeq
  (i_S)_*(\nuIquot),
$$
it suffices to show that the action of 
$$
\Qell(\beta)[\GLK]
\simeq
i_S^*\rl_S\bigt{\Perf(S'\times \GLK)}
$$ 
on $i_S^*\rl_S\bigt{\Sing(X'_s\subseteq X'}$ identifies with the natural action of $\Qell(\beta)[\GLK]$ on $\nu^\IL[-1]$.

\sssec*{Step 5}

By the properties of $\rl_S$, there is a fiber-cofiber sequence 
\begin{equation}\label{eqn: fib-cofib r(Perf(X'_s)) to r(Coh(X'_s -> X')) to r(Sing(X'_s -> X'))}
  i_S^*\rl_S\bigt{\Perf(X'_s)}
  \to
  i_S^*\rl_S\bigt{\Coh(X'_s\subseteq X')}
  \to
  i_S^*\rl_S\bigt{\Sing(X'_s\subseteq X')},
\end{equation}
which is linear for the action of 
$$
i_S^*\rl_S\bigt{\Perf(S'\times \GLK)}
\simeq \Qell[\GLK].
$$
Observe that the action of $\GLK$ on $\rl_S\bigt{\Perf(X'_s)} \simeq \rl_S\bigt{\Perf(X_s)}$ is trivial.
On the other hand, by \cite[Theorem 4.2.4]{beraldopippi23}, the second factor lives in the following fiber-cofiber sequence
$$
  i_S^*\rl_S\bigt{\Coh(X'_s\subseteq X')}
  \to
  i_S^*\rl_S\bigt{\Perf( X')}
  \to
  i_S^*\rl_S\bigt{\Perf(X'_\eta)}
$$
of $\Qell(\beta)[\GLK]\simeq i_S^*\rl_S\bigt{\Perf(S'\times \GLK)}$-modules.
Notice that the natural action of $i_S^*\rl_S\bigt{\Perf(S'\times \GLK)}\simeq \Qell(\beta)[\GLK]$ on
$$
  i_S^*\rl_S\bigt{\Perf(X'_\eta)}
  \simeq 
  p_{s*}\Psi_p\bigt{\Ql{,X}(\beta)}^\IL
$$
corresponds to that on the nearby cycles (by definition).
From this, it is easy to see that \eqref{eqn: fib-cofib r(Perf(X'_s)) to r(Coh(X'_s -> X')) to r(Sing(X'_s -> X'))} identifies with
$$
  p_{s*}\Ql{,X_s}
  \to
  p'_{s*}(X'_s\to X')^!\Ql{,X'}
  \to
  \nu^{\IL}[-1].
$$
Moreover, this identification is compatible with the natural module structures under the equivalence $\Qell(\beta)[\GLK]\simeq i_S^*\rl_S\bigt{\Perf(S'\times \GLK)}$.%
\end{proof}

\begin{prop}\label{prop: r(U) self dual over r(C)}
The left $\rl_S(\sC)$-module $\rl_S(\sU)$ is dualizable, with dual the right $\rl_S(\sC)$-module $\rl_S(\sU)(1)$.
\end{prop}

\begin{proof}
Recall the isomorphisms $\rl_S(\sU)\simeq \nuIquot$ and $\rl_S(\sC)\simeq \QellI(\beta)(\GLK)$, compatible with the natural actions on both sides (Proposition \ref{prop: rl(C) acts on rl(U)}).

Since $\nuIquot$ is a direct summand of $\nu^\IL$ as a $\QellI(\beta)$-module, the dualizability claim follows as in the previous proof.

The results of Appendix \ref{appendix: traces over group algebras} imply that $\nuIquot$ is a dualizable left $\QellI(\beta)(\GLK)$-module.

Moreover, we can explicitly compute its dual.

Consider the commutative square
\begin{equation*}
  \begin{tikzpicture}[scale=1.5]
    \node (UL) at (0,1) {$\Ql{,X_s}(\beta)$};
    \node (UR) at (2,1) {$i_{X'}^!\Ql{,X'}(\beta)$};
    \node (LL) at (0,0) {$i_{X'}^*\omega_{X'}(\beta)$};
    \node (LR) at (2,0) {$\omega_{X_s}(\beta).$};
    \draw[->] (UL) to node[above] {${}$} (UR);
    \draw[->] (UL) to node[left] {${}$} (LL);
    \draw[->] (UR) to node[right] {${}$} (LR);
    \draw[->] (LL) to node[above] {${}$} (LR);
  \end{tikzpicture}
\end{equation*}
This is cartesian, as
$$
  \coFib \bigt{\Ql{,X_s}(\beta)\to i_{X'}^*\omega_{X'}(\beta)}
  \simeq
  i_{X'}^*\omega^{\sg}_{X'}
  \simeq
  i_{X'}^!\omega^{\sg}_{X'}
  \simeq
  \coFib \bigt{i_{X'}^!\Ql{,X'}(\beta)\to \omega_{X_s}(\beta)}.
$$
We deduce that
$$
  \coFib \bigt{\Ql{,X_s}(\beta)\to i_{X'}^!\Ql{,X'}(\beta)}
  \simeq
  \coFib \bigt{i_{X'}^*\omega_{,X'}(\beta)\to \omega_{X_s}(\beta)}
  \simeq 
  \nu^\IL[-1].
$$
Furthermore, we have that 
$$
  \bbD_{X_s}\bigt{\Ql{,X_s}(\beta)\to i_{X'}^!\Ql{,X'}(\beta)}
  \simeq
  \bigt{i_{X'}^*\omega_{,X'}(\beta)\to \omega_{X_s}(\beta)},
$$
so that 
$$
  \bbD_{X_s}(\nu^\IL[-1])
  \simeq
  \nu^\IL[-2].
$$
Proceeding as in the proof of Proposition \ref{prop: rl(T)/Ql dualizable}, we obtain that
$$
  \Hom_{\QellI(\beta)}\bigt{\nu^\IL[-1],\QellI(\beta)}
  \simeq
  \nu^\IL[-1].
$$
Then it is easy to conclude that
$$
  \Hom_{\QellI(\beta)}\bigt{\nuIquot,\QellI(\beta)}
  \simeq
  \nuIquot[-2]
  \simeq
  \nuIquot(1).
$$
\end{proof}

\sssec{}
We conclude this section with an observation that will be useful later.

\begin{prop}\label{prop: r(U) otimes_(r(C)) r(U) = r(U) otimes_(r(C)) r(U)(1)}
There is a cartesian square
\begin{equation*}
  \begin{tikzpicture}[scale=1.5]
    \node (UL) at (0,1) {$\nu^\IK \otimes_{\QellI(\beta)}\nu^\IK(1)$};
    \node (UR) at (2,1) {$\nu^\IK \otimes_{\QellI(\beta)}\nu^\IK$};
    \node (LL) at (0,0) {$(\nu \otimes \nu)^\IK(1)$};
    \node (LR) at (2,0) {$(\nu \otimes \nu)^\IK.$};
    \draw[->] (UL) to node[above] {${}$} (UR);
    \draw[->] (UL) to node[left] {${}$} (LL);
    \draw[->] (UR) to node[right] {${}$} (LR);
    \draw[->] (LL) to node[above] {${}$} (LR);
  \end{tikzpicture}
\end{equation*}
Moreover, it induces an isomorphism
$$
  \upsilon:
  \bigt{\nuIquot} \otimes_{\QellIGredbeta} \bigt{\nuIquot}(1)
  \xto{\simeq}
  \bigt{\nuIquot} \otimes_{\QellIGredbeta} \bigt{\nuIquot}.
$$
\end{prop}

\begin{proof}
The morphism $\nu^\IK \otimes_{\QellI(\beta)}\nu^\IK(1)\to \nu^\IK \otimes_{\QellI(\beta)}\nu^\IK$ is the one appearing in the fiber-cofiber sequence
$$
  \nu^\IK \otimes \nu^\IK(1)
  \to
  \nu^\IK \otimes_{\QellI(\beta)}\nu^\IK(1)
  \to
  \nu^\IK \otimes_{\QellI(\beta)}\nu^\IK
$$
induced by the multiplication $\QellI \otimes \QellI \to \QellI$.
In the first term, the tensor product is taken over $\Qell(\beta)$.

Next, we write the lower horizontal morphism in a fiber cofiber-sequence as well.

By \cite[Lemma 2.9]{luzheng19}, there is a fiber-cofiber sequence
$$
  \nu^\IK[-2]
  \to
  \nu(1)^\tau
  \to
  \nu,
$$
which we tensor with $\nu$ (over $\Qell(\beta)$):
$$
  \nu^\IK[-2]\otimes \nu
  \to
  \nu(1)^\tau \otimes \nu
  \to
  \nu \otimes \nu.
$$
Next we apply $(-)^\IK$ and obtain the fiber-cofiber sequence
$$
  \bigt{\nu^\IK[-2]\otimes \nu}^\IK
  \to
  \bigt{\nu(1)^\tau \otimes \nu}^\IK
  \to
  \bigt{\nu \otimes \nu}^\IK.
$$
The morphism $(\nu \otimes \nu)^\IK(1)\to (\nu \otimes \nu)^\IK$ identifies with the morphism $\bigt{\nu(1)^\tau \otimes \nu}^\IK \to \bigt{\nu \otimes \nu}^\IK$ above under the equivalence
$$
  \bigt{\nu(1)^\tau \otimes \nu}^\IK
  \simeq 
  (\nu \otimes \nu)^\IK(1).
$$
Moreover, notice that
$$
  \bigt{\nu^\IK[-2]\otimes \nu}^\IK
  \simeq
  \nu^\IK[-2] \otimes \nu^\IK
  \simeq
  \nu^\IK \otimes \nu^\IK(1)
$$
and we obtain that the map
$$
  \coFib \bigt{\nu^\IK \otimes_{\QellI(\beta)}\nu^\IK(1)\to \nu^\IK \otimes_{\QellI(\beta)}\nu^\IK}
  \to
  \coFib \Bigt{\bigt{\nu(1)^\tau \otimes_{\Qell(\beta)} \nu}^\IK \to \bigt{\nu \otimes_{\Qell(\beta)} \nu}^\IK}
$$
is an equivalence. 
In particular, the square in the statement is cartesian.

The second statement is obtained by looking at the induced equivalence between the cofibers of the vertical maps.
\end{proof}

\ssec{A key split fiber-cofiber sequence}

In this section, we prove that the $\ell$-adic realization of
$$
\sT^{\op}\otimes_{\sB} \sT
\simeq 
\Sing(X \times_S X)
\simeq 
\sU^{\op}\otimes_{\sC} \sU
$$
canonically splits as a direct sum in the following way:
\begin{equation} \label{eqn:magic direct-sum}
\rl_S\bigt{\Sing(X \times_S X)}
\simeq 
\bigt{ 
\rl_S(\sT)\otimes_{\rl_S(\sB)}\rl_S(\sT)
}
\oplus
\bigt{
\rl_S(\sU)\otimes_{\rl_S(\sC)}\rl_S(\sU)}
.    
\end{equation}
This is an essential step of our strategy: in \cite{beraldopippi22}, we explained how to recover $\Bl(X/S)$ from the left-hand side (in the situations of Theorem \ref{mainthm: pure char} and Theorem \ref{mainthm: hypersurf}). We will show later that the two terms of the right-hand side give the two pieces of the total dimension.

\sssec{}

To begin, we compute explicitly the term $\bigt{\rl_S(\sU)\otimes_{\rl_S(\sC)}\rl_S(\sU)}$.
\begin{eqnarray}
\nonumber
\bigt{\nuIquot} \otimes_{\QellIGredbeta}\bigt{\nuIquot}
& \simeq &
\bigt{\nuIquot} \otimes_{\QellIGbeta} \nu^{\IL} \\
\nonumber
& \simeq &
\Bigt{\coFib\bigt{\nu^{\IK} \to \nu^{\IL}} } \otimes_{\QellIGbeta} \nu^{\IL} \\
\nonumber
& \simeq &
\coFib
\Bigt{
\bigt{\nu^{\IK} \otimes_{\QellIGbeta} \nu^{\IL} 
}
\to
\bigt{
\nu^{\IL} \otimes_{\QellIGbeta} \nu^{\IL}
 } } \\
 \nonumber
 & \simeq &
\coFib
\Bigt{
\bigt{
\nu^{\IK} \otimes_{\QellI(\beta)} \nu^{\IK}
} 
\to
\bigt{
\nu^{\IL} \otimes_{\QellIGbeta} \nu^{\IL}
}  
}
 \\
  & \simeq &
  \coFib
  \Bigt{
\bigt{ \nu^{\IK} \otimes_{\QellI(\beta)} \nu^{\IK}} 
\to
(\nu \otimes_{\Qell(\beta)} \nu)^{\IK}
}.
\end{eqnarray}
Notice that in the last equivalence we have used that $\IL$ acts unipotently on $\nu_{X/S}$.

The morphism $\nu^\IK \to \nu^\IL$ identifies with the canonical inclusion under the equivalence $\nu^\IL \simeq \nu^\IK \oplus \nuIquot$. Hence, we have constructed an isomorphism
\begin{equation} \label{eqn:decomp I-invariant nu-otimes-nu}
\bigt{\nu \otimes_{\Qell(\beta)}\nu}^{\IK}
\simeq
 \Bigt{ 
 \nu^{\IK} \otimes_{\QellI(\beta)} \nu^{\IK}
 }
\oplus
\Bigt{
\bigt{\nuIquot} \otimes_{\QellIGredbeta} \bigt{\nuIquot}}.
\end{equation}

\sssec{}

We now use the computations of Propositions \ref{prop: rl(C)} and \ref{prop: rl(U)} to rewrite the right-hand side of \eqref{eqn:decomp I-invariant nu-otimes-nu}. We obtain
\begin{equation}\label{eqn: splitting I_K-invariant external tensor product}
(i_S)_*\bigt{\nu \otimes_{\Qell(\beta)}\nu}^{\IK}
\simeq
\bigt{
\rl_S(\sT)
\otimes_{\rl_s(\sB)}
\rl_S(\sT)
}(-1)
\oplus
\bigt{
\rl_S(\sU)
 \otimes_{\rl_s(\sC)}
 \rl_S(\sU)
 }.
\end{equation}
In particular, we have an arrow
$$
\pr{\sU}:
(i_S)_*\bigt{\nu \otimes_{\Qell(\beta)} \nu}^{\IK}
\longto
\rl_S(\sU)
 \usotimes_{\rl_S(\sC)}
 \rl_S(\sU).
$$

The following observation is crucial.

\begin{prop}\label{prop: key splif fib-cofib sequence}
There is a split fiber-cofiber sequence
$$
  \rl_S(\sT)\otimes_{\rl_S(\sB)}\rl_S(\sT)
  \xto{\mu_{\sT/\sB}}
  \rl_S(\sT^{\op}\otimes_{\sB}\sT)
  \xto{\pi_\sU} 
  \rl_S(\sU)\otimes_{\rl_S(\sC)}\rl_S(\sU).
$$
\end{prop}

The projection induced by the above splitting
$$
  \rl_S(\sT^\op \otimes_\sB \sT) 
  \to
  \rl_S(\sT)\otimes_{\rl_S(\sB)}\rl_S(\sT)
$$
is the morphism $\pi_{\sT}$ considered in Section \ref{sssec: introduction pi_T}.

\begin{proof}
By Theorem \ref{thm: I-invariant Thom-Sebastiani}, we have a canonical fiber-cofiber sequence
$$
  (i_S)_*\bigt{\nu^\IK \otimes_{\Qell(\beta)}\nu}^\IK(1)
  \to
  (i_S)_*\bigt{\nu \circledast \nu}^\IK(1)
  \simeq
  \rl_S(\sT^{\op}\otimes_{\sB}\sT)
  \xto{\on{can}}
  \bigt{\nu \otimes \nu}^\IK.
$$
Also, by (\ref{eqn: splitting I_K-invariant external tensor product}), we have a split fiber-cofiber sequence
$$
  (i_S)_*\bigt{\nu^\IK \otimes_{\QellI(\beta)}\nu^\IK}
  \to
  (i_S)_*\bigt{\nu \otimes \nu}^\IK
  \xto{\pr{\sU}} 
  \rl_S(\sU)\otimes_{\rl_S(\sC)}\rl_S(\sU).
$$

Next, observe that we have a commutative triangle  
\begin{equation*}
    \begin{tikzpicture}[scale=1.5]
      \node (Lu) at (0,1) {$\rl_S(\sT)\otimes_{\rl_S(A)}\rl_S(\sT)$};
      \node (Ru) at (3,1) {$\rl_S(\sT)\otimes_{\rl_S(\sB)}\rl_S(\sT)$};
      \node (Rd) at (3,0) {$\rl_S\bigt{\sT^\op \otimes_\sB \sT}$};
      \draw[->] (Lu) to node[above] {$ $} (Ru);
      \draw[->] (Lu) to node[left] {${}$} (Rd);
      \draw[->] (Ru) to node[right] {$\mu_{\sT/\sB}$} (Rd);
    \end{tikzpicture}
\end{equation*}
where the diagonal arrow corresponds to the canonical map 
$$
  \nu^\IK \otimes_{\Qell(\beta)}\nu^\IK(1)
  \to
  \bigt{\nu \circledast \nu}^\IK(1).
$$
Notice that the map
$$
\rl_S(\sT)\otimes_{\rl_S(A)}\rl_S(\sT)
\to
\rl_S(\sT)\otimes_{\rl_S(\sB)}\rl_S(\sT)
$$
is induced by the multiplication map, which sits in a split fiber-cofiber sequence
$$
  \rl_S(\sB)\otimes_{\rl_S(A)}\rl_S(\sB)
  \xto{\on{mult.}}
  \rl_S(\sB)
  \to
  \rl_S(\sB)(-1).
$$
Putting all pieces together, we get the following commutative diagram
\begin{equation*}
    \begin{tikzpicture}[scale=1.5]
      \node (Lu) at (0,1) {$\rl_S(\sT)\otimes_{\rl_S(A)}\rl_S(\sT)$};
      \node (Cu) at (3,1) {$\rl_S(\sT)\otimes_{\rl_S(\sB)}\rl_S(\sT)$};
      \node (Ru) at (6,1) {$\rl_S(\sT)\otimes_{\rl_S(\sB)}\rl_S(\sT)(-1)$};
      \node (Lc) at (0,0) {$\rl_S(\sT)\otimes_{\rl_S(A)}\rl_S(\sT)$};
      \node (Cc) at (3,0) {$\rl_S\bigt{\sT^\op \otimes_\sB \sT}$};
      \node (Rc) at (6,0) {$(i_S)_*\bigt{\nu \otimes \nu}^\IK$};
      \node (Cd) at (3,-1) {$\rl_S(\sU)\otimes_{\rl_S(\sC)}\rl_S(\sU)$};
      \node (Rd) at (6,-1) {$\rl_S(\sU)\otimes_{\rl_S(\sC)}\rl_S(\sU),$};
      \draw[->] (Lu) to node[above] {$ $} (Cu);
      \draw[->] (Cu) to node[above] {$ $} (Ru);
      \draw[->] (Lc) to node[above] {$ $} (Cc);
      \draw[->] (Cc) to node[above] {$\on{can}$} (Rc);
      \draw[->] (Cd) to node[above] {$\id$} (Rd);
      \draw[->] (Lu) to node[left] {${}$} (Lc);
      \draw[->] (Cu) to node[right] {$\mu_{\sT/\sB}$} (Cc);
      \draw[->] (Ru) to node[right] {$ $} (Rc);
      \draw[->] (Cc) to node[right] {$\pi_\sU$} (Cd);
      \draw[->] (Rc) to node[right] {$\pr{\sU}$} (Rd);
    \end{tikzpicture}
\end{equation*}
where the top right square is cartesian and all rows  and columns are fiber-cofiber sequences. 
The vertical fiber-cofiber sequence in the middle is the one referred to in the statement of the proposition.
It splits in view of Lemma \ref{lem: r(T) otimes_(r(B)) r(T) direct factor of r(T otimes_B T)}.
\end{proof}

\begin{rmk}\label{rmk: factorization pi_U}
Notice that the morphism
$$
  \rl_S(\sT^\op \otimes_{\sB}\sT)
  \xto{\pi_{\sU}}
  \rl_S(\sU)\otimes_{\rl_S(\sC)}\rl_S(\sU)
$$
factors as
$$
  \rl_S(\sT^\op \otimes_{\sB}\sT)
  \to
  \rl_S(\sU)\otimes_{\rl_S(\sC)}\rl_S(\sU)(1)
  \xto{\upsilon}
  \rl_S(\sU)\otimes_{\rl_S(\sC)}\rl_S(\sU),
$$
where $\upsilon$ is the equivalence of Proposition \ref{prop: r(U) otimes_(r(C)) r(U) = r(U) otimes_(r(C)) r(U)(1)}.
The first map identifies with the canonical morphism
$$
  \rl_S\bigt{\End_{\sC}(\sU)}
  \to
  \End_{\rl_S(\sC)}\bigt{\rl_S(\sU)}
$$
under the equivalences
$$
  \rl_S(\sT^\op \otimes_{\sB}\sT)
  \simeq
  \rl_S\bigt{\End_{\sC}(\sU)},
$$
$$
  \rl_S(\sU)\otimes_{\rl_S(\sC)}\rl_S(\sU)(1)
  \simeq
  \End_{\rl_S(\sC)}\bigt{\rl_S(\sU)},
$$
where the second equivalence follows from Proposition \ref{prop: r(U) self dual over r(C)}.
\end{rmk}

\ssec{Duality data for the \texorpdfstring{$\ell$}{l}-adic realizations of \texorpdfstring{$\sT$}{T} and \texorpdfstring{$\sU$}{U}} \label{ssec:duality-data for rl}

We now combine the material of the previous sections to show that the duality datum $(\coev,\ev)$ of $\sT$ over $\sB$ induces duality data for $\rl_S(\sT)$ over $\rl_S(\sB)$ and for $\rl_S(\sU)$ over $\rl_S(\sC)$.

\sssec{}

Let us write down the candidate evaluation for $\rl_S(\sT)$ as a $\bigt{\rl_S(\sB),\rl_S(A)}$-bimodule. Start with the evaluation at the categorical level, which is a $(\sB,\sB)$-linear dg-functor
$$
\ev: 
\sT \otimes_A \sT^\op
\to \sB.
$$
Applying $\rl_S$ and using its lax-monoidal structure, we get a $\bigt{\rl_S(\sB),\rl_S(\sB)}$-linear map
$$
\rl_S(\sT) \otimes_{\rl_S(A)} \rl_S(\sT^\op)
\xto{\mu_{\sT/A}}
\rl_S(\sT \otimes_A \sT^\op)
\xto{\rl_S(\ev_{\sT/\sB})}
\rl_S(\sB).
$$
Now observe that $\rl_S(\sT^\op) \simeq \rl_S(\sT)$. 
The resulting map
$$
\eps_\sT:
\rl_S(\sT) \otimes_{\rl_S(A)} \rl_S(\sT)
\to
\rl_S(\sB)
$$
is our candidate for the evaluation of the $\bigt{\rl_S(\sB),\rl_S(A)}$-bimodule $\rl_S(\sT)$.

\sssec{}
By the commutativity of $\rl_S(\sB)$, $\rl_S(\sT)$ is automatically a $\bigt{\rl_S(\sB),\rl_S(\sB)}$-bimodule. Moreover, $\eps_{\sT}$ factors as
$$
\rl_S(\sT) \otimes_{\rl_S(A)} \rl_S(\sT)
\to
\rl_S(\sT) \otimes_{\rl_S(\sB)} \rl_S(\sT)
\xto{\wt{\eps}_{\sT}}
\rl_S(\sB).
$$

We take
$$
\wt{\eps}_\sT: \rl_S(\sT) \otimes_{\rl_S(\sB)} \rl_S(\sT)
\to
\rl_S(\sB)
$$
as our candidate evaluation for the $\rl_S(\sB)$-module $\rl_S(\sT)$.

\sssec{} 

A parallel construction yields the map
$$
\eps_\sU: \rl_S(\sU) \otimes_{\rl_S(A)} \rl_S(\sU)
\to
\rl_S(\sC).
$$
Notice however that the algebra $\rl_S(\sC)$ is not commutative in general, so that this map does not factor through $\rl_S(\sU) \otimes_{\rl_S(\sC)} \rl_S(\sU)$.
Rather, we have an induced morphism
$$
  \eps_{\sU}^{\HH}:
  \rl_S(\sU) \otimes_{\rl_S(\sC)} \rl_S(\sU)
  \to
  \HH_*\bigt{\rl_S(\sC)/\rl_S(A)}.
$$

\sssec{}
Since $\rl_S(\sC)\simeq (i_S)_*\bigt{\QellI\otimes \Qell(\GLK)}(\beta)$ and $(i_S)_*\QellI(\beta)$ is a commutative algebra, $\rl_S(\sU)$ is automatically a $\bigt{\rl_S(\sC),(i_S)_*\QellI(\beta)}$-bimodule.
Then, $\eps_{\sU}$ factors as
$$
  \rl_S(\sU) \otimes_{\rl_S(A)} \rl_S(\sU)
  \to
  \rl_S(\sU) \otimes_{(i_S)_*\QellI(\beta)} \rl_S(\sU)
  \xto{\wt{\eps}_{\sU}}
  \rl_S(\sC).
$$
We take
$$
\wt{\eps}_\sU: \rl_S(\sU) \otimes_{(i_S)_*\QellI(\beta)} \rl_S(\sU)
\to
\rl_S(\sC)
$$
as our candidate evaluation for the $\bigt{\rl_S(\sC),(i_S)_*\QellI(\beta)}$-bimodule $\rl_S(\sU)$.
This map does not factor through $\rl_S(\sU) \otimes_{\rl_S(\sC)} \rl_S(\sU)$, but it induces a map
$$
  \wt{\eps}_{\sU}^{\HH}:
  \rl_S(\sU) \otimes_{\rl_S(\sC)} \rl_S(\sU)
  \to
  \HH_*\bigt{\rl_S(\sC)/(i_S)_*\QellI(\beta)}.
$$

\sssec{}\label{sssec: definition gamma_T gamma_U}

Now recall the isomorphism provided by Theorem \ref{thm: I-invariant Thom-Sebastiani}
$$
\rl_S (\sT^\op \otimes_\sB \sT) 
\simeq 
\rl_S \bigt{\Sing(X\x{S} X)}
\simeq
(i_S)_*\bigt{\nu \circledast \nu }^{\IK}(1).
$$

Using the coevaluation $\Perf(S) \to \sT^\op \otimes_\sB \sT$ (which equals the coevaluation $\Perf(S) \to \sU^\op \otimes_\sC \sU$) and its $\ell$-adic realization
$$
\rl_S(\Delta_X):
\rl_S \bigt{\Perf(S)} = \bbQ_{\ell,S}(\beta)
\to
\rl_S(\sT^\op \otimes_\sB \sT)
\simeq
(i_S)_*\bigt{\nu \circledast \nu }^{\IK}(1),
$$
we obtain the following maps:
$$
\gamma_\sT: 
\Ql{,S}(\beta)
\xto{\;\rl (\Delta_X)\;}
\bigt{\nu \circledast \nu }^{\IK}(1)
\xto{\;\;\pi_\sT\;\;}
\rl_S(\sT)
\usotimes_{\rl_S(\sB)}
\rl_S(\sT),
$$
$$
\gamma_\sU: 
\Ql{,S}(\beta)
\xto{\;\rl (\Delta_X)\;}
\bigt{\nu \circledast \nu }^{\IK}(1)
\xto{\;\;\pi_\sU\;\;}
\rl_S(\sU)
\usotimes_{\rl_S(\sC)}
\rl_S(\sU).
$$
These are the candidate coevaluations for the $\bigt{\rl_S(\sB),\rl_S(A)}$-bimodule $\rl_S(\sT)$ and the $\bigt{\rl_S(\sC),\rl_S(A)}$-bimodule $\rl_S(\sU)$ respectively.

\begin{thm}\label{thm: duality r(T)}

The pair $(\gamma_\sT, \epsilon_\sT)$ is a duality datum for the $\bigt{\rl_S(\sB),\rl_S(A)}$-bimodule $\rl_S(\sT)$.
\end{thm}
\begin{proof}
To reduce clutter, in the course of this proof, we use the notation $r := \rl_S$.
Recall that $\sT$ is dualizable as a left $\sB$-module, with dual the right $\sB$-module $\sT^{\op}$. In particular, the composition
$$
\sT \xto{\simeq}
\sT \otimes_A A
\xto{\id \otimes \coev}
\sT \otimes_A \sT^{\op}\otimes_{\sB}\sT
\xto{\ev \otimes \id} 
\sB \otimes_{\sB}\sT
\xto{\simeq}
\sT
$$
is the identity.
Since $r$ is lax-monoidal, we obtain the following diagram of $\ell$-adic sheaves:
\begin{small}
\begin{equation*}
  \begin{tikzpicture}[scale=1.5]
  
    \node (LLu) at (-4.8,1) {$r(\sT)$};
    \node (Lu) at (-3.2,1) {$r(\sT) \otimes_{r(A)} r(A)$};
    \node (Cu) at (0,1) {$r(\sT)\otimes_{r(A)} r(\sT^\op) \otimes_{r(\sB)} r(\sT)$};
    \node (Ru) at (3.2,1) {$r(\sB) \otimes_{r(\sB)} r(\sT)$};
    \node (RRu) at (4.8,1) {$r(\sT)$};
    
    \node (Lc) at (-1.6,-0.2) {$r(\sT)\otimes_{r(A)} r(\sT^\op \otimes_\sB \sT)$};
    \node (Rc) at (1.6,-0.2) {$r(\sT \otimes_A \sT^\op) \otimes_{r(\sB)} r(\sT)$};
    
    \node (LLd) at (-4.8,-1.2) {$r(\sT \otimes_A A)$};
    \node (Cd) at (0,-1.2) {$r(\sT\otimes_A \sT^\op \otimes_\sB \sT)$};
    \node (RRd) at (4.8,-1.2) {$r(\sB\otimes_\sB \sT)$.};

    \draw[->] (LLu) to node[above] {$\simeq$} (Lu) ;
    \draw[->] (Lu) to node[above] {$\id \otimes \gamma_\sT$} (Cu) ;
    \draw[->] (Cu) to node[above] {$\eps_\sT \otimes \id$} (Ru) ;
    \draw[->] (Ru) to node[above] {$\simeq$} (RRu) ;
    
    \draw[->] (Lu) to node[left] {$\id \otimes r(\coev)\;\;$} (Lc) ;
    \draw[->] (Cu) to node[left] {$\id \otimes \mu_{\sT/\sB}\;\;$} (Lc) ;
    \draw[->] (Cu) to node[right] {$\;\; \mu_{\sT/A}\otimes \id$} (Rc) ;
    \draw[->] (Rc) to node[right] {$\;\;r(\ev_{\sT})\otimes \id$} (Ru) ;
      
    \draw[->] (LLu) to node[right] {$\simeq$} (LLd) ;
    \draw[->] (RRd) to node[left] {$\simeq$} (RRu) ;
    
    \draw[->] (Lc) to node[left] {$\mu_{\sT,\sT^{\op}\otimes_{\sB}\sT} \;\;\;$} (Cd) ;
    \draw[->] (Rc) to node[right] {$\;\;\; \mu_{\sT \otimes_A \sT^{\op},\sT}$} (Cd) ;
    
    \draw[->] (LLd) to node[below] {$r(\id \otimes \coev)$} (Cd) ;
    \draw[->] (Cd) to node[below] {$r(\ev \otimes \id)$} (RRd) ;
  \end{tikzpicture}
\end{equation*}
\end{small}
Notice that the three squares and the right triangle in this diagram are commutative. However, the triangle
\begin{equation*}
  \begin{tikzpicture}[scale=1.5]
    \node (Lu) at (-2,1) {$r(\sT) \otimes_{r(A)} r(A)$};
    \node (Ru) at (2,1) {$r(\sT)\otimes_{r(A)} r(\sT^\op) \otimes_{r(\sB)} r(\sT)$};
    \node (Cd) at (0,0) {$r(\sT)\otimes_{r(A)} r(\sT^\op \otimes_\sB \sT).$};
    
    \draw[->] (Lu) to node[above] {$\id \otimes \gamma_{\sT}$} (Ru) ;
    \draw[->] (Lu) to node[left] {$\id \otimes r(\coev)\;\;$} (Cd) ;
    \draw[->] (Ru) to node[right] {$\;\;\; \id \otimes \mu_{\sT/\sB}$} (Cd) ;
    
  \end{tikzpicture}
\end{equation*}
is not commutative.
Consider the morphism
$$
  \id \otimes \pi_\sT:
  r(\sT)\otimes_{r(A)}r(\sT^\op \otimes_{\sB}\sT)
  \to
  r(\sT)\otimes_{r(A)} r(\sT^\op) \otimes_{r(\sB)} r(\sT).
$$
In order to conclude that 
$$
  (\epsilon_\sT\otimes \id)\circ (\id \circ \gamma_\sT)
  \sim
  \id,
$$
it suffices to prove that the diagram
\begin{equation}\label{eqn: fundamental comm square for (gamma_T,epsilon_T)}
  \begin{tikzpicture}[scale=1.5]
    \node (Lu) at (-2,1) {$r(\sT)\otimes_{r(A)}r(\sT^\op \otimes_{\sB}\sT)$};
    \node (Ru) at (2,1) {$r(\sT)\otimes_{r(A)} r(\sT^\op) \otimes_{r(\sB)} r(\sT)$};
    \node (Ld) at (-2,0) {$r(\sT\otimes_A \sT^\op \otimes_\sB \sT)$};
    \node (Rd) at (2,0) {$r(\sT)$};
    
    \draw[->] (Lu) to node[above] {$\id \otimes \pi_\sT$} (Ru) ;
    \draw[->] (Ld) to node[below] {$r(\ev \otimes \id)$} (Rd) ;
    
    \draw[->] (Lu) to node[left] {$\mu_{\sT, \sT^{\op}\otimes_{\sB}\sT}$} (Ld) ;
    \draw[->] (Ru) to node[right] {$\epsilon_\sT \otimes \id$} (Rd) ;
    
  \end{tikzpicture}
\end{equation}
is commutative.

Notice that since $\sT^\op$ is the $\sB$-dual of $\sT$,
$$
  \sT^\op \otimes_\sB \sT
  \simeq
  \End_\sB(\sT).
$$

Thus, in view of Proposition \ref{prop: rl(T)/Ql dualizable}, \eqref{eqn: fundamental comm square for (gamma_T,epsilon_T)} can be identified with
\begin{equation*}
  \begin{tikzpicture}[scale=1.5]
    \node (Lu) at (-2,1) {$r(\sT)\otimes_{r(A)}r(\End_\sB(\sT))$};
    \node (Ru) at (2,1) {$r(\sT)\otimes_{r(A)} \End_{r(\sB)}\bigt{r(\sT)}$};
    \node (Ld) at (-2,0) {$r(\sT\otimes_A \End_\sB(\sT))$};
    \node (Rd) at (2,0) {$r(\sT),$};
    
    \draw[->] (Lu) to node[above] {${}$} (Ru) ;
    \draw[->] (Ld) to node[below] {$r(\on{act})$} (Rd) ;
    
    \draw[->] (Lu) to node[left] {$\mu_{\sT, \End_{\sB}(\sT)}$} (Ld) ;
    \draw[->] (Ru) to node[right] {$\on{act}$} (Rd) ;
    
  \end{tikzpicture}
\end{equation*}
which is commutative.

The verification that
$$
  (\id \otimes \epsilon_\sT) \circ (\gamma_\sT \otimes \id)
  \sim
  \id
$$
is similar and left to the reader.

This proves that $r(\sT)$ is a self-dual $\bigt{r(\sB), r(A)}$-bimodule with duality datum provided by $(\gamma_\sT,\eps_\sT)$.
\end{proof}

\sssec{}
Notice that $\gamma_{\sT}:\Ql{,S}(\beta)\to \rl_S(\sT)\otimes_{\rl_S(\sB)}\rl_S(\sT)$ factors through $\Ql{,S}(\beta)\to \rl_S(\sB)$, so that we have a map
$$
  \wt{\gamma}_{\sT}:
  \rl_S(\sB)
  \to
  \rl_S(\sT)\otimes_{\rl_S(\sB)}\rl_S(\sT).
$$

\sssec{}

We next observe the following general fact.

Let $R\to R'$ be a morphism of commutative algebra objects in a symmetric monoidal $\oo$-category $\ccC^{\otimes}$. 
Assume that $C$ is an associative algebra in $\ccC^{\otimes}$ and let $R''=R'\otimes_RC$.
Let $M$ be a $(R'',R)$-bimodule.
Notice that $M$ is automatically a $(R'',R')$-bimodule.
Assume that $M$ is a dualizable $(R'',R)$-bimodule, with a specified duality datum
$$
  \ev: 
  M\otimes_R M^{\vee}
  \to
  S,
$$
$$
  \coev:
  R
  \to
  M^{\vee}\otimes_{R''}M.
$$
Notice that $M^{\vee}$ is automatically a $(R',R'')$-bimodule.

Since $\ev$ is $S\otimes_R S$-linear, it is in particular $R'\otimes_RR'$-linear. Then it factors through the $R'$-linear morphism.
$$
  \widetilde{\ev}: 
  M\otimes_{R'} M^{\vee}
  \to
  R''.
$$
Similarly, since $R'$ is commutative $M^{\vee}\otimes_{S}M$ is a $R'$-module, so that $\coev$ determines a morphism
$$
  \widetilde{\coev}:
  R'
  \to
  M^{\vee}\otimes_{R''}M.
$$

\begin{lem}\label{lem: (S,R) dualizable => (S,R') dualizable}
The pair $(\wt{\coev},\wt{\ev})$ provides a duality datum for the $(R'',R')$-bimodule $M$.
\end{lem}
\begin{proof}
That $(\wt{\ev}\otimes_{R''} \id)\circ (\id \otimes_{R'} \wt{\coev})\simeq \id$ can be seen from the following commutative diagram:
\begin{equation*}
    \begin{tikzpicture}[scale=1.5]
      \node (LLu) at (-5,1) {$M$};
      \node (Lu) at (-3,1) {$M\otimes_RR$};
      \node (Cu) at (0,1) {$M\otimes_RM^{\vee}\otimes_{R''}M$};
      \node (Ru) at (3,1) {$R''\otimes_{R''}M$};
      \node (RRu) at (5,1) {$M$};
      
      \node (Ld) at (-3,0) {$M\otimes_{R'}R'$};
      \node (Cd) at (0,0) {$M\otimes_{R'}M^{\vee}\otimes_{R''}M.$};
      
      \draw[->] (LLu) to node[above] {$\simeq$} (Lu);
      \draw[->] (Lu) to node[above] {$\id \otimes_R \coev$} (Cu);
      \draw[->] (Cu) to node[above] {$\ev \otimes_{R''} \id$} (Ru);
      \draw[->] (Ru) to node[above] {$\simeq$} (RRu);
      
      \draw[->] (LLu) to node[left] {$\simeq \;\;\;$} (Ld);
      \draw[->] (Lu) to node[left] {$\simeq$} (Ld);
      \draw[->] (Ld) to node[above] {$\id \otimes_{R'} \widetilde{\coev}$} (Cd);
      \draw[->] (Cu) to node[left] {${}$} (Cd);
      \draw[->] (Cd) to node[right] {$\;\;\;\;\; \widetilde{\ev}\otimes_{R''}\id$} (Ru);
      
      \draw[ ->] (LLu.north) .. controls +(up:8mm) and +(up:8mm) .. (RRu.north);
      \node at (0,1.9) {$\id$};
    \end{tikzpicture}
\end{equation*}
The verification that $(\id \otimes_{R''} \widetilde{\ev})\circ (\widetilde{\coev} \otimes_{R'} \id)\simeq \id$ is similar.
\end{proof}

\begin{cor}\label{cor: duality r(T) over r(B)}
  The pair $(\wt{\gamma}_{\sT}, \wt{\eps}_{\sT})$ provides a duality datum for the $\rl_S(\sB)$-module $\rl_S(\sT)$.
\end{cor}
\begin{proof}
This is a consequence of Lemma \ref{lem: (S,R) dualizable => (S,R') dualizable}.
\end{proof}

\begin{thm}\label{thm:duality for r(U)}
The pair $(\gamma_\sU, \epsilon_\sU)$ is a duality datum for the $\bigt{\rl_S(\sC),\rl_S(A)}$-bimodule $\rl_S(\sU)$.
\end{thm}

\begin{proof}
We use again the notation $r := \rl_S$.
To prove that $(\epsilon_\sU \otimes \id  )\circ (\id \otimes \gamma_{\sU}) \simeq \id$, consider the following diagram, obtained using the dualizability of $\sU$ as a left $\sC$-module and the lax-monoidal structure on $\rl_S$.
\begin{small}
\begin{equation*}
  \begin{tikzpicture}[scale=1.5]
  
    \node (LLu) at (-4.8,1) {$r(\sU)$};
    \node (Lu) at (-3.2,1) {$r(\sU) \otimes_{r(A)} r(A)$};
    \node (Cu) at (0,1) {$r(\sU)\otimes_{r(A)} r(\sU^\op) \otimes_{r(\sC)} r(\sU)$};
    \node (Ru) at (3.2,1) {$r(\sC) \otimes_{r(\sC)} r(\sU)$};
    \node (RRu) at (4.8,1) {$r(\sU)$};
    
    \node (Lc) at (-1.7,0) {$r(\sU)\otimes_{r(A)} r(\sU^\op \otimes_\sC \sU)$};
    \node (Rc) at (1.7,0) {$r(\sU \otimes_A \sU^\op) \otimes_{r(\sC)} r(\sU)$};
    
    \node (LLd) at (-4.8,-1) {$r(\sU \otimes_A A)$};
    \node (Cd) at (0,-1) {$r(\sU\otimes_A \sU^\op \otimes_\sC \sU)$};
    \node (RRd) at (4.8,-1) {$r(\sC\otimes_\sC \sU)$};

    \draw[->] (LLu) to node[above] {$\simeq$} (Lu) ;
    \draw[->] (Lu) to node[above] {$\id \otimes \gamma_\sU$} (Cu) ;
    \draw[->] (Cu) to node[above] {$\epsilon_\sU \otimes \id$} (Ru) ;
    \draw[->] (Ru) to node[above] {$\simeq$} (RRu) ;
      
    \draw[->] (LLd) to node[below] {$r(\id \otimes \coev)$} (Cd) ;
    \draw[->] (Cd) to node[below] {$r(\ev \otimes \id)$} (RRd) ;
      
    \draw[->] (LLu) to node[left] {$\simeq$} (LLd) ;
    \draw[->] (RRd) to node[left] {$\simeq$} (RRu) ;
      
    \draw[->] (Lu) to node[left] {$\id \otimes r(\coev)\;\;\;$} (Lc) ;
    \draw[->] (Cu) to node[left] {$\id \otimes \mu_{\sU/\sC}\;\;\;$} (Lc) ;
    \draw[->] (Cu) to node[right] {$\;\;\; \mu_{\sU/A}\otimes \id$} (Rc) ;
    \draw[->] (Rc) to node[right] {$\;\;\; r(\ev_{\sU})\otimes \id$} (Ru) ;
      
    \draw[->] (Lc) to node[left] {$\mu_{\sU,\sU^{\op}\otimes_{\sC}\sU}\;\;\;$} (Cd) ;
    \draw[->] (Rc) to node[right] {$\;\;\; \mu_{\sU \otimes_A \sU^{\op},\sU}$} (Cd) ;
      
  \end{tikzpicture}
\end{equation*}
\end{small}
As in the previous case, every piece of this diagram is commutative except for the triangle
\begin{equation*}
  \begin{tikzpicture}[scale=1.5]
    \node (Lu) at (-2,1) {$r(\sU) \otimes_{r(A)} r(A)$};
    \node (Ru) at (2,1) {$r(\sU)\otimes_{r(A)} r(\sU^\op) \otimes_{r(\sC)} r(\sU)$};
    \node (Cd) at (0,0) {$r(\sU)\otimes_{r(A)} r(\sU^\op \otimes_\sC \sU).$};
    
    \draw[->] (Lu) to node[above] {$\id \otimes \gamma_{\sU}$} (Ru) ;
    \draw[->] (Lu) to node[left] {$\id \otimes r(\coev) \;\;\;\;\;$} (Cd) ;
    \draw[->] (Ru) to node[right] {$\;\;\;\;\; \id \otimes \mu_{\sU/\sC}$} (Cd) ;
    
  \end{tikzpicture}
\end{equation*}
Consider the morphism
$$
  \id \otimes \pi_\sU:
  r(\sU)\otimes_{r(A)}r(\sU^\op \otimes_{\sC}\sU)
  \to
  r(\sU)\otimes_{r(A)} r(\sU^\op) \otimes_{r(\sC)} r(\sU).
$$
We claim that the square
\begin{equation}\label{eqn: fundamental comm square for (gamma_U,epsilon_U)}
  \begin{tikzpicture}[scale=1.5]
    \node (Lu) at (-2,1) {$r(\sU)\otimes_{r(A)}r(\sU^\op \otimes_{\sC}\sU)$};
    \node (Ru) at (2,1) {$r(\sU)\otimes_{r(A)} r(\sU^\op) \otimes_{r(\sC)} r(\sU)$};
    \node (Ld) at (-2,0) {$r(\sU\otimes_A \sU^\op \otimes_\sC \sU)$};
    \node (Rd) at (2,0) {$r(\sU)$};
    
    \draw[->] (Lu) to node[above] {$\id \otimes \pi_\sU$} (Ru) ;
    \draw[->] (Ld) to node[below] {$r(\ev \otimes \id)$} (Rd) ;
    
    \draw[->] (Lu) to node[above] {${}$} (Ld) ;
    \draw[->] (Ru) to node[right] {$\epsilon_\sU \otimes \id$} (Rd) ;
    
  \end{tikzpicture}
\end{equation}
is commutative.
As $\sU^\op$ is the $\sC$-dual of $\sU$, we have
$$
  \sU^\op \otimes_\sC \sU
  \simeq
  \End_\sC(\sU).
$$
Hence \eqref{eqn: fundamental comm square for (gamma_U,epsilon_U)} - thanks to Proposition \ref{prop: r(U) self dual over r(C)} and Remark \ref{rmk: factorization pi_U} - identifies with the commutative diagram
\begin{equation*}
  \begin{tikzpicture}[scale=1.5]
    \node (Lu) at (-2,1) {$r(\sU)\otimes_{r(A)}r(\End_\sC(\sU))$};
    \node (Ru) at (2,1) {$r(\sU)\otimes_{r(A)} \End_{r(\sC)}\bigt{r(\sU)}$};
    \node (RRu) at (6,1) {$r(\sU)\otimes_{r(A)} r(\sU)\otimes_{r(A)}r(\sU)$};
    \node (Ld) at (-2,0) {$r(\sU\otimes_A \End_\sC(\sU))$};
    \node (Rd) at (2,0) {$r(\sU).$};
    
    \draw[->] (Lu) to node[above] {${}$} (Ru) ;
    \draw[->] (Ru) to node[above] {$\id \otimes \upsilon$} (RRu);
    \draw[->] (Lu.north) .. controls +(up:8mm) and +(up:8mm) .. (RRu.north);
    \node at (2,2) {$\id \otimes \pi_{\sU}$};
    \draw[->] (Ld) to node[below] {$r(\on{act})$} (Rd) ;
    
    \draw[->] (Lu) to node[left] {$\mu_{\sU, \End_{\sC}(\sU)}$} (Ld) ;
    \draw[->] (Ru) to node[right] {$\on{act}$} (Rd) ;
    
    \draw[->] (RRu) to node[right] {$\;\;\; \eps_{\sU}\otimes \id$} (Rd);
  \end{tikzpicture}
\end{equation*}
The proof that
$$
  (\id \otimes \epsilon_\sU) \circ (\gamma_\sU \otimes \id)
  \sim
  \id
$$
is similar.
Thus, $(\gamma_\sU,\epsilon_\sU)$ provides a duality datum for the $\bigt{r(\sC), r(A)}$-bimodule $r(\sU)$.
\end{proof}

\sssec{}
Notice that $\gamma_{\sU}:\Ql{,S}(\beta)\to \rl_S(\sU)\otimes_{\rl_S(\sC)}\rl_S(\sU)$ factors through $(i_S)_*\QellI(\beta)$, thus inducing a morphism
$$
  \wt{\gamma}_{\sU}:(i_S)_*\QellI(\beta)
  \to 
  \rl_S(\sU)\otimes_{\rl_S(\sC)}\rl_S(\sU).
$$
\begin{cor}\label{cor: r(U) dualizable over (r(C),QellI)}
The pair $(\wt{\gamma}_{\sU},\wt{\eps}_{\sU})$ provides a duality datum for the $\bigt{\rl_S(\sC),(i_S)_*\QellI(\beta)}$-bimodule $\rl_S(\sU)$.
\end{cor}
\begin{proof}
This follows from Lemma \ref{lem: (S,R) dualizable => (S,R') dualizable}.
\end{proof}

\sec{Categorical generalized Bloch conductor formula}\label{sec: categorical BCC}

In this section, we prove Theorem \ref{thm: categorical Bloch}.
As a corollary, we prove Theorem \ref{mainthm: pure char}.

\ssec{The categorical Bloch class}

After reviewing the definition of the categorical Bloch class $\Bl^{\cat}(X/S)$, we use the results of the previous section to express $\Bl^{\cat}(X/S)$ as a sum of two canonically defined classes. This is a step in the right direction, as the categorical total dimension also arises as a sum of two classes.

\sssec{} \label{sssec:cat Bloch class} 

Recall that $\sT$ is dualizable as a left $\sB$-module, with dual $\sT^\op$. Hence, we have a dg-functor
\begin{equation}\label{eqn:Trace id of T/B}
\Perf(S)
\xto{\coev_{\sT/\sB}}
\sT^{\op} 
\otimes_{\sB} 
\sT
\xto{\ev^{\HH}_{\sT/\sB }}
\HH_*(\sB/A).
\end{equation}
Applying $\rl_S$, we obtain a morphism
$$
\Ql{,S}(\beta)
\xto{
\rl_S \bigt{\ev^{\HH}_{\sT/\sB}\circ \coev_{\sT/\sB}}
}
\rl_S(\HH_*(\sB/A))
$$
in 
$\Mod_{\Ql{,S}(\beta)}\bigt{\Shv(S)}$.
By adjunction, this morphism yields a class
$$
\Bl^{\cat}(X/S)
\in 
\uH^0_{\et}
\Bigt{ 
S, \rl_S \bigt{\HH_*(\sB/A)}
},
$$
which is called the \emph{categorical Bloch intersection class}. This definition is due to \TV{}, see \cite{toenvezzosi22}.

\begin{rmk}
More precisely, \TV{} define the categorical Bloch intersection class to the class
$$
  [\ev^{\HH}_{\sT/\sB}\circ \coev_{\sT/\sB}]
  \in
  \HK_0\bigt{\HH_*(\sB/A)}.
$$
With our notation we have that
$$
  \Bl(X/S)^{\cat}
  =
  \chern \bigt{[\ev^{\HH}_{\sT/\sB}\circ \coev_{\sT/\sB}]}.
$$
\end{rmk}

\sssec{}

Thanks to \eqref{diag:HH-evaluations}, we can identify $\sT^\op \otimes_\sB \sT$ with $\sU^\op \otimes_\sC \sU$ and $\HH_*(\sB/A)$ with $\HH_*(\sC/A)$.
With this agreement in place, we write $\ev^{\HH}$ for either of the dg-functors 
$$
\ev^{\HH}_{\sT/\sB}:
  \sT^\op \otimes_\sB \sT 
  \to
  \HH_*(\sB/A)
$$
$$
\ev^{\HH}_{\sU/\sC}:
  \sU^\op \otimes_{\sC}\sU 
\to 
  \HH_*(\sC/A).
$$
Similarly, we write $\coev$ for both dg-functors
$$
\coev_{\sT/\sB} :
\Perf(S)
\to 
\sT^\op \otimes_\sB \sT 
$$
$$
\coev_{\sU/\sC} :
\Perf(S)
\to 
\sU^\op \otimes_\sC \sU.
$$

\sssec{}

In particular, the dg-functor \eqref{eqn:Trace id of T/B}
agrees with
\begin{equation} \label{eqn:Trace id of U/C}
\Tr_{\sC}(\id:\sU)
:
\Perf(S) \xto{\coev}
\sU^\op \otimes_\sC \sU
\xto{\ev^{\HH}}
\HH_*(\sC/A).
\end{equation}
Hence, in Section \ref{sssec:cat Bloch class}, we could have alternatively obtained $\Bl^{\cat}(X/S)$ using $\sU$ and $\sC$.
Our present goal is to compare the $\ell$-adic realization of $\ev^{\HH}\circ \coev$ with the the traces $\Tr_{\rl_S(\sB)}\bigt{\id:\rl_S(\sT)}$ and $\Tr_{\rl_S(\sC)}\bigt{\id:\rl_S(\sU)}$.

\sssec{}

From \eqref{diag:big-comm-diagram}, we obtain the commutative diagram
\begin{equation}  \label{diag:big-comm-diagram-rl}
  \begin{tikzpicture}[scale=1.5]
    \node (Lu) at (-4,1.2) {$\rl_S(\sT) \otimes_{\rl_S(\sB)} \rl_S(\sT)$};
    \node (Ld) at (-4,0) {$\rl_S(\sB)$};
    \node (Cu) at (0,1.2) {$\rl_S(\sT^\op \otimes_\sB \sT )\simeq \rl_S(\sU^\op \otimes_\sC \sU)$};
    \node (Cd) at (0,0) {$\rl_S \bigt{\HH_*(\sB/A)} \simeq \rl_S \bigt{\HH_*(\sC/A)}$};
    \node (Ru) at (4,1.2) {$\rl_S(\sU) \otimes_{\rl_S(\sC)} \rl_S(\sU)$};
    \node (Rd) at (4,0) {$ \HH_*\bigt{\rl_S(\sC)/(i_S)_*\QellI(\beta)}$,};
    \draw[->] (Lu) to node[right] { $\wt{\eps}_{\sT}$ } (Ld);
    \draw[->] (Cu) to node[right] { $\rl_S(\ev^{\HH})$ } (Cd);
    \draw[->] (Ru) to node[right] { $\wt{\eps}_{\sU}^{\HH}$ } (Rd);
    \draw[->] (Ru) to node[above]{$ \mu_{\sU/\sC}$} (Cu);
    \draw[->] (Rd) to node[above]{$-\arcat$} (Cd);
    \draw[->] (Lu) to node[above]{$ \mu_{\sT/\sB}$} (Cu);
    \draw[->] (Ld) to node[above]{$\idcat$} (Cd);
  \end{tikzpicture}
\end{equation}
where $\mu_{\sT/\sB}$ and $\mu_{\sU/\sC}$ are the morphisms induced by the lax-monoidal structure on $\rl_S$. 
Moreover, 
$$
  \idcat:= \rl_S \bigt{\sB \to \HH_*(\sB/A)}
$$
and
$$
  -\arcat:
  \HH_*\bigt{\rl_S(\sC)/(i_S)_*\QellI(\beta)}
  \to
  \rl_S \bigt{\HH_*(\sC/A)}
$$
denotes the composition of the morphism 
$$
  \HH_*\bigt{\rl_S(\sC)/\rl_S(A)}
  \to
  \rl_S \bigt{\HH_*(\sC/A)}
$$
induced by the lax monoidal structure of $\rl_S$ with the canonical morphism
$$
  \HH_*\bigt{\rl_S(\sC)/(i_S)_*\QellI(\beta)}
  \to
  \HH_*\bigt{\rl_S(\sC)/\rl_S(A)}.
$$
Notice that the following diagram commutes:
\begin{equation}\label{eqn: epsHH factors through wt(eps)HH}
    \begin{tikzpicture}[scale=1.5]
      \node (Lu) at (0,1) {$\rl_S(\sU)\otimes_{\rl_S(\sC)}\rl_S(\sU)$};
      \node (Ru) at (3,1) {$\HH_* \bigt{\rl_S(\sC)/(i_S)_*\QellI(\beta)}$};
      \node (Ld) at (0,0) {$\HH_* \bigt{\rl_S(\sC)/\rl_S(A)}\simeq \HH_* \bigt{\QellI(\beta)(\GLK)/\Qell(\beta)}.$};
    
      \draw[->] (Lu) to node[above] {$\wt{\eps}_{\sU}^{\HH}$} (Ru);
      \draw[->] (Lu) to node[left] {$\eps_{\sU}^{\HH}$} (Ld);
      \draw[->] (Ru) to node[above] {${}$} (Ld);
    \end{tikzpicture}
\end{equation}

\sssec{}

Next, we add the coevaluation arrows on the top of the above diagram:

\begin{equation}  \label{diag:master-diagram}
  \begin{tikzpicture}[scale=1.5]

    \node (Luu) at (-4,2.4) {$\rl_S(\sB)$ };
    \node (Cuu) at (0,2.4) {$\rl_S(A)$ };
    \node (Ruu) at (4,2.4) {$(i_S)_*\QellI(\beta)$ };

    \node (Lu) at (-4,1.2) {$\rl_S(\sT) \otimes_{\rl_S(\sB)} \rl_S(\sT)$};
    \node (Ld) at (-4,0) {$\rl_S(\sB)$};
    \node (Cu) at (0,1.2) {$\rl_S(\sT^\op \otimes_\sB \sT )\simeq \rl_S(\sU^\op \otimes_\sC \sU)$};
    \node (Cd) at (0,0) {$\rl_S \bigt{\HH_*(\sB/A)} \simeq \rl_S \bigt{\HH_*(\sC/A)}$};
    \node (Ru) at (4,1.2) {$\rl_S(\sU) \otimes_{\rl_S(\sC)} \rl_S(\sU)$};
    \node (Rd) at (4,0) {$ \HH_*\bigt{\rl_S(\sC)/(i_S)_*\QellI(\beta)}.$};
    
    \draw[->] (Cuu) to node[right] { $\rl_S(\coev)$ } (Cu);
    \draw[->] (Ruu) to node[right] { $\wt{\gamma}_\sU$ } (Ru);
    \draw[->] (Luu) to node[left] { $\wt{\gamma}_{\sT}$ } (Lu);
    
    \draw[->] (Lu) to node[right] { $\wt{\eps}_{\sT}$ } (Ld);
    \draw[->] (Cu) to node[right] { $\rl_S(\ev^{\HH})$ } (Cd);
    \draw[->] (Ru) to node[right] { $\wt{\eps}_{\sU}^{\HH}$ } (Rd);
    \draw[->] (Cuu) to node[right] { ${}$ } (Luu);
    \draw[->] (Cuu) to node[right] { ${}$ } (Ruu);
    
    \draw[->] (Ru) to node[above]{$ \mu_{\sU/\sC}$} (Cu);
    \draw[->] (Rd) to node[above]{$-\arcat$} (Cd);
    \draw[->] (Lu) to node[above]{$ \mu_{\sT/\sB}$} (Cu);
    \draw[->] (Ld) to node[above]{$\idcat$} (Cd);
\end{tikzpicture}
\end{equation}

\sssec{}
We warn the reader that the two squares on the top are not commutative. 
Rather, by the splitting \eqref{eqn:magic direct-sum} (see Proposition \ref{prop: key splif fib-cofib sequence}) and the definitions of $\gamma_{\sT}$ and $\gamma_{\sU}$ (see Section \ref{sssec: definition gamma_T gamma_U}), we have
\begin{equation}\label{eqn: magic}
  [\rl_S(\coev)]
  =
  [\mu_{\sT/\sB}\circ \wt{\gamma}_\sT]
  +
  [\mu_{\sU/\sC}\circ \wt{\gamma}_\sU]
  \in
  \uH^0_{\et}\Bigt{S,\rl_S \bigt{\Sing(X\times_SX)}}.
\end{equation}
Now apply the bottom vertical map: the commutativity of the two squares below in \eqref{diag:master-diagram} yields the equality
\begin{equation} \label{eqn:trace of T/B as a sum}
\Bl^{\cat}(X/S)
:=
[\rl_S( \ev^{\HH} \circ \coev)]
=
\idcat
([\wt{\eps}_\sT \circ \wt{\gamma}_\sT])
-\arcat
([\wt{\eps}^{\HH}_\sU \circ \wt{\gamma}_\sU])
\in \uH^0_{\et}\Bigt{S,\rl_S\bigt{\HH(\sB/A)}}.
\end{equation}

\ssec{Proof of Theorem \ref{thm: categorical Bloch}}

To conclude the proof of Theorem \ref{thm: categorical Bloch}, it remains to simplify the two terms on the right-hand side of \eqref{eqn:trace of T/B as a sum}.

We first need the following elementary observation.

\begin{lem}\label{lem: i_* Ql-mod --> i_*Ql-mod sym mon}
Let $i_S:s\to S$ denote the canonical closed embedding. Then
$$
  (i_S)_*: \Shv(s)
  \to
  \Mod_{(i_S)_*\Ql{,s}}(\Shv(S))
$$
is a symmetric monoidal functor.
\end{lem}
\begin{proof}
It is clear that it is a lax monoidal functor.

We are left to verify that, for every $M,N \in \Shv(s)$, the canonical morphism
$$
  (i_S)_*M\otimes_{(i_S)_*\Ql{,s}}(i_S)_*N
  \to
  (i_S)_*(M\otimes_{\Ql{,s}} N)
$$
is an equivalence.

Recall that, for every $P \in \Shv(S)$, the projection formula yields
$$
  (i_S)_*(M\otimes_{\Ql{,s}}i_S^*P)
  \simeq
  (i_S)_*M\otimes_{\Ql{,S}}P.
$$
Then
\begin{align*}
    (i_S)_*M\otimes_{\Ql{,S}}(i_S)_*N & \simeq (i_S)_*\bigt{M\otimes_{\Ql{,s}} i_S^*(i_S)_*N}\\
                              & \simeq (i_S)_*\bigt{M\otimes_{\Ql{,s}} N}.
\end{align*}
Moreover,
\begin{align*}
    (i_S)_*M\otimes_{\Ql{,S}}(i_S)_*N & \simeq (i_S)_*M\otimes_{(i_S)_*\Ql{,s}}\bigt{(i_S)_*\Ql{,s}\otimes_{\Ql{,S}}(i_S)_*N}\\
                              & \simeq (i_S)_*M\otimes_{(i_S)_*\Ql{,s}}(i_S)_*\bigt{\Ql{,s}\otimes_{\Ql{,s}}i_S^*(i_S)_*N}\\
                              & \simeq (i_S)_*M\otimes_{(i_S)_*\Ql{,s}}(i_S)_*N.
\end{align*}
\end{proof}
\begin{cor}\label{cor: tr over Ql^I= tr over i_*Ql^I}
Let $M \in \Shvcon(s)$. Then $(i_S)_*M$ is dualizable in $\Mod_{(i_S)_*\Ql{,s}}(\Shv(S))$ and
$$
  \Tr_{\Ql{,s}}(f:M)
  =
  \Tr_{(i_S)_*\Ql{,s}}\bigt{(i_S)_*f:(i_S)_*M}
$$
for every endomorphism $f$.
\end{cor}

\sssec{}

By Theorem \ref{thm: duality r(T)} we have:

\begin{align}\label{eqn:integration of B-side}
    [\wt{\eps}_{\sT} \circ \gamma_{\sT}] & = [\wt{\eps}_{\sT} \circ \wt{\gamma}_{\sT}] \\
    \nonumber
                                       & = \Tr_{\rl_S(\sB)}\bigt{\id:\rl_S(\sT)} && \text{Cor. \ref{cor: duality r(T) over r(B)}}\\
    \nonumber
                                       & = \Tr_{\QellI(\beta)}\bigt{\id:\nu^{\IK}[-1]} && \text{\cite[Thm. 4.39]{brtv18}, Cor. \ref{cor: tr over Ql^I= tr over i_*Ql^I}}.
\end{align}

\sssec{}
To understand the term $[\wt{\eps}^{\HH}_\sU \circ \wt{\gamma}_\sU]$, first recall that in Proposition \ref{prop: rl(C)} we proved that there is an equivalence $\rl_S(\sC) \simeq \QellI(\beta)(\GLK)$.
We obtain that
$$
\HH_0\bigt{\rl_S(\sC)/(i_S)_*\QellI(\beta)}
\simeq
\HH_0\bigt{\QellI(\beta)(\GLK)/\QellI(\beta)}
\simeq
\Qell(\GLK)/ [\Qell(\GLK),\Qell(\GLK)].
$$

It follows that any ring homomorphism from $\QellIGred$ to a commutative ring $R$ descends uniquely to a map $\HH_0\bigt{\rl_S(\sC)/(i_S)_*\QellI(\beta)} \to R$. This is the case for any ring homomorphism induced by a class function, such as the Artin character. 

\sssec{} 

By Corollary \ref{cor: r(U) dualizable over (r(C),QellI)}, Corollary \ref{cor: tr over Ql^I= tr over i_*Ql^I} and Corollary \ref{cor: trace reduced group alg}, we have
\begin{align*}
[\wt{\eps}^{\HH}_\sU \circ \wt{\gamma}_\sU] 
& = 
\Tr_{\rl_S(\sC)}\bigt{\id:\rl_S(\sU)}
\\
& = \frac 1 {|\GLK|}\sum_{g \in \GLK} \Tr_{\QellI(\beta)} (g^{-1}: \nuIquot) \cdot \langle g \rangle 
\in 
\HH_0
\bigt{
\QellI(\beta)(\GLK)/\QellI(\beta)
}.
\end{align*}
Since $\GLK$ acts trivially on $\nu^{\IK}$, we can simplify the factors in the sum as follows:
\begin{align*}
\Tr_{\QellI(\beta)} (g: \nuIquot) & = \Tr_{\QellI(\beta)}(g: \nu^{\IL}) - \Tr_{\QellI(\beta)} (g: \nu^{\IK}) \\
                                              & = \Tr_{\QellI(\beta)}(g: \nu^{\IL}) - \Tr_{\QellI(\beta)}(\id: \nu^{\IK}).
\end{align*}
We obtain that
\begin{align*}
[\wt{\eps}^{\HH}_\sU \circ \wt{\gamma}_\sU] & = \frac 1 {|\GLK|}\sum_{h \in \GLK} \Tr_{\QellI(\beta)}(g: \nu^{\IL}) \cdot \langle g \rangle - \frac {\Tr_{\QellI(\beta)}(\id: \nu^{\IK})} {|\GLK|}\sum_{g \in \GLK} \langle g \rangle\\
                                            & = \frac 1 {|\GLK|}\sum_{h \in \GLK} \Tr_{\QellI(\beta)}(g: \nu^{\IL}) \cdot \langle g \rangle.
\end{align*}
The second equality follows from the identity 
$$
\sum_{g\in \GLK} \langle g \rangle
=
0
\in 
\HH_0 \bigt{\rl_S(\sC)/(i_S)_*\QellI(\beta)}.
$$

\sssec{}

By applying the map $\arcat$ we obtain the categorical Artin conductor:
\begin{align}\label{eqn:integration of C-side}
\Ar^{\cat}(\Phi) & := \arcat ([\wt{\eps}^{\HH}_\sU \circ \wt{\gamma}_\sU]) \\
\nonumber
                  & = \frac 1 {|\GLK|}\sum_{g \in \GLK} \arcat \bigt{\Tr_{\QellI(\beta)}(g^{-1}: \nu^{\IL}) \cdot \langle g \rangle}.
\end{align}

\sssec{} 

The combination of \eqref{eqn:trace of T/B as a sum} with \eqref{eqn:integration of B-side} and \eqref{eqn:integration of C-side} yields
\begin{align*}
  \Bl^{\cat}(X/S) & = \idcat \bigt{\Tr_{\QellI(\beta)}(\id:\nu^{\IK}[-1])} - \Ar^{\cat}(\Phi)\\  
                  & =: -\dimtot^{\cat} (\Phi),
\end{align*}
as desired.

\ssec{Proof of Theorem \ref{mainthm: pure char}}
We now show how the categorical generalized Bloch conductor formula implies the generalized Bloch conductor formula in pure characteristic.
\sssec{}

Suppose that we are in the pure characteristic situation, in which case $s\to S$ admits a retraction $S\to s$.
We then have equivalences 
$$
  \sB 
  \simeq 
  \MF(s,0)
  \simeq 
  \Perf\bigt{k[u,u^{-1}]}
$$
of $\bbE_{\oo}$-algebras in $\dgCat_k$, where $u$ is a free variable of cohomological degree $2$.
In other words, the monoidal structure on $\sB$ is \emph{symmetric}. This is a major simplification, which occurs in the pure characteristic situation only. 
Accordingly, the case of pure characteristic is simpler than the one of mixed characteristic, when tackled with the methods developed in this paper.

\sssec{}\label{sssec: retraction ev^(HH)_(B/B)}

Since $\sB$ is symmetric monoidal, the dg-functor
$$ 
\ev^{\HH}_{\sB/\sB}:
\sB
\to 
\HH_*(\sB/A)
$$
admits a retraction
$$
  \on{m}:
  \HH_*(\sB/A)
  \to
  \sB.
$$

\sssec{}
Since $S'$ is the zero locus of an Eisenstein polynomial, by \cite{beraldopippi22} we have a dg-functor (see also the beginning of Section \ref{ssec: proof of thm B})
$$
  \int_{S'/S}^{\on{dg}}:
  \Sing(S'\times_S S')
  \to
  \MF(S,0)_s.
$$
The following lemma allows us to relate the $\HH$-evaluation of $\sE$ over $\sB$ with this dg-functor.

\begin{lem}

The diagram
\begin{equation}\label{diag: m circ evHH vs int}
  \begin{tikzpicture}[scale=1.5]
    \node (UM) at (0,1) {$\sC$  };
    \node (M) at (0,0) {$\HH_*(\sB/A)\simeq \HH_*(\sC/A)$};
    \node (UN) at (3,1) {$\MF(S,0)_s$};
    \node (N) at (3,0) {$\sB$}; 
    \draw[->] (UM) to node[left] { $\ev^{\HH}_{\sE/\sB}$ } (M);
    \draw[->] (UN) to node[right] { $\bigt{K(S,0)\to K(s,0)}_*$ } (N);
    \draw[->] (UM) to node[above]{$\int_{S'/S}^{\on{dg}}$} (UN);
    \draw[->] (M) to node[below]{$\on{m}$} (N);
    \end{tikzpicture}
    \end{equation}
is commutative.
\end{lem}

\begin{proof}
In view of $\sC \simeq (\sE^{\op}\otimes_A \sE) \otimes_{\sB^{\env}} \sB$,
the functor obtained by pre-composing with the arrow 
$$
  \ffF:=i_{S''*}\bigt{\ul \Hom_{s'}(-,k')\boxtimes_s -}:
  \sE^{\op}\otimes_A \sE
  \to
  \sC
$$
induces an equivalence
$$
\Hom_{\sB} (\sC, \sB)
\simeq
\Hom_{\sB^{\env}} (\sE^{\op} \otimes_A \sE, \sB).
$$
Hence, since the two circuits of the above diagram are both $\sB$-linear, it suffices to prove the commutativity of
\begin{equation}\label{diag: m circ evHH vs int simplified}
\nonumber
  \begin{tikzpicture}[scale=1.5]
    \node (UM) at (0,1) {$ \sE^{\op}  \otimes_A \sE$  };
    \node (M) at (0,0) {$\HH_*(\sB/A)\simeq \HH_*(\sC/A)$};
    \node (UN) at (3,1) {$\MF(S,0)_s$};
    \node (N) at (3,0) {$\sB$. }; 
    \draw[->] (UM) to node[left] { $\ev^{\HH}_{\sE/\sB} \circ \, \ffF$ } (M);
    \draw[->] (UN) to node[right] { $\bigt{K(S,0)\to K(s,0)}_*$ } (N);
    \draw[->] (UM) to node[above]{$\int_{S'/S}^{\on{dg}} \circ \, \ffF$} (UN);
    \draw[->] (M) to node[below]{$\on{m}$} (N);
    \end{tikzpicture}
    \end{equation}
The bottom circuit can be simplified as
$$
\on{m}\circ \ev^{\HH}_{\sE/\sB}\circ \, \ffF
\simeq \ev_{\sE/\sB}.
$$
To manipulate the top circuit, we will use the two fiber squares
\begin{equation*}
  \begin{tikzpicture}[scale=1.5]
    \node (Ld) at (0,0) {$K(S,0)$};
    \node (Lu) at (0,1) {$K(s,0)\simeq s\times_Ss$};
    \node (Cd) at (3,0) {$K(S',0)$};
    \node (Cu) at (3,1) {$K(s',0)\simeq
s'\times_{S'}s'$};
    \node (Rd) at (6,0) {$S' \times_S S'.$};
    \node (Ru) at (6,1) {$s'\times_s s'$};
    \draw[->] (Cd) to node[above]{$v_{K(S,0)}$} (Ld);
    \draw[->] (Cd) to node[above]{$d\delta_{S'/S}$} (Rd);
    \draw[->] (Cu) to node[above]{$v_{K(s,0)}$} (Lu);
    \draw[->] (Cu) to node[above]{$d\delta_{s'/s}$} (Ru);
    \draw[->] (Lu) to node[right]{$i_{K(S,0)}$} (Ld);
    \draw[->] (Cu) to node[right]{$i_{K(S',0)}$} (Cd);
    \draw[->] (Ru) to node[right]{$i_{S''}$} (Rd);
  \end{tikzpicture}
\end{equation*}
By construction (see \cite{beraldopippi22}) we have
$$
  \int_{S'/S}^{\on{dg}}
 =
  v_{K(S,0)*}\circ d\delta_{S'/S}^*.
$$
On the other hand, $\ev_{\sE/\sB}$ can be written geometrically as
$$
\ev_{\sE/\sB}
\simeq
v_{K(s,0)*}\circ \bigt{ \ul \Hom_{s'}(-,k')\boxtimes_{S'} -}.
$$
It follows that 
$$ 
\int_{S'/S}^{\on{dg}} \circ \, \ffF
\simeq i_{K(S,0)*}\circ \ev_{\sE/\sB}.
$$
We conclude that
$$
\bigt{K(S,0)\to K(s,0)}_* 
\circ \int_{S'/S}^{\on{dg}}
\circ \,
\ffF 
\simeq 
\bigt{K(S,0)\to K(s,0)}_* \circ i_{K(S,0)*}\circ \ev_{\sE/\sB}
 \simeq 
 \ev_{\sE/\sB}
$$
as desired.
\end{proof}

\begin{prop}\label{prop: nc nature ar}
The map
$$ 
\int_{S'/S}^{\on{dg}}:
\sC
\to
\MF(S,0)_s
$$
induces the opposite of the Artin character upon taking $\HK_0$.
\end{prop}
\begin{proof}
We have that
$$
  \HK_0(\sC)
  \simeq 
  \bbZ[\GLK]/ \langle \sum_{g\in \GLK} [\Gamma_g] \rangle.
$$
This follows from the fact that $\HK$ sends (abstract) blow-up squares to cartesian squares (\cite{cisinski13}): we apply this property to the cartesian square obtained from
$$
  s\to S'' \leftarrow S'\times \GLK.
$$
Using the definition of $d\delta_{S'/S}$, see \cite{beraldopippi22}, one computes that there are derived fiber squares
\begin{equation*}
    \begin{tikzpicture}[scale=1.5]
    \node (Lu) at (0,1) {$K(S',0)$};
    \node (Ru) at (3,1) {$K(S',0)$};
    \node (Ld) at (0,0) {$S'$};
    \node (Rd) at (3,0) {$S'\times_SS'$};
    \draw[->] (Lu) to node[above] {$$} (Ru);
    \draw[->] (Ld) to node[above] {$\delta_{S'/S}$} (Rd);
    \draw[->] (Lu) to node[above] {$$} (Ld);
    \draw[->] (Ru) to node[right] {$d\delta_{S'/S}$} (Rd);
      
    \end{tikzpicture}
    \hspace{0.5cm}
    \begin{tikzpicture}[scale=1.5]
    \node (Lu) at (0,1) {$\up{V}\bigt{g(\pi_{\uL})-\pi_{\uL}}$};
    \node (Ru) at (3,1) {$K(S',0)$};
    \node (Ld) at (0,0) {$S'$};
    \node (Rd) at (3,0) {$S'\times_SS'$};
    \draw[->] (Lu) to node[above] {$$} (Ru);
    \draw[->] (Ld) to node[above] {$\gamma_g$} (Rd);
    \draw[->] (Lu) to node[above] {$$} (Ld);
    \draw[->] (Ru) to node[right] {$d\delta_{S'/S}$} (Rd);
      
    \end{tikzpicture}
\end{equation*}
where in the square on the right $g\neq \id$.
Recall that $K(S',0)$ is the derived scheme associated to the simplicial Koszul algebra generated by one element in degree $1$.
In the derived fiber square on the left, the top arrow displayed is the one induced by sending the generator $\epsilon$ in degree $1$ to $E'(\pi_{\uL})\cdot \epsilon$, where
$$
E'(\pi_{\uL})=\prod_{g\in \GLK \setminus \{\id\}}\bigt{g(\pi_{\uL})-\pi_{\uL}}.
$$
Therefore, by applying derived base change one sees that
$$ 
\int_{S'/S}^{\on{dg}}:
\sC
\to
\MF(S,0)_s
$$
maps $\Delta_{S'/S}$ to the matrix factorization
\begin{equation*}
    \begin{tikzpicture}[scale=1.5]
      \node (a) at (0,0) {$\ccO_{S'}$};
      \node (b) at (1.5,0) {$\ccO_{S'}$};
      \path[->,font=\scriptsize,>=angle 90]
           ([yshift= 1.5pt]b.west) edge node[above] {$E'(\pi_{\uL})$} ([yshift= 1.5pt]a.east);
      \path[->,font=\scriptsize,>=angle 90]
           ([yshift= -1.5pt]a.east) edge node[below] {$0$} ([yshift= -1.5pt]b.west);
\end{tikzpicture}
\end{equation*}
and $\gamma_{g*}\ccO_{S'}$ to the matrix factorization
\begin{equation*}
    \begin{tikzpicture}[scale=1.5]
      \node (a) at (0,0) {$\ccO_{S'}$};
      \node (b) at (1.5,0) {$\ccO_{S'}.$};
      \path[->,font=\scriptsize,>=angle 90]
           ([yshift= 1.5pt]b.west) edge node[above] {$0$} ([yshift= 1.5pt]a.east);
      \path[->,font=\scriptsize,>=angle 90]
           ([yshift= -1.5pt]a.east) edge node[below] {$g(\pi_{\uL})-\pi_{\uL}$} ([yshift= -1.5pt]b.west);
\end{tikzpicture}
\end{equation*}
The claim follows easily from the well-known facts that
$$
  \deg:
  \uG_0(S)_s
  \xto{\simeq}
  \bbZ
$$
and that
$$
  \HK_0(\MF(S,0)_s)
  \xto{\simeq}
  \uG_0(S)_s
$$
$$
  [M\leftrightarrows N]
  \mapsto
  [\coker(M\to N)]-[\coker(N\to M)].
$$
\end{proof}
\begin{rmk}
Notice that in the proposition above we didn't use that $S$ is of pure characteristic. 
In fact, the proof works is mixed characteristic as well.
\end{rmk}

\sssec{}

Applying $\rl_S$ to \eqref{diag: m circ evHH vs int} and invoking Proposition \ref{prop: nc nature ar}, we obtain the commutative diagram
\begin{equation} \label{diag:rl(m circ ev^HH)=-ar}
  \begin{tikzpicture}[scale=1.5]
    \node (Lu) at (0,1.5) {$(i_S)_*\QellI(\beta)(\GLK)$  };
    \node (Ld) at (0,0) {$\rl_S\bigt{\HH_*(\sB/A)}$};
    \node (Ru) at (5,1.5) {$\rl_S\bigt{\MF(S,0)_s}$};
    \node (Rd) at (5,0) {$\rl_S(\sB)\simeq (i_S)_*\QellI(\beta)$.};
    \draw[->] (Lu) to node[right] { $\wideprime{\rl_S(\ev^{\HH}_{\sC/\sC})}$ } (Ld);
    \draw[->] (Ru) to node[right] { $\simeq$ } (Rd);
    \draw[->] (Lu) to node[above]{$\int_{S'/S}=\rl_S(\int_{S'/S}^{\on{dg}}) \sim -\ar$} (Ru);
    \draw[->] (Ld) to node[above]{$\rl_S(\on{m})$} (Rd);
  \end{tikzpicture}
\end{equation}

\sssec{}

We can now conclude the proof of Theorem \ref{mainthm: pure char}.
Thanks to Theorem \ref{thm: categorical Bloch}, we have an equality
\begin{align*}
  \Bl^{\cat}(X/S) & = - \dimtot^{\cat}(X/S;\Qell) \\
                  & := \idcat([\varepsilon_{\sT}\circ \gamma_{\sT}])-\arcat([\varepsilon_{\sU}\circ \gamma_{\sU}]) 
\end{align*}
of classes in $\uH^0_\et\Bigt{S,\rl_S \bigt{\HH_*(\sB/A)}}$.
Let us now apply  
$$
\rl_S(\on{m}):\uH^0_\et\Bigt{S,\rl_S \bigt{\HH_*(\sB/A)}} \to \uH^0_\et\bigt{S,\rl_S(\sB)}\simeq \Qell
$$
to the above equality. It follows from \cite{beraldopippi22} that 
$$
\rl_S(\on{m}) \bigt{\Bl^{\cat}(X/S)}
=
\Bl(X/S).
$$
We compute:
\begin{align*}
    \Bl(X/S) & = \rl_S(\on{m}) \bigt{\Bl(X/S)^{\cat}} \\
             & = \rl_S(\on{m}) \Bigt{\idcat([\wt{\eps}_{\sT}\circ \wt{\gamma}_{\sT}])-\arcat([\wt{\varepsilon}^{\HH}_{\sU}\circ \wt{\gamma}_{\sU}]) } && \text{Thm. \ref{thm: categorical Bloch}} \\
             & = \Tr_{\rl_S(\sB)}\bigt{\id: \rl_S(\sT)} - \ar \Bigt{\Tr_{\rl_S(\sC)}\bigt{\id:\rl_S(\sU)}} && \text{\ref{sssec: retraction ev^(HH)_(B/B)}, \eqref{diag:rl(m circ ev^HH)=-ar} } \\
             & && \text{Thm. \ref{cor: duality r(T) over r(B)} }\\
             & && \text{Thm. \ref{thm:duality for r(U)}}\\
             & = -\Bigt{\Tr_{\QellI(\beta)}(\id : \nu^{\IK}) + \frac{1}{|\GLK|}\sum_{g\in \GLK} \Tr_{\QellI(\beta)}(g:\nu^{\IL}) \cdot \ar(g)} && \text{\cite[Thm. 4.39]{brtv18}  }\\
             & && \text{Prop. \ref{prop: rl(C) acts on rl(U)} }\\
             & && \text{Cor. \ref{cor: tr over Ql^I= tr over i_*Ql^I} }\\
             & && \text{Cor. \ref{cor: trace reduced group alg}  }\\
             & && \ar(g)\in \bbZ\\
             & = - \Tr_{\QellI}(\id : \Phi^{\IK})  - \Ar(\Phi) && \text{\eqref{eqn: def Ar}}\\
             & = -\dimtot (\Phi). && \text{\eqref{eqn: dimtot = Ar + Tr_(Qell^I)(id:V^I)}}
\end{align*}

\sec{The hypersurface case}\label{sec: connection to the classical conjectures}

In this section, we adapt the proof of Theorem \ref{thm: categorical Bloch} to prove Theorem \ref{mainthm: hypersurf}.

\ssec{Proof of Theorem \ref{mainthm: hypersurf}}\label{ssec: proof of thm B}

\sssec{} 

In this section, we do not impose any restriction on the characteristic of $S$. Instead, we assume that $X$ is the zero locus of a section $\sigma \in \uH^0(P,L)$ of a line bundle $L$ on a smooth $S$-scheme $P$.
Let $\iota:Z\subseteq X$ denote the singular locus of $X \to S$.
We require that the composition $Z \hto X \to S$ be proper, while $X \to S$ does not need to be.

\sssec{} 

Under these circumstances, the main construction of \cite{beraldopippi22} yields a dg-functor
\begin{equation} \label{eqn:pullback fat diag}
   \Sing(X\times_S X)
  \to
  \MF(X, L,0)_Z
\end{equation}
obtained by pulling back along a quasi-smooth closed embedding
$$
  d\delta_{X/S}:
  K(X, \restr L X,0)
  \to 
  X\times_S X,
$$
which is a derived enhanced of the diagonal $\delta_{X/S} : X \to X \times_S X$.

\sssec{}

As explained in \cite{beraldopippi22}, upon taking realizations, \eqref{eqn:pullback fat diag} induces a morphism
$$
  \int_{X/S}:
  \rl_S\bigt{\Sing(X\times_SX)}
  \to
  \rl_S\bigt{\MF(X,\restr L X,0)_Z}
  \xto{\deg}
  \rl_S\bigt{\MF(S,0)_s}
$$
in $\Mod_{\Ql{,S}(\beta)}\bigt{\Shv(S)}$. 
The morphism $\deg$ identifies with the degree map in homology under the identifications
$$
  \rl_S\bigt{\MF(X,\restr L X,0)_Z}
  \simeq
  (i_S)_*(p_Z)_*\bigt{\oml{,Z}(\beta)\oplus \oml{,Z}(\beta)[1]},
$$
$$
  \rl_S\bigt{\MF(S,0)_s}
  \simeq
  (i_S)_*\bigt{\Qell(\beta)\oplus \Qell(\beta)[1]}.
$$

Using the $\ell$-adic Chern character, we further obtain a homomorphism of abelian groups
$$
  \uG_0\bigt{X\times_S X}
  \to
  \HK_0\bigt{\Sing(X\times_S X)}
  \xto{\chern}
  \uH^0_\et \Bigt{S, \rl_{S} \bigt{\Sing(X\times_S X)}}
  \xto{\int_{X/S}}
  \uH^0_\et \Bigt{S, \rl_{S} \bigt{\MF(S,0)_s}}.
$$

\sssec{} 

It follows immediately from \cite[Proposition 3.30]{brtv18} that $\uH^0_\et \Bigt{S, \rl_S \bigt{\MF(S,0)_s}} \simeq \Qell$.
Under this identification, we showed in \cite{beraldopippi22} that
$$
  \int_{X/S}[\Delta_X]
  =
  \Bl(X/S).
$$

\sssec{} 

We can apply the above construction with $s \to S$ replacing $X\to S$. In this particular case, \eqref{eqn:pullback fat diag} becomes the dg-functor 
$$
  \sB
  \to
  \MF(S,0)_s
$$
that corresponds, under the $A$-linear equivalence $\Sing(G)\simeq \MF(s,0)$, to the push-forward
$$
  \bigt{K(s,0)\to K(S,0)}_*:
  \MF(s,0)
  \to
  \MF(S,0)_s.
$$
We denote by $\int_{s/S}: \rl_S(\sB) \to \rl(\MF(S,0)_s)$ the image of $\sB \to \MF(S,0)_s$ under $\rl_S$.

\sssec{}
Let
$$
  \int'_{S'/S}:
  \HH_*\bigt{\rl_S(\sC)/\rl_S(A)}
  \to
  \rl_S\bigt{\MF(S,0)_s}
$$
denote the morphism induced by $\int_{S'/S}$.
The following result is crucial: it allows - in the hypersurface case - to avoid reference to the cohomology group $\uH^0_{\et}\Bigt{S, \rl_S \bigt{\HH_*(\sB/A)}}$.
\begin{prop} \label{prop:two commutative squares}
{The two squares below are commutative:}
\begin{equation}  \label{diag:big-comm-diagram-rl (GR)}
  \begin{tikzpicture}[scale=1.5]
    \node (Lu) at (-4,1.2) {$\rl_S(\sT) \otimes_{\rl_S(\sB)} \rl_S(\sT)$};
    \node (Ld) at (-4,0) {$\rl_S(\sB)$};
    \node (Cu) at (0,1.2) {$\rl_S\bigt{\Sing(X\times_S X)}$};
    \node (Cd) at (0,0) {$ \rl_S\bigt{\MF(S,0)_s}$};
    \node (Ru) at (4,1.2) {$\rl_S(\sU) \otimes_{\rl_S(\sC)} \rl_S(\sU)$};
    \node (Rd) at (4,0) {$ \HH_*\bigt{\rl_S(\sC)/(i_S)_*\QellI(\beta)}$.};
    \draw[->] (Lu) to node[right] { $\wt{\eps}_{\sT}$ } (Ld);
    \draw[->] (Cu) to node[right] { $\int_{X/S}$ } (Cd);
    \draw[->] (Ru) to node[right] { $\wt{\eps}_{\sU}^{\HH}$ } (Rd);
    \draw[->] (Ru) to node[above]{$ \mu_{\sU/\sC}$} (Cu);
    \draw[->] (Rd) to node[above]{$\int'_{S'/S}$} (Cd);
    \draw[->] (Lu) to node[above]{$ \mu_{\sT/\sB}$} (Cu);
    \draw[->] (Ld) to node[above]{$\int_{s/S}$} (Cd);
  \end{tikzpicture}
\end{equation}
\end{prop}
\begin{proof}
The proof is postponed to Section \ref{ssec:tech-prop}.

\end{proof}
\sssec{}

Let us assume this proposition for the moment and conclude the proof of Theorem \ref{mainthm: hypersurf}. The argument is similar to the proof of the categorical Bloch conductor formula.
Using the commutative squares above, we obtain (similarly to \eqref{eqn: magic}) that 
\begin{equation}\label{eqn: Bl=int_(s/S)Tr(id:r(T))+int_(S'/S) Tr(id:r(U))}
    \Bl(X/S) = \int_{X/S}[\Delta_X] = \int_{s/S}[\wt{\eps}_\sT \circ \wt{\gamma}_\sT] + \wideprime{\int_{S'/S}}[\wt{\eps}_\sU^{\HH} \circ \wt{\gamma}_\sU].
\end{equation}
Observe that
\begin{equation}\label{eqn: int_(s/S)=id}
  \id \simeq \int_{s/S}:
  \: \rl_S(\sB)
  \simeq
  (i_S)_*\QellI(\beta)
  \to
  \rl_S\bigt{\MF(S,0)_s}
  \simeq
  (i_S)_*\QellI(\beta)
\end{equation}
and that (see Proposition \ref{prop: nc nature ar})
\begin{equation}\label{eqn: int_(S'/S)=-ar}
  -\ar \simeq \int_{S'/S}:
  \: \rl_S(\sC)
  \simeq 
  (i_S)_*\QellI(\beta)(\GLK)
  \to
  \rl_S\bigt{\MF(S,0)_s}
  \simeq
  (i_S)_*\QellI(\beta).
\end{equation}
We conclude with the same computation as in the pure characteristic case:
\begin{align*}
    \Bl(X/S) & = \int_{s/S}[\wt{\eps}_\sT \circ \wt{\gamma}_\sT] + \wideprime{\int_{S'/S}}[\wt{\eps}_\sU^{\HH} \circ \wt{\gamma}_\sU] && \eqref{eqn: Bl=int_(s/S)Tr(id:r(T))+int_(S'/S) Tr(id:r(U))}\\
             & = [\wt{\eps}_\sT \circ \wt{\gamma}_\sT]-\ar \bigt{ [\wt{\eps}_\sU^{\HH} \circ \wt{\gamma}_\sU]} && \text{\eqref{eqn: int_(s/S)=id}, \eqref{eqn: int_(S'/S)=-ar} }\\
             & = \Tr_{\rl_S(\sB)}\bigt{\id:\rl_S(\sT)}-\ar \Bigt{\Tr_{\rl_S(\sC)}\bigt{\id:\rl_S(\sU)}} && \text{Thm. \ref{thm: duality r(T)} }\\
             & && \text{Thm. \ref{thm:duality for r(U)}}\\ 
             & = \Tr_{i_{S*}\QellI(\beta)}\bigt{\id:i_{S*}\nu^{\IK}[-1]}-\ar \Bigt{\Tr_{i_{S*}\QellI(\beta)(\GLK)}\bigt{\id:i_{S*}(\nuIquot)}} && \text{\cite[Thm. 4.39]{brtv18} }\\
             & && \text{Thm. \ref{prop: rl(C) acts on rl(U)} }\\
             & = \Tr_{\QellI(\beta)}\bigt{\id:\nu^{\IK}[-1]} - \frac{1}{|\GLK|}\sum_{g\in \GLK}\ar(g)\cdot \Tr_{\QellI(\beta)}\bigt{g:\nuIquot} && \text{Cor. \ref{cor: tr over Ql^I= tr over i_*Ql^I} }\\
             & && \text{Cor. \ref{cor: trace reduced group alg}}\\
             & && \ar(g)\in \bbZ\\
             & = - \Tr_{\QellI}(\id : \Phi^{\IK})  - \Ar(\Phi) && \text{\eqref{eqn: def Ar}}\\
             & = -\dimtot (\Phi). && \text{\eqref{eqn: dimtot = Ar + Tr_(Qell^I)(id:V^I)}}
\end{align*}

\ssec{A technical proposition} \label{ssec:tech-prop}

It remains to prove Proposition \ref{prop:two commutative squares}.

\sssec{}
Let $i_{\sT}: \rl_S(\sT)\otimes_{\rl_S(A)}\rl_S(\sT)\to \rl_S\bigt{\Sing(X\times_SX)}$ denote the morphism induced by the lax-monoidal structure on $\rl_S$ composed with the morphism
$$
  \rl_S\Bigt{(X_s\times_sX_s \to X\times_SX)_*\bigt{-\boxtimes_s \bbD_{X_s}(-)}}:
  \rl_S(\sT \otimes_A \sT^{\op})
  \to
  \rl_S\bigt{\Sing(X\times_SX)}.
$$
We denote $i_{\sU}: \rl_S(\sU)\otimes_{\rl_S(A)}\rl_S(\sU)\to \rl_S\bigt{\Sing(X\times_SX)}$ the morpshims obtained in a similar way on the other side of the Morita equivalence.
We will now prove the following.
\begin{prop}\label{prop:pre-big-comm-diagram-rl (GR)}
The two squares below are commutative:
\begin{equation}  \label{diag:pre-big-comm-diagram-rl (GR)}
  \begin{tikzpicture}[scale=1.5]
    \node (Lu) at (-4,1.2) {$\rl_S(\sT) \otimes_{\rl_S(A)} \rl_S(\sT)$};
    \node (Ld) at (-4,0) {$\rl_S(\sB)$};
    \node (Cu) at (0,1.2) {$\rl_S\bigt{\Sing(X\times_S X)}$};
    \node (Cd) at (0,0) {$ \rl_S\bigt{\MF(S,0)_s}$};
    \node (Ru) at (4,1.2) {$\rl_S(\sU) \otimes_{\rl_S(A)} \rl_s(\sU)$};
    \node (Rd) at (4,0) {$\rl_S(\sC)$.};
    \draw[->] (Lu) to node[right] { $\eps_{\sT}$ } (Ld);
    \draw[->] (Cu) to node[right] { $\int_{X/S}$ } (Cd);
    \draw[->] (Ru) to node[right] { $\eps_{\sU}$ } (Rd);
    \draw[->] (Ru) to node[above]{$ i_\sU$} (Cu);
    \draw[->] (Rd) to node[above]{$\int_{S'/S}$} (Cd);
    \draw[->] (Lu) to node[above]{$ i_\sT$} (Cu);
    \draw[->] (Ld) to node[above]{$\int_{s/S}$} (Cd);
  \end{tikzpicture}
\end{equation}
\end{prop}

\begin{proof}
The proof is the combination of Lemma \ref{lem: left square tech-prop} and Lemma \ref{lem: right square tech-prop} below, which deal with each square separately.
\end{proof}

\begin{lem}\label{lem: left square tech-prop}
The left square of diagram (\ref{diag:pre-big-comm-diagram-rl (GR)}) commutes.
\end{lem}
\begin{proof}
The proof is the same as in \cite[Lemma 4.3.3]{beraldopippi22}; we reproduce it here for the reader's convenience. We proceed in steps.
\sssec*{Step 1}
Both $\rl_S(\sT)$ and $\rl_S(\sB)$ are dualizable over $\rl_S(A)$.
It is easy to compute that
$$
  \ul \Hom_{\Ql{,S}(\beta)}\bigt{\rl_S(\sB),\Ql{,S}(\beta)}
  \simeq
  \rl_S(\sB)(1)[1],
$$
$$
  \ul \Hom_{\Ql{,S}(\beta)}\bigt{\rl_S(\sT),\Ql{,S}(\beta)}
  \simeq
  \rl_S(\sT)(1)[1].
$$
By duality, and also invoking the equivalence $\int_{s/S}:\rl_S(\sB)\to \rl_S\bigt{\MF(S,0)_s}$, one can see that $\eps_{\sT}:\rl_S(\sT)\otimes_{\rl_S(A)}\rl_S(\sT)\to \rl_S(\sB)$ and $\int_{X/S}\circ i_{\sT}:\rl_S(\sT)\otimes_{\rl_S(A)}\rl_S(\sT)\to \rl_S\bigt{\MF(S,0)_s}$ correspond to two morphisms of the form
\begin{equation*}
    \rl_S(\sB)\otimes_{\rl_S(\sA)}\rl_S(\sT)
    \to
    \rl_S(\sT).
\end{equation*}
Invoking \cite[Proposition 4.27, Theorem 4.39]{brtv18}, these identify with two morphisms of the form
\begin{equation*}
    (i_S)_*\bigt{\QellI(\beta)\otimes_{\Qell(\beta)}(p_Z)_*\scrV_{X/S}^\IK[-1]}
    \simeq
    (i_S)_*(p_Z)_*\bigt{\Ql{,Z}^{\on{I}}(\beta)\otimes_{\Ql{,Z}(\beta)}\scrV_{X/S}^{\IK}[-1]}
    \to
    (i_S)_*(p_Z)_*\scrV_{X/S}^{\IK}[-1].
\end{equation*}
The one corresponding to $\eps_{\sT}$ is the image along $(i_S)_*(p_Z)_*$ of the morphism 
\begin{equation}\label{eqn: morphism a}
    a:
    \Ql{,Z}^{\on{I}}(\beta)\otimes_{\Ql{,Z}(\beta)}\scrV_{X/S}^{\IK}[-1]
    \to
    \scrV_{X/S}^{\IK}[-1]
\end{equation}
in $\Shv(Z)$ induced by the action of $\Ql{,Z}^{\on{I}}(\beta)$ on $\scrV_{X/S}^{\IK}[-1]$.

\sssec*{Step 2}
We will now describe the origin of the morphism
$$
 (i_S)_*(p_Z)_*\bigt{\Ql{,Z}^{\on{I}}(\beta)\otimes_{\Ql{,Z}(\beta)}\scrV_{X/S}^{\IK}[-1]}
    \to
    (i_S)_*(p_Z)_*\scrV_{X/S}^{\IK}[-1]
$$
which corresponds to $\int_{X/S}\circ i_{\sT}$.

Since $X\simeq P\times_{\sigma,L,0}P$, the ($\Perf(P)$-linear) dg-category $\Sing(X_s)$ carries a canonical action of the monoidal dg-category $\MF(P,L,0)$ (see \cite{pippi22a}).
Via the lax-monoidal structure on the functor $\rl_P$ (see Remark \ref{rmk: rl over general bases}), we obtain an action of $\rl_P\bigt{\MF(P,L,0)}$ on $\rl_P\bigt{\Sing(X_s)}$.

Since 
$$
  \rl_P\bigt{\Sing(X_s)}
  \simeq
  k_* \scrV_{X/S}^{\IK}[-1],
$$
where $k:Z\to P$, this action factors through
$$
  \rl_P\bigt{\MF(P,L,0)}
  \to
  k_*k^* \rl_P\bigt{\MF(P,L,0)}.
$$
Notice that (see \cite[Proposition 3.3.1]{beraldopippi22}) 
$$
  k_*k^* \rl_P\bigt{\MF(P,L,0)}
  \simeq
  k_*\Ql{,Z}^{\on{I}}(\beta).
$$
Then the morphism above is the image along $(i_S)_*(p_Z)_*$ of this action of $\Ql{,Z}^{\on{I}}(\beta)$ on $\scrV_{X/S}^\IK[-1]$, which we denote
\begin{equation}\label{eqn: morphism a'}
    a':
    \Ql{,Z}^{\on{I}}(\beta)\otimes_{\Ql{,Z}(\beta)}\scrV_{X/S}^{\IK}[-1]
    \to
    \scrV_{X/S}^{\IK}[-1].
\end{equation}

\sssec*{Step 3} 

Thanks to the previous two steps, to conclude the proof of the lemma it suffices to show that $a$ and $a'$ are homotopic.
This is a local statement on $Z$.
In particular, we may assume that $L$ is the trivial line bundle\footnote{Notice that by working locally on $Z$ we have to drop the properness assumption. However, the morphism \eqref{eqn: morphism a} and \eqref{eqn: morphism a'} are always defined, regardless of the properness assumption on $Z$: this hypothesis is only needed to obtain, by duality, the morphisms $\eps_{\sT}$ and $\int_{X/S}\circ i_{\sT}$.}.
Hence we can assume that $X$ is the zero locus of a function $f: P\to \bbA^1_S$ on a smooth $S$-scheme $P$: 
$$
  X
  \simeq
  S\times_{0,\bbA^1_S,f}P.
$$

\sssec*{Step 4}

Consider now the following cube (whose faces are all cartesian):
\begin{equation*}
    \begin{tikzpicture}[scale=1]
      \node(LLuu) at (-2,2) {$X_s$};
      \node(RRuu) at (2,2) {$P_s$};
      \node(LLdd) at (-2,-2) {$X$};
      \node(RRdd) at (2,-2) {$P$.};
      \node(Lu) at (-1,1) {$s$};
      \node(Ru) at (1,1) {$\bbA_s^1$};
      \node(Ld) at (-1,-1) {$S$};
      \node(Rd) at (1,-1) {$\bbA_S^1$};
    
      \draw[->] (LLuu) to node[above] {${}$} (RRuu);
      \draw[->] (Lu) to node[above] {${}$} (Ru);
      
      \draw[->] (Ld) to node[above] {${}$} (Rd);
      \draw[->] (LLdd) to node[above] {${}$} (RRdd);
      
      \draw[->] (LLuu) to node[above] {${}$} (LLdd);
      \draw[->] (Lu) to node[above] {${}$} (Ld);
      
      \draw[->] (Ru) to node[above] {${}$} (Rd);
      \draw[->] (RRuu) to node[above] {${}$} (RRdd);
      
      \draw[->] (LLuu) to node[above] {${}$} (Lu);
      \draw[->] (LLdd) to node[above] {${}$} (Ld);
      
      \draw[->] (RRuu) to node[above] {${}$} (Ru);
      \draw[->] (RRdd) to node[above] {${}$}(Rd);
    \end{tikzpicture}
\end{equation*}
%
%
%
In particular, $X_s$ is canonically acted upon by the following three derived groupoids:
$$
G=s\times_Ss,
\hspace{.4cm}
G^{\on{comm}}=s\times_{\bbA^1_s}s,
\hspace{.4cm}
G_S^{\on{comm}}=S\times_{\bbA^1_S}S.
$$
Since $X$, $P_s$ and $P$ are regular, we obtain actions of  $\Sing(G)$, $\Sing(G^{\on{comm}})$ and $\Sing(G_S^{\on{comm}})$ on $\Sing(X_s)$.
By construction, the morphism $a$ is induced by the action of $\Sing(G)$ on $\Sing(X_s)$ and the morphism $a'$ by the action of $\Sing(G^{\on{comm}})$ on $\Sing(X_s)$.
We conclude by observing that, at the level of $\ell$-adic sheaves, these actions coincide: they are both induced by the action of $\Sing(G_S^{\on{comm}})$ on $\Sing(X_s)$.
Indeed, under the equivalences
$$
  \rl_S\bigt{\Sing(G_S^{\on{comm}})}
  \simeq 
    \Ql{,S}^{\on{I}}(\beta),
$$
$$
  \rl_S\bigt{\Sing(G)}
  \simeq
  \rl_S\bigt{\Sing(G^{\on{comm}})}
  \simeq
  (i_S)_*i_S^*\QellI(\beta),
$$
the morphisms $\rl_S\bigt{\Sing(G_S^{\on{comm}})}\to  \rl_S\bigt{\Sing(G})$ and $\rl_S\bigt{\Sing(G_S^{\on{comm}})}\to  \rl_S\bigt{\Sing(G^{\on{comm}})}$ are both given by the canonical map
$$
  \Ql{,S}^{\on{I}}(\beta)
  \to
  (i_S)_*i_S^*\QellI(\beta)
$$
coming from the unit of the adjunction $\bigt{i_S^*,(i_S)_*}$.
\end{proof}

\begin{lem}\label{lem: right square tech-prop}
The right square of diagram (\ref{diag:pre-big-comm-diagram-rl (GR)}) commutes.
\end{lem}

\begin{proof}
We proceed in steps.

\sssec*{Step 1}
Notice that $\rl_S(\sU)$ is a dualizable $\rl_S(A)$-module, with dual
$$
  \ul \Hom_{\Ql{,S}(\beta)}\bigt{\rl_S(\sU),\Ql{,S}(\beta)}
  \simeq
  \rl_S(\sU)(1)[1].
$$
Since $\rl_S\bigt{\MF(S,0)_s}\simeq (i_S)_*\QellI(\beta)$ is also dualizable over $\rl_S(A)$, with dual $\rl_S\bigt{\MF(S,0)_s}(1)[1]$, we see that $\int_{S'/S}\circ \eps_{\sU}$ and $\int_{X/S}\circ i_{\sU}$ correspond to two morphisms
$$
  a,a':
  \rl_S\bigt{\MF(S,0)_s}\otimes_{\rl_S(A)}\rl_S(\sU)
  \to
  \rl_S(\sU).
$$
Using the equivalences $$\rl_S\bigt{\MF(S,0)_s}\simeq (i_S)_*\QellI(\beta)
$$
$$\rl_S(\sU)\simeq (i_S)_*\bigt{\nuIquot}
$$
and the projection formula, $a$ and $a'$ can be identified with morphisms
$$
  b,b':
  (i_S)_*(p_Z)_*\bigt{\Ql{,Z}^{\uI}(\beta)\otimes_{\Ql{,Z}(\beta)}\scrV_{X/S}^{\IL}/\scrV_{X/S}^{\IK}}
  \to
  (i_S)_*(p_Z)_*\bigt{\scrV_{X/S}^{\IL}/\scrV_{X/S}^{\IK}}.
$$

\sssec*{Step 2}
We now describe the geometric origin of the morphism $a: \rl_S\bigt{\MF(S,0)_s}\otimes_{\rl_S(A)}\rl_S(\sU) \to \rl_S(\sU)$, corresponding to $\int_{S'/S}\circ \eps_{\sU}$.
Notice that $X'$ sits in the following fiber product of schemes:
\begin{equation*}
    \begin{tikzpicture}[scale=1.5]
      \node (Lu) at (0,2) {$X'$};
      \node (Ru) at (1,2) {$\bbA_X^1$};
      \node (Lc) at (0,1) {$S'$};
      \node (Rc) at (1,1) {$\bbA_S^1$};
      \node (Ld) at (0,0) {$S$};
      \node (Rd) at (1,0) {$\bbA_S^1.$};
      
      \draw[->] (Lu) to node[above] {${}$} (Ru);
      \draw[->] (Lc) to node[above] {$\pi_{\uL}$} (Rc);
      \draw[->] (Ld) to node[above] {${}$} (Rd);
      
      \draw[->] (Lu) to node[above] {${}$} (Lc);
      \draw[->] (Lc) to node[above] {${}$} (Ld);
      
      \draw[->] (Ru) to node[above] {${}$} (Rc);
      \draw[->] (Rc) to node[right] {$E$} (Rd);
    \end{tikzpicture}
\end{equation*}
This formally implies that the derived groupoid $G_S^{\on{comm}}=S\times_{\bbA_S^1}S$ acts on $X'$.
Since $\bbA_X^1$ is regular, this induces an action of $\Sing(S\times_{\bbA_S^1}S)\simeq \MF(S,0)$ on $\Sing(X')$.
Taking $\rl_S$, this yields an action of $\rl_S\bigt{\MF(S,0)}$ on $\rl_S\bigt{\Sing(X')}$, which gives the morphism $a$.

It follows that $b$ is the image along $(i_S)_*(p_Z)_*$ of the morphism
$$
  \Ql{,Z}^{\uI}(\beta)\otimes_{\Ql{,Z}(\beta)}\scrV_{X/S}^{\IL}/\scrV_{X/S}^{\IK}
  \to
  \scrV_{X/S}^{\IL}/\scrV_{X/S}^{\IK}
$$
which corresponds to the canonical action of $\Ql{,Z}^{\uI}(\beta)$ on $\scrV_{X/S}^{\IL}/\scrV_{X/S}^{\IK}$.
\sssec*{Step 3} 
We now describe the geometric origin of the morphism $a': \rl_S\bigt{\MF(S,0)_s}\otimes_{\rl_S(A)}\rl_S(\sU) \to \rl_S(\sU)$ corresponding to $\int_{X/S}\circ i_{\sU}$.

By our standing assumption on $X$, we immediately obtain that $X'\simeq P'\times_{\sigma',L',0}P'$.
This formally implies that the derived groupoid $P'\times_{0,L',0}P'$ acts on $X'$.
Since $P'$ is regular, this induces an action of the (symmetric) monoidal dg-category $\MF(P',L',0)=\Sing(P'\times_{0,L',0}P')$ on the ($\Perf(P')$-linear) dg-category $\Sing(X')$, see \cite{pippi22a}.
This endows $\rl_{P'}\bigt{\Sing(X')}$ with an action of $\rl_{P'}\bigt{\MF(P',L',0)}$.
We now observe that 
$$
  \rl_{P'}\bigt{\Sing(X')}
  \simeq
  l_*\bigt{\scrV_{X/S}^{\IL}/\scrV_{X/S}^{\IK}},
$$
where $l:Z\to P'$ denotes the closed embedding of $Z$ into $P'$.
This formally implies that the action of $\rl_{P'}\bigt{\MF(P',L',0)}$ on $\rl_{P'}\bigt{\Sing(X')}$ factors through 
\begin{align*}
    \rl_{P'}\bigt{\MF(P',L',0)} & \to l_*l^*\rl_{P'}\bigt{\MF(P',L',0)} \\
                                & \simeq l_*\Ql{,Z}^{\uI}(\beta). && \text{\cite[Proposition 3.3.1]{beraldopippi22}} 
\end{align*}
Then $b'$ corresponds to the image along $(P'\to S)_*$ of the morphism
$$
  l_*\Ql{,Z}^{\uI}(\beta)\otimes_{\Ql{,P'}(\beta)}l_*\bigt{\scrV_{X/S}^{\IL}/\scrV_{X/S}^{\IK}}
  \to
  l_*\bigt{\scrV_{X/S}^{\IL}/\scrV_{X/S}^{\IK}}
$$
induced by this action.

\sssec*{Step 4}
Since both $b$ and $b'$ are images along $(i_S)_*(p_Z)_*$ of morphisms
$$
  c,c':
  \Ql{,Z}^{\uI}(\beta)\otimes_{\Ql{,Z}(\beta)}\bigt{\scrV_{X/S}^{\IL}/\scrV_{X/S}^{\IK}}
  \to
  \scrV_{X/S}^{\IL}/\scrV_{X/S}^{\IK},
$$
in order to show that $b\simeq b'$ it suffices to show that $c$ and $c'$ are homotopic.
This is a local statement on $Z$ and therefore we may assume that $L\simeq \ccO_P$ and $\sigma=f:P\to \bbA_S^1$ is a function on $P$.

\sssec*{Step 5}
We then see that $X'$ sits in the following fiber square:
\begin{equation*}
    \begin{tikzpicture}[scale=1.5]
      \node (Lu) at (0,2) {$X'$};
      \node (Ru) at (1,2) {$P'$};
      \node (Lc) at (0,1) {$S'$};
      \node (Rc) at (1,1) {$\bbA_{S'}^1$};
      \node (Ld) at (0,0) {$S$};
      \node (Rd) at (1,0) {$\bbA_S^1.$};
      
      \draw[->] (Lu) to node[above] {${}$} (Ru);
      \draw[->] (Lc) to node[above] {${}$} (Rc);
      \draw[->] (Ld) to node[above] {${}$} (Rd);
      
      \draw[->] (Lu) to node[above] {${}$} (Lc);
      \draw[->] (Lc) to node[above] {${}$} (Ld);
      
      \draw[->] (Ru) to node[right] {$f\times_SS'$} (Rc);
      \draw[->] (Rc) to node[right] {${}$} (Rd);
    \end{tikzpicture}
\end{equation*}
In particular, the derived groupoid $S\times_{\bbA_S^1}S$ acts on $X'$.
Since $P'$ is regular, this induces an action of $\MF(S,0)$ on $\Sing(X')$.
The morphism $c'$ is induced by this action.

\sssec*{Step 6}
We finish by noticing that both $c$ and $c'$ identify with the canonical action of $\Ql{,Z}^{\uI}(\beta)$ on $\scrV_{X/S}^{\IL}/\scrV_{X/S}^{\IK}$.
To see this, consider the diagram
\begin{equation*}
    \begin{tikzpicture}[scale=1]
      \node(LLuu) at (-2,2) {$X'$};
      \node(RRuu) at (2,2) {$P'$};
      \node(LLdd) at (-2,-2) {$\bbA_X^1$};
      \node(RRdd) at (2,-2) {$P\times_S\bbA_S^1$};
      
      \node(Lu) at (-1,1) {$S$};
      \node(Ru) at (1,1) {$\bbA_S^1$};
      \node(Ld) at (-1,-1) {$\bbA_S^1$};
      \node(Rd) at (1,-1) {$\bbA_S^2$};
    
      \draw[->] (LLuu) to node[above] {${}$} (RRuu);
      \draw[->] (Lu) to node[above] {${}$} (Ru);
      
      \draw[->] (Ld) to node[above] {${}$} (Rd);
      \draw[->] (LLdd) to node[above] {${}$} (RRdd);
      
      \draw[->] (LLuu) to node[above] {${}$} (LLdd);
      \draw[->] (Lu) to node[above] {${}$} (Ld);
      
      \draw[->] (Ru) to node[above] {${}$} (Rd);
      \draw[->] (RRuu) to node[above] {${}$} (RRdd);
      
      \draw[->] (LLuu) to node[above] {${}$} (Lu);
      \draw[->] (LLdd) to node[above] {${}$} (Ld);
      
      \draw[->] (RRuu) to node[above] {${}$} (Ru);
      \draw[->] (RRdd) to node[above] {${}$}(Rd);
    \end{tikzpicture}
\end{equation*}
where all squares are cartesian. It follows that the derived groupoid $S\times_{\bbA^2_S}S$ acts on $X'$, which in turn induces an action of $\rl_S\bigt{\Sing(S\times_{\bbA^2_S}S)}\boxtimes_S(Z\to X')_*\Ql{,Z}$ on $\rl_{X'}\bigt{\Sing(X')}$.
The two actions of $\rl_S\bigt{\Sing(S\times_{\bbA^1_S}S)}\boxtimes_S(Z\to X')_*\Ql{,Z}$ mentioned above are induced by this action.

We finish by observing\footnote{This can be proved as in \cite[Proposition 4.27]{brtv18} or, alternatively, using the equivalence $\Sing(S\times_{\bbA_S^2}S)\simeq \MF\bigt{\bbP_S^1,\ccO(1),0}$ and the results of \cite{pippi22a}.} that 
$$
  \rl_S\bigt{\Sing(S\times_{\bbA^2_S}S)}
  \simeq 
  \Ql{,S}^{\uI}(\beta)
$$
so that the morphisms $c$ and $c'$ are actually homotopic. 
\end{proof}

\sssec{} \label{sssec: proof of prop on two commutative squares}
We are now ready to show that the two squares in (\ref{diag:big-comm-diagram-rl (GR)}) commute.  This was the statement of Proposition \ref{prop:two commutative squares}.

It clear that this follows from Proposition \ref{prop:pre-big-comm-diagram-rl (GR)}: both $\rl_S(\sT)\otimes_{\rl_S(A)}\rl_S(\sT)\to \rl_S\bigt{\Sing(X\times_SX)}$ and $\rl_S(\sT)\otimes_{\rl_S(A)}\rl_S(\sT)\to \rl_S(\sB)$ factor through the canonical morphism
$$
  \rl_S(\sT)\otimes_{\rl_S(A)}\rl_S(\sT)
  \to
  \rl_S(\sT)\otimes_{\rl_S(\sB)}\rl_S(\sT).
$$
Similarly, $\rl_S(\sU)\otimes_{\rl_S(A)}\rl_S(\sU)\to \rl_S\bigt{\Sing(X\times_SX)}$ factors through 
$$
  \rl_S(\sU)\otimes_{\rl_S(A)}\rl_S(\sU)
  \to
  \rl_S(\sU)\otimes_{\rl_S(\sC)}\rl_S(\sU)
$$
and $\rl_S(\sU)\otimes_{\rl_S(A)}\rl_S(\sU)\to \rl_S(\sC)\to \rl_S\bigt{\MF(S,0)_s}$ factors through 
$$
  \rl_S(\sU)\otimes_{\rl_S(\sC)}\rl_S(\sU)
  \to
  \HH_*\bigt{\rl_S(\sC)/\rl_S(A)}.
$$
This shows that the diagram
\begin{equation*}
  \begin{tikzpicture}[scale=1.5]
    \node (Lu) at (-4,1.2) {$\rl_S(\sT) \otimes_{\rl_S(\sB)} \rl_S(\sT)$};
    \node (Ld) at (-4,0) {$\rl_S(\sB)$};
    \node (Cu) at (0,1.2) {$\rl_S\bigt{\Sing(X\times_S X)}$};
    \node (Cd) at (0,0) {$ \rl_S\bigt{\MF(S,0)_s}$};
    \node (Ru) at (4,1.2) {$\rl_S(\sU) \otimes_{\rl_S(\sC)} \rl_S(\sU)$};
    \node (Rd) at (4,0) {$ \HH_*\bigt{\rl_S(\sC)/\rl_S(A)}$.};
    \draw[->] (Lu) to node[right] {$\wt{\eps}_{\sT}$} (Ld);
    \draw[->] (Cu) to node[right] {$\int_{X/S}$} (Cd);
    \draw[->] (Ru) to node[right] {$\eps_{\sU}^{\HH}$} (Rd);
    \draw[->] (Ru) to node[above]{$\mu_{\sU/\sC}$} (Cu);
    \draw[->] (Rd) to node[above]{$\int'_{S'/S}$} (Cd);
    \draw[->] (Lu) to node[above]{$\mu_{\sT/\sB}$} (Cu);
    \draw[->] (Ld) to node[above]{$\int_{s/S}$} (Cd);
  \end{tikzpicture}
\end{equation*}
commutes. 
To conclude that the squares in \eqref{diag:big-comm-diagram-rl (GR)} commute as well, it suffices to observe that $\eps_{\sU}^{\HH}$ factors through $\wt{\eps}_{\sU}^{\HH}$, see \eqref{eqn: epsHH factors through wt(eps)HH}.

\appendix 
\sec{}\label{appendix: traces over group algebras}

In this appendix, we review the well-known formula for the trace of a dualizable module over a (reduced) group algebra.

\ssec{Traces over group algebras}

\sssec{}
Let $\ccC$ be a symmetric monoidal stable $\oo$-category, which we assume to be $\bbQ$-linear.
Let $\uF$ be a commutative algebra object in $\ccC$ and $H$ a finite group. Denote by
$$
  \uF[H]
  :=
  \bigoplus_{h\in H} \uF \cdot e_{h}
$$
the group algebra of $H$ with coefficients in $\uF$.
Given a left $\uF[H]$-module $M$ and a right $\uF[H]$-module $N$, we have 
$$
  (N\otimes_{\uF}M)\otimes_{\uF[H]^{\env}}\uF[H]
  \simeq 
  N\otimes_{\uF[H]}M
  \simeq
  (N\otimes_{\uF}M)^{H},
$$
where $H$ acts diagonally on $N\otimes_{\uF}M$.
Here $\uF[H]^{\env}=\uF[H]^{\rev}\otimes_{\uF}\uF[H]$ denotes the \virg{enveloping} algebra.

\sssec{} 

Recall that a left $\uF[H]$-module $M$ is dualizable if there are a right $\uF[H]$-module $N$, an $\uF[H]$-bilinear map
$$
  \ev:
  M\otimes_{\uF}N
  \to
  \uF[H]
$$
and an $\uF$-linear map
$$
  \coev:
  \uF
  \to
  N\otimes_{\uF[H]}M
$$
such that the compositions
$$
  M
  \simeq
  M\otimes_{\uF}\uF
  \xto{\id \otimes \coev}
  M\otimes_{\uF}N\otimes_{\uF[H]}M
  \xto{\ev\otimes \id}
  \uF[H]\otimes_{\uF[H]}M
  \simeq 
  M
$$
$$
  N
  \simeq
  \uF \otimes_{\uF}N
  \xto{\coev \otimes \id}
  N\otimes_{\uF[H]}M\otimes_{\uF}N
  \xto{\id \otimes \ev}
  N\otimes_{\uF[H]}\uF[H]
  \simeq 
  N
$$
are the identity maps.

\sssec{}

Given a dualizable $\uF[H]$-module $M$ and an $H$-linear map $T: M \to M$, we define the trace 
$$
\Tr_{\uF[H]}(T: M) \in \HH_0(\uF[H]/F)
$$
in the standard way, that is, as the $\uF$-linear map
$$
\uF 
\xto{\coev} 
N \otimes_{\uF[H]} M
\xto{\id \otimes T}
N \otimes_{\uF[H]} M
\xto{\ev^{\HH}}
\uF[H]
\usotimes_{\uF[H]^{\env}}
\uF[H]
=
\HH_*\bigt{\uF[H]/\uF},
$$
where $\ev^{\HH}$ is the arrow induced naturally by $\ev$.

\begin{lem}
Let $M$ be a left $\uF[H]$-module whose underlying $\uF$-module is dualizable. Then $M$ is dualizable as a left $\uF[H]$-module and
\begin{equation} \label{eqn:appendix-trace-formula}
  \Tr_{\uF[H]}(\id:M)
  =
  \frac{1}{|H|} \sum_{h\in H}\Tr_{\uF}(h^{-1}:M) \langle e_h \rangle
\end{equation}
in $\HH_0(\uF[H]/\uF)$.
\end{lem}

\begin{proof}
We first show that $M^\vee := \Hom_{\uF}(M,\uF)$, equipped with its natural right $\uF[H]$-module structure, is dual to the left $\uF[H]$-module $M$.
Choose a duality datum
$$
  \gamma:
  \uF
  \to M^{\vee}\otimes_{\uF} M
$$
$$
  \eps:
  M\otimes_{\uF} M^{\vee}
  \to
  \uF
$$
for $M$ as an $\uF$-module and denote by
$$ 
r= 
M^\vee \otimes_{\uF} M
 \to
M^\vee \otimes_{\uF[H]} M
$$
the canonical projection.
We claim that the morphisms
$$
  \coev
  :=
  r\circ \gamma
  :
  \uF 
  \to
   M^\vee \otimes_{\uF[H]} M
$$
and
$$
  \ev:=
  \frac{1}{|H|}\sum_{g\in H}\eps \circ (g^{-1}\otimes \id) \cdot e_g
  :
  M \otimes_{\uF} M^{\vee}
  \to
  \uF[H]
$$
form a duality datum. To prove this, we need to show that
$$
  M
  \simeq
  M \otimes_{\uF}\uF
  \xto{\id \otimes \coev}
  M \otimes_{\uF} M^{\vee}\otimes_{\uF[H]} M
  \xto{\ev\otimes \id}
  \uF[H]\otimes_{\uF[H]} M
  \simeq
  M
$$
is the identity (plus a parallel fact for the other triangular identity, which we leave to the reader). 
A straightforward diagram chase shows that this composition can be rewritten as
$$
  M
  \simeq
  M\otimes_{\uF}\uF
  \xto{\id \otimes \gamma}
  M\otimes_{\uF} M^{\vee}\otimes_{\uF} M
  \xto{\ev\otimes \id}
  \uF[H]\otimes_{\uF} M
  \xto{\act}
  M.
$$
Thanks to the obvious
$$
  (\eps \otimes \id)\circ (\act_h \otimes \id \otimes \id)\circ (\id \otimes \gamma)
 \simeq 
 (\eps \otimes \id)\circ (\id \otimes \gamma)
 \circ \act_h
 \simeq \act_h,
$$
we obtain
$$
  (\ev \otimes \id) \circ (\id \otimes \coev)
  = \frac{1}{|H|} \sum_{g\in H} g^{-1} \cdot e_g.
$$
and the duality claim follows.
We now prove \eqref{eqn:appendix-trace-formula}.
Observe that
$$
  \on{can} \circ \ev = \ev^{\HH}\circ r,
$$
where
$$
  \on{can}: 
  \uF[H] 
  \to
  \HH_0\bigt{\uF[H]/\uF}
$$
is the canonical map. It follows that
$$
  \Tr_{\uF[H]}(\id: M)
  =
  \on{can} \Bigt{\frac{1}{|H|} \sum_{h\in H} \eps \circ (h\otimes \id) \circ \gamma \cdot e_h}
  =
  \frac{1}{|H|}\sum_{h\in H} \Tr_{\uF}(h^{-1} : M) \langle e_h \rangle,
$$
as desired.
\end{proof}

\ssec{Traces over reduced group algebras}

\sssec{}

The reduced group algebra of $H$ with coefficients in $\uF$ is defined as
$$
  \uF(H)
  :=
  \uF[H]/\langle \sum_{h\in H}e_h \rangle.
$$
The $\infty$-category of left $\uF(H)$-modules is equivalent to the full subcategory of $\Mod_{\uF[H]}(\ccC)$ spanned by the objects that have zero $H$-invariants.

\begin{cor}\label{cor: trace reduced group alg}
Let $M$ be a left $\uF[H]$-module whose underlying $\uF$-module is dualizable. Then $M/M^H: = \coFib(M^H \to M)$ is a dualizable left $\uF(H)$-module and
    $$
    \Tr_{\uF(H)}(\id:M/M^H)
    =
    \frac{1}{|H|} \sum_{h\in H}\Tr_{\uF}(h^{-1}:M/M^H) \langle e_h \rangle
    $$
    in $\HH_0(\uF(H)/F)$.
\end{cor}

\begin{proof}
The dualizability of $M/M^H$ follows immediately from that of $M$, since the symmetric monoidal functor
$$
  \uF(H)\otimes_{\uF[H]}-:
  \Mod_{\uF[H]}(\ccC)
  \to
 \Mod_{\uF(H)}(\ccC)
$$
verifies
$$
  \uF(H)\otimes_{\uF[H]} M
  \simeq
  M/M^H.
$$
The trace formula follows from \eqref{eqn:appendix-trace-formula}. Indeed, we have the following chain of equalities in $\HH_0\bigt{\uF(H)/\uF}$:
\begin{align*}
  \Tr_{\uF(H)}(\id:M/M^H) 
 & = \Tr_{\uF[H]}(\id:M) \\
 & = \frac{1}{|H|} \sum_{h\in H}\Tr_{\uF}( h^{-1} :M) \langle e_h \rangle \\
 & = \frac{1}{|H|} \sum_{h \in H}\Tr_{\uF}(h^{-1} :M/M^H) \langle e_h \rangle  + \frac{1}{|H|} \sum_{h\in H}\Tr_{\uF}(h^{-1} :M^H) \langle e_h  \rangle \\
& = \frac{1}{|H|} \sum_{ h\in H}\Tr_{\uF}( h^{-1} :M/M^H) \langle e_h \rangle,
\end{align*}
where the last equality follows from the observation that 
$$
  \sum_{h\in H}\langle e_h \rangle 
  =
  0 
$$
in  $\HH_0\bigt{\uF(H)/\uF}$. 
\end{proof}
\printbibliography
\end{document}